\documentclass[11pt, reqno]{amsart}

\usepackage{amsmath}
\usepackage{amsthm}
\usepackage{amssymb}
\usepackage[left=3.5cm,right=3.5cm,top=3cm,bottom=3cm, marginparwidth=2cm]{geometry}
\usepackage{amscd}
\usepackage{stmaryrd}
\usepackage[normalem]{ulem}
\usepackage{MnSymbol}
\usepackage{bbm}
\usepackage[arrow, curve, matrix]{xy}
\usepackage{tikz-cd}
\usepackage{accents}

\usepackage[utf8]{inputenc}
\usepackage[cyr]{aeguill}
\usepackage{xspace}

\usepackage{tabu, multirow}

\usepackage{listings}
\usepackage{verbatim}

\def\XXint#1#2#3{{\setbox0=\hbox{$#1{#2#3}{\int}$ }
\vcenter{\hbox{$#2#3$ }}\kern-.58\wd0}}

\usepackage{mathbbol}
\usepackage{amssymb}             

\DeclareSymbolFontAlphabet{\mathbb}{AMSb}%
\DeclareSymbolFontAlphabet{\mathbbl}{bbold}

\usepackage{enumerate}

\usepackage{array}

\usepackage{hyperref}
\usepackage{varioref}
\usepackage[hyperpageref]{backref}

\usepackage{xcolor}
\hypersetup{
    colorlinks,
    linkcolor={red!60!black},
    citecolor={blue!60!black},
    urlcolor={blue!50!black}
}

\theoremstyle{thm} \newtheorem{thm}{Theorem}
\theoremstyle{thm} \newtheorem*{thm*}{Theorem}
\theoremstyle{thm} \newtheorem*{claim*}{Claim}
\newtheorem{prop}{Proposition} [section]
\newtheorem{nota}{Notation} [section]
\newtheorem{lem}[prop]{Lemma}
\newtheorem{claim}[prop]{Claim}
\newtheorem{corol}[prop]{Corollary}

\theoremstyle{definition} 
\theoremstyle{remark} 
\theoremstyle{remark} \newtheorem{rem}[prop]{Remark}
\theoremstyle{definition} \newtheorem{defi}[prop]{Definition}

\newcommand{\quotient}[2]{{\left.\raisebox{-.2em}{$#1$}\middle\backslash\raisebox{.2em}{$#2$}\right.}}
\newcommand{\quotientd}[2]{{\left.\raisebox{.2em}{$#1$}\middle\slash\raisebox{-.2em}{$#2$}\right.}}

\makeatletter
\newcommand{\xleftrightarrow}[2][]{\ext@arrow 3359\leftrightarrowfill@{#1}{#2}}
\newcommand{\xdashrightarrow}[2][]{\ext@arrow 0359\rightarrowfill@@{#1}{#2}}
\newcommand{\xdashleftarrow}[2][]{\ext@arrow 3095\leftarrowfill@@{#1}{#2}}
\newcommand{\xdashleftrightarrow}[2][]{\ext@arrow 3359\leftrightarrowfill@@{#1}{#2}}
\def\rightarrowfill@@{\arrowfill@@\relax\relbar\rightarrow}
\def\leftarrowfill@@{\arrowfill@@\leftarrow\relbar\relax}
\def\leftrightarrowfill@@{\arrowfill@@\leftarrow\relbar\rightarrow}
\def\arrowfill@@#1#2#3#4{%
  $\m@th\thickmuskip0mu\medmuskip\thickmuskip\thinmuskip\thickmuskip
   \relax#4#1
   \xleaders\hbox{$#4#2$}\hfill
   #3$%
}
\makeatother

\makeatletter
\def\blfootnote{\xdef\@thefnmark{}\@footnotetext}
\makeatother

\makeatletter
\newcommand{\thickhline}{%
    \noalign {\ifnum 0=`}\fi \hrule height 1pt
    \futurelet \reserved@a \@xhline
}
\newcolumntype{"}{@{\hskip\tabcolsep\vrule width 1pt\hskip\tabcolsep}}
\makeatother

\newcommand{\scc}{{specially chain connected }}

\renewcommand{\U}{\operatorname{U}}
\renewcommand{\Im}{\operatorname{Im}}
\DeclareMathOperator{\PU}{PU}

\DeclareMathOperator{\PGL}{PGL}

\title{Uniformization of varieties with log-canonical singularities}
\date{}

\author[B. Cadorel]{Beno\^{i}t Cadorel} 
\email{benoit.cadorel@univ-lorraine.fr}
\address{Institut \'Elie Cartan de Lorraine, Universit\'e de Lorraine, F-54000 Nancy,
	France.}
\urladdr{http://www.normalesup.org/~bcadorel/}

\begin{document}

\begin{abstract}
	We study the problem of uniformizing quasi-projective varieties with log-canonical compactifications. More precisely, given a complex projective variety \(X\) with log-canonical singularities, we give criteria for \(X\) to be isomorphic to a Baily-Borel-Mok compactification of a ball quotient, asking on the one hand the equality case in a suitable Miyaoka-Yau (MY) inequality, and on the other hand some adequate assumptions on the singularities. We also give as a result of independent interest that log-resolutions of log-canonical singularities have their fibers connected by chains of special varieties in the sense of Campana; this is used in the proof to control the behaviour of the period map near the exceptional divisors of such resolutions.

	We also show that it is necessary to assume that the singularities are at least log-canonical: some examples of Deligne-Mostow-Deraux can be manipulated to provide examples of singular varieties satisfying the equality case in MY, while not being isomorphic to such Baily-Borel-Mok compactifications.
\end{abstract}

\maketitle

\section{Introduction}

Let \(X = \quotientd{\mathbb{B}^{n}}{\Gamma}\) be a quotient of the complex unit ball by a {\em lattice} i.e. a discrete subgroup \(\Gamma \subset \mathrm{Aut}(\mathbb{B}^{n})\) with finite Bergman covolume. In the case where \(X\) is non-compact, the work of Baily-Borel \cite{BB66} and Mok \cite{Mok12} implies that \(X\) is a quasi-projective variety, admitting a normal compactification with boundary made of a finite number of points: \(X^{\ast} = X \sqcup \{p_{1}, \dotsc, p_{m}\}\). This compactification has log-canonical singularities and ample canonical bundle \(K_{X^{\ast}}\). Possibly after replacing \(\Gamma\) by a finite index subgroup, we may assume that \(X \subset X^{\ast}\) is smooth, and that all boundary singularities \(p_{i}\) are locally analytically isomorphic to a cone over an abelian variety. In this situation, there exists a log-resolution \(\widehat{X} \to X\) where every boundary divisor has discrepancy \(-1\).  
\medskip

The goal of these notes is to investigate under which conditions a given projective variety can be obtained as such a Baily-Borel-Mok compactification. In dimension \(2\), the situation is quite well-understood thanks to the work of Kobayashi \cite{Kob85} -- these results can be used to give a criterion for a singular surface to be a ball-quotient compactification, in terms of equality in some Miyaoka-Yau inequality (see Section~\ref{sec:surfacecase}). It is natural to ask how to extend this type of statements to higher dimensional singular varieties -- assuming for example that the singularities are {\em punctual}. As far as the Miyaoka-Yau inequality is concerned, it is actually possible to obtain such a statement when \(X^{\ast}\) has log-canonical singularities and is {\em smooth in codimension \(2\)}, as follows:

\begin{thm}[= Theorem~\ref{thm:existencemap}] \label{thm:THM1}
	Let \(X^{\ast}\) be a variety of dimension \(n \geq 3\), with log-canonical singularities and smooth in codimension \(2\). We assume that \(K_{X^{\ast}}\) is an ample \(\mathbb{Q}\)-Cartier divisor. Then \(X^{\ast}\) satisfies a Miyaoka-Yau inequality: 
	\begin{equation} \label{eq:MYineqintro}
		\big(2(n+1)\, c_{2}(T_{X}) - n\, c_{1}^{2}(T_{X})\big) 
		\cdot K_{X^{\ast}}^{n-2} \geq 0.	
	\end{equation}
	where \(X \subset X^{\ast}\) is the smooth locus (see Section~\ref{sec:unifmap} for more details). In equality case, there exists a map \(\varphi : \widetilde{X} \to \mathbb{B}^{n}\) that is étale everywhere.
\end{thm}

Proving the kind of statement above has now become quite classical and can be done for example using the ideas of Greb-Kebekus-Peternell-Taji \cite{GKPT19a, GKPT19a, GKPT20}. The inequality can be obtained by first showing the stability of an adequate Higgs bundle in restriction to a complete intersection surface sitting inside \(X\), the latter stability boiling down to a Metha-Ramanathan type statement.
\medskip

The natural question is now to know whether this map \(\varphi\) is always a biholomorphism: that would ensure that \(X^{\ast}\) is indeed the Baily-Borel-Mok compactification of a ball quotient. The hypothesis on the singularities seems important: it is indeed possible to contract divisors in some examples of Deligne-Mostow \cite{DM86, DM93} and Deraux \cite{Der05} to get counter-examples.

\begin{thm}(cf. Theorem~\ref{thm:contractionproj} + Section~\ref{sec:DMDex}) \label{thm:examples} There exist examples of varieties \(X^{\ast}\) with punctual, {\em non log-canonical singularities} and ample cotangent bundle, that satisfy equality in \eqref{eq:MYineqintro}, and that admit a non-isomorphic étale period map \(\varphi : \widetilde{X} \to \mathbb{B}^{n}\). The smooth part \(X\) cannot be given the structure of a ball quotient.
\end{thm}

In these examples, there exist a resolution of singularities \(\overline{X} \to X\) with exceptional divisor \(D = D_{1} \sqcup \dotsc \sqcup D_{m}\), such that each \(D_{i}\) is a smooth ball quotient: the diagram
\begin{equation} \label{eq:diagintro}
	\begin{tikzcd}
		\widetilde{X} \arrow[r, "\varphi"] \arrow[d] & \mathbb{B}^{n} \\
		X
	\end{tikzcd}
\end{equation}
has a limit near each \(D_{i}\), inducing a totally geodesic embedding \(\widetilde{D}_{i} \hookrightarrow \mathbb{B}^{n}\) whose image is a ball embedded as a hyperplane section in \(\mathbb{B}^{n}\). The fact that the \(D_{i}\) are ball quotients implies in particular that these singularities of the \(X^{\ast}\) are not log-canonical.  
\medskip

We believe that a suitable hypothesis on the singularities of \(X^{\ast}\), jointly with the equality case in \eqref{eq:MYineqintro} should be enough for \(X^{\ast}\) to be a Baily-Borel-Mok compactification -- we don't actually know if it is enough to assume that the singularities of \(X^{\ast}\) are log-canonical. We managed nonetheless to obtain partial results with stronger assumptions. This is maybe the main result of this work:

\begin{thm}(=Theorem~\ref{thm:completeness})
	Let \(X^{\ast}\) be a projective variety with {\em punctual, log-canonical singularities} and ample canonical divisor. Assume that:
	\begin{enumerate}[(a)]
		\item there exists a log-resolution of \(X^{\ast}\) where every exceptional divisor has discrepancy equal to \(-1\);
		\item there is equality in \eqref{eq:MYineqintro}. 
	\end{enumerate}
	Then the smooth part \(X \subset X^{\ast}\) is a ball quotient, and \(X^{\ast}\) is its Baily-Borel-Mok compactification. 
\end{thm}

The proof of the theorem consist in studying the asymptotic behavior of the diagram \eqref{eq:diagintro} near the singularities. Again, let \(\overline{X} \to X^{\ast}\) be a resolution of singularities, with exceptional divisor \(D\). After some computations, essentially based on Schmid's nilpotent orbit theorem \cite{Schmid73} (or their recent version for \(\mathbb{C}\)-VHS due to Sabbah-Schnell~\cite{SS22}, see also Deng \cite{Deng23}), one can induce limiting maps  \(\varphi_{T} : \widetilde{T} \longrightarrow \overline{\mathbb{B}^{n}}\), where \(T \subset D\) is any smooth stratum of \(D\). To conclude, we want to show that any such map lands, not in \(\mathbb{B}^{n}\), but in its boundary \(\partial \mathbb{B}^{n}\). A topological argument permits to reduce the problem to showing that if \(\varphi_{T}\) takes its image in \(\mathbb{B}^{n}\), then it is actually constant. But now, remark that \(\varphi_{T}\) is the period map of a polarized \(\mathbb{C}\)-VHS. To show that it is constant, one can try to apply the results of \cite{CDY23}, that have as a corollary that a polarized \(\mathbb{C}\)-VHS on any quasi-projective manifold that is {\em special in the sense of Campana}, has constant period map (see Theorem~\ref{thm:isotrivial}): this statement should be compared to a similar isotriviality theorem for families of canonically polarized manifolds, due to Taji \cite{Taji16}. 
\medskip

The last piece of the proof then comes from the following theorem, which we believe is interesting in its own right.

\begin{thm} (= Theorem~\ref{thm:scc} + Corollary~\ref{corol:strataspecial}) \label{thm:specialintro}
	Let \(X\) be a complex analytic space with log-canonical singularities. Let \(\pi : \overline{X} \to X\) be a log-resolution of singularities. Then the fibers of \(\pi\) are connected by chains of {\em special varieties}.
	\medskip

	In addition, if for some \(x \in X\), the sets \(\pi^{-1}(x)\) and \(\pi^{-1}(x) \cup \pi^{-1}(\mathrm{Sing}(X))\) are both divisors with normal crossings, and if all discrepancies are equal to \(-1\), then the smooth locally closed strata of \(\pi^{-1}(x)\) are themselves {\em special} quasi-projective varieties.
\end{thm}

The previous result can be seen as generalization to a log-canonical setting of the work of Hacon-McKernan \cite{HM07}, who proved that log-resolutions of klt singularities have rationally chain connected fibers. The proof follows the same circle of ideas, and is ultimately based on an application of Campana's result around the subadditivity of the orbifold Kodaira dimension -- this is used to show that the core fibration of the strata above must be trivial.

\subsection{Comparison with earlier work} The problem of uniformization in the open or singular case has been a long standing one, and many authors have contributed to that question.

\begin{enumerate}
	\item Deng \cite{Deng22} gave a criterion for uniformization in terms of the polystability of some natural logarithmic Higgs bundle on the open part of a log-pair \((\overline{X}, D)\). To compare his hypotheses with ours, let us resume the assumptions of our Theorem~\ref{thm:THM1}, and let \(\overline{X} \overset{\pi}{\longrightarrow} X\) be a log-resolution with boundary divisor \(D\). 
		
		Then, one finds that \(K_{\overline{X}} + D =  \pi^{\ast} K_{X}\) is a nef and big divisor on \(\overline{X}\). In this situation, one can use the work of Guenancia \cite{Gue15} joint to a classical computation (see \cite[Corollary~7.2]{GKPT19a} and the proof of Theorem~\ref{thm:existencemap}), to show that the standard Higgs bundle \((\Omega_{\overline{X}} (\log D) \oplus \mathcal{O}_{\overline{X}}, \theta)\) considered by Deng is stable with respect to \(K_{\overline{X}} + D\), for an ample \(\mathbb{Q}\)-line bundle \(L\) sufficiently close to \(K_{\overline{X}} + D\).

		We are thus allowed to apply \cite[Theorem 4.7]{Deng22}, which yields a Miyaoka-Yau inequality with respect to \(L\) --  we may then reobtain \eqref{eq:MYineqintro} by letting \(L\) tend to \(K_{\overline{X}} + D\). If we want to have the uniformization statement, Deng requires an equality in the Miyaoka-Yau inequality associated to \(L\):
		\[
			(2(n+1) c_{2}(T_{\overline{X}}(-\log D))
			- 
			n c_{1}(T_{\overline{X}}(-\log D))) \cdot L^{n-2} = 0.	
		\]
		which is maybe less intrinsic than the equality case in \eqref{eq:MYineqintro}. Also, the description of \cite[Theorem~4.7]{Deng22} does not allow to prove that the situation where \(D\) is not smooth cannot occur -- in the setting of Theorem~\ref{thm:completeness}, we do not need to require anything on the regularity of \(D\), and we are actually able to eventually exclude the case where \(D\) is not smooth. Of course, in our setting the line bundle \(K_{X} + D = \pi^{\ast} K_{X^{\ast}}\) is semi-ample and not merely nef and big, so our hypotheses are not strictly weaker than the ones of \cite{Deng22}.

	\item In \cite{Yau93}, Yau claims that for any log pair \((\overline{X}, D)\) with \(K_{\overline{X}} + D\) nef, big and ample modulo \(D\), with equality in the logarithmic Miyaoka-Yau inequality, the open part \(X = \overline{X} - D\) is a ball quotient. This would give almost immediately Theorem~\ref{thm:completeness} (even without the assumption that the singularities are punctual). Unfortunately, we were not quite able to follow his arguments : Yau claims that we can apply the Schwarz lemma of \cite{Yau78} to deduce the completeness of the natural K\"{a}hler-Einstein metric, but this lemma seems to already require the completeness of the said metric.

	\item As we mentioned before, we don't know if having purely log-canonical singularities is enough to obtain a uniformization theorem; the methods of the present work are in any case not quite sufficient to deal with the case where discrepancies \(a_{E} > - 1\) are allowed to occur. The work of \cite{GKPT19a, GKPT19b} -- and the recent generalization to the case of orbifold pairs by Claudon-Graf-Guenancia \cite{CGG23} -- focuses on the klt case, where all discrepancies are \(> -1\). One could certainly mix the methods of their work and the current article to deal with some intermediate situations (e.g. uniformizing a variety \(X^{\ast}\) with klt locus {\em disjoint} from a finite union of singular points admitting a resolution with \(-1\) as only discrepancy). Since to us the most interesting case should be the one where we have some discrepancy \(> -1\) {\em above one of the only singular points}, we have preferred to stick to the pure situation to avoid any artificiality in our discussion.
\end{enumerate}

\subsection{Organization of the article}

\begin{enumerate}
	\item Section~\ref{sec:notation}: We have gathered here several preliminary lemmas and notation for the rest of the article.
	\item Section~\ref{sec:unifmap}: After a brief discussion of the two-dimensional case, we give a proof of Theorem~\ref{thm:THM1}. The method follows closely some ideas present in the work of Greb-Kebekus-Peternell-Taji and Mochizuki. 
	\item Section~\ref{sec:example}: We first prove a criterion for contracting totally geodesic divisors in manifolds admitting a correspondence with the ball, so as to obtain a singular variety satisfying the outcome of Theorem~\ref{thm:examples} (see Theorem~\ref{thm:contractionproj}). We then apply this criterion to several examples of Deligne-Mostow and Deraux in Section~\ref{sec:DMDex}.
	\item Section~\ref{sec:special}: We recall a few facts on Campana's notion of special varieties, and then prove Theorem~\ref{thm:specialintro}.
	\item Section~\ref{sec:VHSspecial}: We give some details on the proof of the isotriviality of p-\(\mathbb{C}\)VHS on special quasi-projective manifolds, that can be seen as a corollary of the main results of \cite{CDY23}. Most of the material of this section might be well-known to experts, but we have chosen to include it for lack of a self-contained reference.
	\item Section~\ref{sec:asymptotic}: Given a manifold with an open subset admitting a period map, we perform the local and global computations that allow to define a limiting period map on the strata of the complementary divisor.
	\item Section~\ref{sec:uniformization}: Finally, under the hypotheses of Theorem~\ref{thm:completeness}, we prove that the period map is indeed uniformizing.
	\item Section~\ref{sec:topology}: we have put here two criteria for a \'{e}tale map to be a universal covering. The first one plays an important role at the end of the proof of Theorem~\ref{thm:completeness}. 
\end{enumerate}

\subsection{Acknowledgments} \medskip
The author would like to thank Yohan Brunebarbe, Ya Deng, Mihai P\u{a}un and Behrouz Taji for their helpful comments and answers to his questions. Special thanks are due to Henri Guenancia, for the many enlightening discussions during all the preparation of this work.

\medskip
The author acknowledges support from the French ANR project KARMAPOLIS (ANR-21-CE40-0010). He would also like to thank the Isaac Newton Institute for Mathematical Sciences for the support and hospitality during the programme {\em New equivariant methods in algebraic and differential geometry} when work on this paper was undertaken. This work was supported by: EPSRC Grant Number EP/R014604/1. This work was also partially supported by a grant from the Simons Foundation.

\section{Notation and preliminary remarks} \label{sec:notation}

\subsection{Linear algebraic groups}

If \(K\) is a field, we will use the notation \(G_{K}, H_{K}, L_{K}...\) to denote linear algebraic groups defined over \(K\) (which for us, will simply be an algebraic subvariety of some \(GL_{n, K} \cong \mathbb{A}_{K}^{n^{2}} \setminus \{\det = 0\}\), invariant under the group structure). If \(K \subset K'\) is any field extension, we let \(G_{K'} := G_{K} \otimes_{K} K'\) the group deduced by extension of scalars. The \(K'\)-rational points of \(G_{K'}\) will be denoted by \(G(K')\).
\medskip

As usual, we will use the notation \(\mathrm{U}(m), \mathrm{U}(p, q), \mathrm{PU}(p, q)\) to denote the real algebraic groups of isometries (or projective isometries) with respect to a hermitian form of the type
\[
	\sum_{j=1}^{m} |z_{j}|^{2}
\quad
\text{or}
\quad
	\sum_{j=1}^{p} |z_{j}|^{2} - \sum_{j=1}^{q} |z_{p+j}|^{2}.
\]

\subsection{Unitary similarities} We let \(\mathrm{Sim}(\mathbb{C}^{m}) = \mathbb{C}^{m} \rtimes \U(m)\) be the group of unitary similarities of \(\mathbb{C}^{m}\): it is the direct product of the group of unitary transformations by the group of translations. 

\begin{lem} \label{lem:centunitsim}
	Let \(A \in \mathrm{Sim}(\mathbb{C}^{m})\). Then there is a unique orthogonal decomposition \(\mathbb{C}^{m} = V \overset{\perp}{\oplus} V^{\perp}\) such that \(A\) acts as a non-zero translation on \(V^{\perp}\), and on \(V\) as a unitary rotation \(R \in \mathrm{Sim}(V)\) around some center \(\omega\in V\). The center is unique if and only if the vectorial part \(R_{0} \in \mathrm{U}(m)\) of \(R\) does not have \(1\) as eigenvalue.
\end{lem}
	 We will say that \(A\) is a {\em pure rotation} if \(V^{\perp} = 0\); this is equivalent to saying that \(A\) has a fixed point.

\begin{proof}
	Write \(A\) under the form \(A \cdot z = A_{0} z + B\), with \(A_{0} \in \U(m)\) and \(B \in \mathbb{C}^{m}\). First take \(W\) as the orthogonal complement of \(W^{\perp} = \mathrm{ker}(A_{0} - I)\), and let \(B = B_{W} + B_{W^{\perp}}\) be the corresponding orthogonal decomposition of \(B\). Let \(\omega :=  (I - A_{0})|_{W}^{-1}(B_{W}) \in W\). Then, if \(z \in \mathbb{C}^{m}\) and writing the orthogonal decomposition \(z = z_{W} + z_{W^{\perp}}\), one has
\begin{align*}
	A \cdot z & = \left(A_{0} z_{W} + B_{W}\right) + \left( A_{0} z_{W^{\perp}} + B_{W^{\perp}} \right) \\
	& =  \left( A_{0} (z_{W} - \omega) + \omega \right) + \left(z_{W^{\perp}} + B_{W^{\perp}}\right).
\end{align*}
	Thus, \(A\) acts on \(W\) as a pure rotation \(R\) of center \(\omega\) and vectorial part \(A_{0}|_{W}\), and on \(W^{\perp}\) as a translation of vector \(B_{W^{\perp}}\). If this is vector is zero , take \(V = \mathbb{C}^{m}\). If not, take \(V = W\). The uniqueness statement comes from the fact that the difference between two centers yields a \(1\)-eigenvector for \(A_{0}|_{W}\).
\end{proof}

The following basic lemma will be quite useful to write a set of pure rotations in a normal form.

\begin{lem} \label{lem:purerotcomm}
	Two pure rotations \(A, B \in \mathrm{Sim}(\mathbb{C}^{n})\) commute if and only if they have a common fixed point and their images in \(\U(m)\) can be simultaneously diagonalized.  
\end{lem}
\begin{proof}
	The statement is well-known if \(A, B \in \U(m)\), and thus it suffices to prove that if \(A, B\) commute, then they have a common fixed point. Assume that the origin is centered at a fixed point of  \(A\) so that \(A \in \U(m)\), and \(Bz = B_{0} z + b\), where \(B_{0} \in \U(m), b\in\mathbb{C}^{m}\). Note that saying \(B\) is a pure rotation is equivalent to saying that \(b\) is orthogonal to the \(1\)-eigenspace \(F\) of \(B_{0}\). 

	Now, since \(A, B\) commute, \(B\) leaves invariant the \(1\)-eigenspace \(E\) of \(A\). The conclusion will come from the fact that \(B|_{E}\) is a pure rotation, so \(B\) has a fixed point in \(E\). To prove this fact, remark that since \(0 \in E\), then \(b = B \cdot 0 \in E\); this allows to write \(B|_{E} = B_{0}|_{E} + T_{b}\), where \(T_{b}\) is the translation by \(b\). Finally, since \(b\) is orthogonal to \(E \cap F\), the \(1\)-eigenspace of \(B_{0}|_{E}\), one deduces that \(B|_{E}\) is a pure rotation.
\end{proof}

\subsection{General remarks on the complex unit ball}

For more details on the material presented here, we refer to \cite{Mok12} and the appendix of \cite{Deng22}.

\subsubsection{Basic notation} For any integer \(n \in \mathbb{N}\), we let \(\mathbb{B}^{n} := \{ z \in \mathbb{C}^{n}\;|\; \sum_{j} |z_{j}|^{2} < 1\}\) denote the complex unit ball. We will also use the notation \(\Delta := \mathbb{B}^{1}\) for the unit disk, and write \(\mathbb{H} := \{z \in \mathbb{C}\;|\; \mathrm{Im}(z) > 0\}\) for the Poincar\'{e} upper half plane. The Bergman metric on \(\mathbb{B}^{n}\) will be denoted by \(h_{\mathbb{B}^{n}}\).

\subsubsection{Automorphism group.} We recall that the group of biholomorphisms of \(\mathbb{B}^{n}\) identifies with \(\PU(n, 1)\) acting on \(\mathbb{B}^{n}\) {\em via} the standard inclusions \(\mathrm{PU}(n, 1) \subset \mathrm{PGL}(n+1)\) and \(\mathbb{B}^{n} \subset \mathbb{C}^{n} \subset \mathbb{P}^{n}\).
We will often write the elements of \(\mathrm{PU}(n, 1)\) as classes of matrices of the form 
\[
	\left(\begin{array}{c|c} A & X \\ \hline Y & \mu \end{array}\right),
		\quad
	 \text{with}\;
	 A \in \mathrm{M}_{n}(\mathbb{C}),\ X, {}^{t} Y \in \mathrm{M}_{n, 1}(\mathbb{C})
	 \;\text{and}\;
	 \mu \in \mathbb{C}.
\]
Note that the action of \(\mathrm{PU}(n, 1)\) on \(\mathbb{B}^{n}\) extends to a continuous action on its closure \(\overline{\mathbb{B}^{n}}\).
\medskip

\subsubsection{Siegel model of the ball} \label{sec:siegelmodel} For each point \(b \in \partial \mathbb{B}^{n}\), there is a corresponding biholomorphism \(\phi_{b} : \mathbb{B}^{n} \to \mathbb{S}_{n}\), where \(\mathbb{S}_{n}\) is the Siegel domain
\[
	\mathbb{S}_{n} := \left\{ (y', y_{n}) \in \mathbb{C}^{n-1} \times \mathbb{C} 
	\; \big| \;
	\Im(y_{n}) > || y' ||^{2} \right\} 
\]

Letting \(l(y', y_{n}) := \Im(y_{n}) - ||y'||^{2}\), we recall that the Bergman metric on \(\mathbb{S}_{n}\) corresponds to the K\"{a}hler form
\begin{align} \label{eq:exprBergman}
	\omega_{\mathbb{S}_{n}}
	& =
	- i \partial \overline{\partial} \log(l)
	=  i \frac{\partial l \wedge \overline{\partial} l}{l^{2}} - \frac{i \partial \overline{\partial} l}{l}  
\end{align}
with \(\partial l = \frac{1}{2i} d y_{n} - \sum_{1 \leq j \leq n - 1} \overline{y}'_{j} dy'_{j}\) and \(i \partial \overline{\partial} l = - \sum_{1 \leq j \leq n - 1} i dy'_{j} \wedge d\overline{y}'_{j}\).

\subsubsection{Stabilizers of the boundary components}  \label{sect:stabboundary}
Let \(N_{b} \subset \mathrm{PU}(n, 1)\) denotes the stabilizer of \(b \in \partial \mathbb{B}^{n}\). This is a real parabolic subgroup admitting a Levi decomposition
\[
	N_{b} = W_{b} \rtimes L_{b}
\]
where \(L_{b} \cong \mathbb{R} \times \mathrm{U}(n-1)\) as a real Lie group. Using this decomposition, an element \((r, A) \in L_{b}\) acts as follows on \(\mathbb{S}_{n}\):
\[
	(r, A) \cdot (y', y_{n})
	=
	(e^{r} A \cdot y',
	e^{2r} y_{n}
	)
\]
The elements \((r, 0)\) in the factor \(\mathbb{R}\) are the so-called {\em transvections of axis \((Ob)\)}.

The unipotent radical \(W_{b}\) is a central extension
\[
	0
	\longrightarrow
	\mathbb{R}
	\longrightarrow
	W_{b}
	\longrightarrow
	\mathbb{C}^{n-1}
	\longrightarrow
	0
\]
There is a natural 1-1 correspondence \(W_{b} \overset{1-1}{=} \mathbb{C}^{n-1} \times \mathbb{R}\). In this decomposition, an element \((a, \tau)\) acts on \(\mathbb{S}_{n}\) as follows:
\[
	(a, \tau)
	\cdot
	(y', y_{n})
	=
	(y' + a,
	y_{n}
	+ 2 i \overline{a} \cdot y'
	+ i || a ||^{2}
	+ \tau).
\]

One sees from the above that the group \(N_{b}\) admits a quotient \(N_{b} \to \mathrm{Sim}(\mathbb{C}^{n-1})\) that is given by the restriction to the first factor of the inclusion \(\mathbb{S}_{n} \subset \mathbb{C}^{n-1} \times \mathbb{C}\). It can be expressed as
\[
	(a, \tau, r, A)
	\in
	\underbrace{N_{b}}
	_{\overset{\text{1-1}}{=}\,\mathbb{C}^{n-1} \times \mathbb{R} \times \mathbb{R}^{\ast}_{+} \times \U(n-1)}
	\longmapsto
	(a, A)
	\in
	\underbrace{\mathrm{Sim}(\mathbb{C}^{n-1}).}
	_{\overset{\text{1-1}}{=}\,\mathbb{C}^{n-1} \times \U(n-1)}
\]

\subsubsection{Iwasawa decomposition.} Let \(b \in \partial \mathbb{B}^{n}\), and let \(K \cong \U(n)\) be the stabilizer of the origin \(o \in \mathbb{B}^{n}\). Then, one has an equality
\[
\mathrm{PU}(n, 1)
=
K N_{b}
\]
Indeed, if \(g \in \mathrm{PU}(n, 1)\), let \(b' := g \cdot b \in \partial \mathbb{B}^{n}\), and let \(k \in K\) be such that \(k \cdot b' = b\). One has then \(k^{-1} g \in N_{b}\). Taking inverses, one also get the reversed equality
\[
	\mathrm{PU}(n, 1)
	=
	N_{b} K.
\]

\subsubsection{Stabilizer of a hyperplane in the ball.}  \label{sec:stabhyp}
Let \(\mathbb{B}^{n}\) be endowed with its standard coordinates \(z_{1}, \dotsc, z_{n}\), and let \(H := \{z_{1} = 0\} \subset \mathbb{B}^{n}\). Then, the stabilizer \(S := \mathrm{Stab}_{\mathrm{PU}(n, 1)}(H)\) identifies with the classes in \(\mathrm{PU}(n, 1)\) of matrices of the form
\[
	\left(
\begin{array}{cc|c}
	\lambda & 0 & 0  \\ 
	0 & A & X  \\
        \hline
        0 &  Y & \mu\\
\end{array}
\right)
\quad
(
|\lambda| = 1,
A \in \mathrm{M}_{n-1}(\mathbb{C}), 
X \in \mathrm{M}_{n-1, 1}(\mathbb{C}), 
Y \in \mathrm{M}_{1, n-1}, 
\mu \in \mathbb{C})
\]
In each class, the element \(\lambda\) can be fixed to be equal to \(1\), determining uniquely the elements \(A, X, Y, \mu\). This allows to identify \(S \cong U(n-1, 1)\), and thus \(S\) sits in a central exact sequence
\[
	1
	\longrightarrow 
	U(1) 
	\overset{\mathrm{diag}}
	{\longrightarrow}
	U(n-1, 1)
	\longrightarrow
	PU(n-1,1)
	\longrightarrow
	1
\]
whose right arrow corresponds to the natural morphism \(S \to \mathrm{Aut}(H) \cong \mathrm{Aut}(\mathbb{B}^{n-1})\). We also have the derived group exact sequence
\begin{equation} \label{eq:derexact}
	1
	\longrightarrow
	SU(n-1, 1)
	\longrightarrow
	U(n-1, 1)
	\overset{\mathrm{det}}{\longrightarrow}
	U(1)
	\longrightarrow
	1.
\end{equation}
The left group is the derived group of \(S\), and the maps \(U(1) \to U(1)\) and \(SU(n-1, 1) \to PU(n-1, 1)\) induced by the previous two exact sequences are both degree \(n + 1\) isogenies.
\medskip

Recall for later reference that the group \(SU(n-1, 1)\) is generated by all the transvections of axis \((Ox)\) for \(x \in \partial\mathbb{B}^{n-1} = \partial \mathbb{B}^{n} \cap H\).

\subsection{Higgs bundles and stability} We refer to \cite{GKPT19a, GKPT19b, GKPT20} for material related to Higgs sheaves on singular spaces and their associated notions of stability. Let us just recall a few facts.

\begin{defi} (see \cite[Example 5.3]{GKPT19a})
	For any normal complex space \(Y\), we let \(\mathcal{E}_{Y} := \Omega_{Y}^{[1]} \oplus \mathcal{O}_{Y}\), and \(\theta_{Y} : \mathcal{E}_{Y} \to \mathcal{E}_{Y} \otimes_{\mathcal{O}_{Y}} \Omega_{Y}\) be the standard Higgs morphism given by 
	\[
		\begin{array}{ccccc}
			\theta_{Y} : & \Omega_{Y}^{[1]}  \oplus \mathcal{O}_{Y} 
			         & \longrightarrow
				 & ( \Omega_{Y}^{[1]}  \oplus  \mathcal{O}_{Y})  \otimes \Omega_{Y}^{[1]} 
				 \\
				 & 
				 (a ,  b) 
				 & \longmapsto
				 & (0 ,  1) \otimes a 
		\end{array}
	\]	  
\end{defi}

We say that \((\mathcal{E}_{Y}, \theta)\) is {\em stable} with respect to a nef line bundle \(H\) if for any coherent subsheaf \(\mathcal{F} \subset \mathcal{E}_{Y}\) such that \(\theta(\mathcal{V}) \subset \mathcal{V} \otimes \Omega_{Y}^{[1]}\) in restriction to the maximal open subset where \(\Omega_{Y}^{[1]}\) is locally free, one has the slope inequality \(\mu_{H}(\mathcal{V}) < \mu_{H}(\mathcal{E}_{Y})\) whenever \(0 < \mathrm{rk}(\mathcal{V}) < \mathrm{rk}(\mathcal{E}_{Y})\).

In the case where \(Y\) is smooth, then this notion coincides with the classical notion of stability for Higgs {\em bundles.}

\subsection{Induced coverings} \label{sec:inducedcovering} Let \(X = \overline{X} - D\), where \(\overline{X}\) is a complex manifold of dimension \(n\), and \(D\) is a SNC divisor. For some integer \(1 \leq k \leq n\), let \(D_{k} \subset D\) be the locally closed smooth stratum of codimension \(k\), and let \(Y\) be one of its connected components. If \(X' \to X\) is any étale covering, there is an associated covering of \(Y\) that we can construct as follows.
\medskip

\begin{defi}
	Let \(i : Y \hookrightarrow  \overline{X}\) and \(j : X \hookrightarrow \overline{X}\) be the natural inclusion maps, and denote by \(\mathcal{S}\) the sheaf of sets on \(X\) given by the local sections of \(X' \to X\). Then the {\em covering \(Y' \overset{q}{\longrightarrow} Y\) induced by \(X'\to X\)} is the covering whose sheaf of local sections is \(i^{\ast} j_{\ast} \mathcal{S}\).
\end{defi}
\medskip

We may describe this covering a bit more concretely: let \(y \in Y\) be any point, and let \(U = (\Delta^{\ast})^{k} \times \Delta^{n-k} \subset X\) be a pointed polydisk adapted to \(Y\) centered at \(y\). Then the fiber product \(U \times_{X} \widetilde{X}\) is a disjoint union of copies of \(\mathbb{H}^{l} \times (\Delta^{\ast})^{k-l} \times \Delta^{n-k}\). It is easy to see that the fiber of  \(Y' \to Y\) over \(y\) is in natural 1-1 correspondence with the set of connected components
\[
	\pi_{0}(U \times_{X} \widetilde{X}).
\]
With this notation, if the point \(y' \in q^{-1}(y)\) corresponds to the component \(C_{y'} \in \pi_{0}(U \times_{X} \widetilde{X})\), then \(C_{y'}\) will be called the connected component of \(U \times_{X} \widetilde{X}\) {\em neighboring} \(y'\).
\medskip

\subsection{Multivalued maps} \label{sec:multivalued}

	If \(X, Y\) are two complex manifolds, and if \(\rho : \pi_{1}(X, *) \to \mathrm{Aut}(Y)\) is a morphism of groups, we will sometimes write "Let \(\psi : X \to Y\) be a multivalued map with monodromy \(\rho\)" instead of "Let \(\widetilde{\psi} : \widetilde{X} \to Y\) be a \(\rho\)-equivariant map". 
	
We will mostly use this terminology in the case where \(X = \Delta^{\ast}\) and \(Y = \mathbb{C}\). In this context, the universal covering will be
\[
	w \in \mathbb{H} \mapsto z = e^{2i\pi w} \in \Delta^{\ast},
\]
and if \(\alpha \in \mathbb{R}\), we will sometimes write \(z^{\alpha}\) instead of \(w^{2 i\pi\alpha}\). We will refrain from using this notation if the implied determination of logarithm matters in the discussion.

\subsection{Contractibility of analytic subsets} \label{sec:contract} Let \(X\) be a complex manifold. Recall that we say that an open subset \(U \subset X\) is {\em strictly Levi pseudo-convex} if for any point \(x \in \partial U\), there exists an open neighborhood \(x \in V \subset X\) and a \(\mathcal{C}^{2}\) function \(\phi : V \to \mathbb{R}\) such that
\begin{enumerate}[(i)]
	\item \(U \cap V = \{ y \in V \;|\; \phi(y) < 0\}\). 
	\item \(i (\partial \overline{\partial} \phi)_{x}\) is positive definite on \(\mathrm{ker}(d \varphi)_{x} \subset T_{x}X\).
\end{enumerate}

If \(A \subset X\) is a compact analytic subset, then it is possible to blow-down \(A\) to a point if and only if \(A\) admits a basis of strictly Levi pseudo-convex neighborhoods (see \cite[p.\ 104]{AMRT10}).

\subsection{Some basic results of complex analysis}

We gather here a few lemmas of analysis in one complex variable, that will allow us to prove extension results for holomorphic maps starting from \(\Delta^{\ast}\). 

The first result is very classical.

\begin{lem} \label{lem:riemannupperhalf}
	Let \(f :\Delta^{\ast} \to \mathbb{H}\) be a holomorphic map. Then \(f\) extends to a holomorphic map \(\widetilde{f} : \Delta \to \mathbb{H}\).
\end{lem}

Indeed, \(\mathbb{H}\) and \(\Delta\) are biholomorphic, so the result follows from Riemann's extension theorem. We will use the following strengthening of this statement.

\begin{lem} \label{lem:riemannboundlog}
	Let \(f : \Delta^{\ast} \to \mathbb{C}\) be a holomorphic map. Assume that there exists \(C > 0\) such that
	\[
		\Im(f(z)) > C \log |z| \quad \text{for all}\; z \in \Delta^{\ast}.
	\]
	Then \(f\) extends to a holomorphic map \(\widetilde{f} : \Delta \to \mathbb{C}\).
\end{lem}

\begin{proof}
	Consider the holomorphic map \(g : z \in \Delta^{\ast} \longmapsto g(z) = e^{i f(z)}\). Then, for all \(z \in \Delta^{\ast}\), one has
	\begin{align*}
		|g(z)| & = e^{-\Im(f(z))} \\
		       & < e^{- C \log |z|} = |z|^{-C}.
	\end{align*}
	Thus, the function \(g\) has at most a pole at \(0\). This gives two possibilities:
	\begin{enumerate}
		\item either there is \(D > 0\) such that \(|g(z)| < D\) for all \(z\) close enough to \(0\). This implies that \(\Im(f(z)) > - \log D\), and so \(f\) extends by Lemma~\ref{lem:riemannupperhalf} (applied to \(f + \log D\)).
		\item or \(|g(z)| \longrightarrow + \infty\) as \(z \to 0\). This implies that there exists \(D > 0\) with \(|g(z)| > D\) for all \(z \in \Delta^{\ast}\), and so \(\Im(f(z)) < - \log D\). Again, \(f\) extends by Lemma~\ref{lem:riemannupperhalf} (applied to \(-f - \log D\)).
	\end{enumerate}
\end{proof}

\begin{lem} \label{lem:complexanalysis3}
	Let \(f, g : \Delta^{\ast} \to \mathbb{C}\) be holomorphic functions, with no essential singularity at \(0\). Let \(\alpha, \beta\) be real numbers, with \(\alpha > 0\). Then there exists a point \(w \in \mathbb{H}\) such that
	\[
		\Im( f(z) + g(z)\, w + i \alpha\, w^{2} + \beta w) < 0
	\]
	where we let \(z = e^{2i\pi w}\).
\end{lem}

\begin{proof}
	Let \(n = \mathrm{ord}_{0}(f)\) and \(m = \mathrm{ord}_{0}(g)\). There are three possible cases :
	\begin{enumerate}
		\item \(n < m\) and \(n < 0\). \quad Then if \(a \in \mathbb{C}\) is the coefficient of \(z^{n}\) in the Taylor series of \(f\), one has
			\[
				f(z) + g(z) w + i \alpha w^{2} + \beta w = a z^{n} + o(|z|^{n}).
			\]
			(recall that \(|w| = O(|\log |z||)\)).
			We may let \(\Im(w) \to + \infty\) in such a way that \(z \to 0\) and \(\Im(a z^{n}) \to - \infty\). This gives the result.
		\item \(m \leq n\) and \(m < 0\). \quad Then, if \(b \in \mathbb{C}\) is the coefficient of \(z^{n}\) in the Taylor series of \(g\), one has
			\[
				f(z) + g(z) w + i \alpha w^{2} + \beta w = b z^{m} w + o(|z|^{m} |\log |z||).
			\]
			This time, we may let \(\Im(w) \to + \infty\) in such a way that \(\mathrm{Re}(b z^{m}) \to - \infty\) and thus \(\Im( b z^{m} w) \to - \infty\). We also get the result.
		\item \(m, n \geq 0\) i.e.\ \(f, g\) extend holomorphically across \(0\). In this case, if we let \(w = i r\) with \(r \in \mathbb{R}_{+}\), one has
			\[
				f(z) + g(z) w + i \alpha w^{2} + \beta w 
				\underset{r \to + \infty}{\sim} - \alpha r^{2} 
				\underset{r \to + \infty}{\longrightarrow} - \infty.
			\]
			This gives the result also in this case.
	\end{enumerate}
\end{proof}

\subsubsection{Campana's varieties of special type and orbifolds}

We will need a few facts pertaining to the notion of special varieties in the sense of Campana, whose general theory is exposed in \cite{Cam04}.

\begin{defi} Let \(f : X \dashedrightarrow Y\) be a rational fibration i.e. an application which is birationally equivalent to a holomorphic fibration \(f' : X' \to Y'\). Let \(D = \sum_{i = 1}^{m} a_{i} D_{i}\) be a divisor on \(X\) with simple normal crossing support and \(a_{i} \in [0, 1] \cap \mathbb{Q}\) (the pair \((X, D)\) is an {\em orbifold} in Campana's terminology).
	\begin{enumerate}
		\item The Kodaira dimension of \(f\) is 
	\[
		\kappa(Y, f) = \min \left\{
			\;
			\kappa(Y', K_{Y'} + \Delta(f', D'))
			\quad
			\big|
			\quad
			f' : (X', D') \to Y'
			\;
			\text{birationally equivalent to}
			\;
			f
			\right\}.
	\]
			In this definition, \(f'\) is a holomorphic fibration, and \(\Delta(f', D')\) denotes the orbifold divisor of \(f'\) as defined in \cite[Definition~1.29]{Cam04}.
			
\item One says that the fibration \(f\) is of {\em general type} if \(\kappa(Y, f) = \dim Y\).
\item One says that an orbifold pair \((X, D)\) is {\em special} (or of {\em special type}) if it admits no non-constant rational fibration \(f : X \dashedrightarrow Y\) of general type. Alternatively, \((X, D)\) is special if it does not admit any Bogomolov sheaf (see \cite[Theorem 2.27]{Cam04} and the discussion of \cite[p. 542]{Cam04}).
\item One says that a quasi-projective variety \(V\) is of special type if there exists a resolution of singularities \(U \to V\) and a log-compactification \(X = U \cup D\) such that \((X, D)\) is an orbifold of special type. 
\end{enumerate}
\end{defi}

The following proposition sums up the invariance properties we will use concerning specialness.

\begin{prop} \label{prop:invariancespecial}
	\begin{enumerate}[(i)]
		\item 
	{\cite[Lemma~2.9]{Cam04}}
	Let \(g : U' \dashedrightarrow U\) be a rational dominant map between two quasi-projective manifolds. If \(U'\) is special, then so is \(U\).
\item  \cite[Theorem~5.12]{Cam04} Let \(g : U' \to U\) be an finite \'{e}tale map between two quasi-projective manifolds. If \(U\) is special, then so is \(U'\).
\item Let \(f : U' \to U\) be a {\em proper} birational map between two quasi-projective manifolds. If \(U\) is special, then so is \(U'\).
	\end{enumerate}
\end{prop}

\begin{proof}
	{\em (iii)} The last item is not explicitly stated in \cite{Cam04}, and dealing with open varieties can lead to some confusion regarding birational maps, so let us explain how this follows from the definitions. We may take \(X = U \sqcup D\) and \(X' = U' \sqcup D'\) to be two projective smooth compactifications with SNC boundary, in such a way that \(f\) extends to \(f : X' \to X\), and so that \(f^{-1}(D) = D'\). But then there exists a Zariski closed subset \(\Sigma \subset X\) of codimension \(2\), such that \(f : (X', D') \to (X, D)\) is an isomorphism above \(X \setminus \Sigma\). Any Bogomolov sheaf on \((X, D')\) then restrict to \(X - \Sigma\) to define a Bogomolov sheaf for \((X, D)\) after taking the reflexive hull. Thus, \((X', D')\) is special if \((X, D)\) is (see \cite[p. 542]{Cam04} on how to define Bogomolov sheaves without referring to adapted covers).
\end{proof}

One crucial result of this theory is the following.

\begin{thm}[Campana \cite{Cam04}] \label{thm:core} Let \((X, D)\) be a complex projective orbifold pair. Then there exists a fibration of general type \(c : (X, D) \dashedrightarrow C\), whose very general fibers are special orbifolds. The fibration \(c\) is unique up to birational equivalence.
\end{thm}

The fibration of Theorem~\ref{thm:core} is called the {\em core fibration} of the pair \((X, D)\). Note that \((X, D)\) is special if and only if \(\dim C = 0\).
\medskip

\section{Miyaoka-Yau inequality and uniformizing map} \label{sec:unifmap}

\subsection{The surface case} \label{sec:surfacecase} Uniformization in the non-compact surface case is already well-understood thanks to the work of Kobayashi \cite{Kob85}. Let us simply say that if \(X^{\ast}\) is a variety with ample canonical bundle and minimal log-resolution \(\overline{X} \overset{\pi}{\longrightarrow} X^{\ast}\) with discrepancy \(-1\) for all exceptional components, then
			\[
				\Delta_{\text{MY}}(X^{\ast})
				:=
				6 c_{2}(T_{\overline{X}}(-\log D)) - 2 c_{1}^{2}(T_{\overline{X}}(-\log D))
				\geq
				0
			\]
			where \(D\) is the exceptional divisor. Indeed, in this situation, \(K_{\overline{X}} + D = \pi^{\ast} K_{X^{\ast}}\) is nef and big, and ample modulo the boundary \(D\). It is then quite easy to check that the hypotheses of \cite[Theorem~1]{Kob85} are all met. Also, in the case of equality, the regular part \(X = \overline{X} - D\) is uniformized by the ball \(\mathbb{B}^{2}\).

\subsection{The higher dimensional case} Let us now focus on the higher dimensional situation. It is quite easy to define an intrinsic Miyaoka-Yau characteristic for varieties with ample canonical bundle that are smooth in codimension \(2\), in a manner similar to \cite{GKPT19a}:

\begin{defi} \label{defi:MY}
	Let \(X^{\ast}\) be a complex \(n\)-dimensional variety smooth in codimension \(2\), and with ample canonical bundle \(K_{X^{\ast}}\). Assume that \(n \geq 3\). We define the {\em Miyaoka-Yau characteristic number} of \(X^{\ast}\) to be the rational number \(\Delta_{\text{MY}}(X^{\ast})\) computed as follows.
	
	Consider a smooth complete intersection \(S = H_{1} \cap \dotsc \cap H_{n-2}\), where \(H_{i} \in | m_{i} K_{X^{\ast}}|\) is a generic hypersurface (with \(m_{1}, \dotsc, m_{n-2} \gg 0\)). Since the singularities of \(X^{\ast}\) are only in codimension \(2\), \(S\) does not intersect the singular locus, and so \(T_{X}\) is locally free around \(S\). Thus, the following number readily makes sense:
			\begin{equation} \label{eq:defiMY}
			\Delta_{\text{MY}}(X^{\ast})
			:=
			\frac{1}{\prod_{j} m_{j}}
			\left[
				2 (n+1) c_{2}(T_{X}|_{S}) - n c_{1}^{2}(T_{X}|_{S})
			\right]
			\end{equation}
			It is easy to check that this number depends only on \(X^{\ast}\), but not on the choice of \(S\).
\end{defi}

The main goal of this section is to prove the following.

\begin{thm} \label{thm:existencemap}
	Let \(X^{\ast}\) be a complex projective variety of dimension \(n \geq 3\) that is smooth in codimension \(2\). Assume that the singularities of \(X^{\ast}\) are log-canonical, and that \(K_{X^{\ast}}\) is ample. Then one has
	\[
		\Delta_{\rm MY}(X^{\ast}) \geq 0.
	\]
	In case of equality, if one denotes by \(X \subset X^{\ast}\) the smooth locus, there exists a representation \(\rho : \pi_{1}(X) \to \PU(n,1)\) and an étale holomorphic map \(\psi : \widetilde{X} \to \mathbb{B}^{n}\) that is \(\rho\)-equivariant.
\end{thm}

\begin{proof} The beginning of the proof follows closely the main steps of \cite{GKPT19a, GKPT19b, GKPT20}, so we will present only the necessary details.
	\medskip

\noindent
	{\em Step 1. Stability of the Higgs sheaf \((\mathcal{E}_{X}, \theta_{X})\).} By the work of Guenancia \cite[Theorem A]{Gue15}, the tangent sheaf \(\mathcal{T}_{X}\) is polystable with respect to \(K_{X^{\ast}}\). This implies by the computation of \cite[Corollary~7.2]{GKPT19a} that the Higgs sheaf \((\mathcal{E}_{X}, \theta_{X})\) is stable.
	\medskip

\noindent
{\em Step 2. Restriction theorem for Higgs sheaves.} Let \(S \subset X\) be a smooth complete intersection \(S = H_{1} \cap \dotsc \cap H_{n-2}\), where \(H_{i} \in | m_{i} K_{X^{\ast}}|\) is a generic hypersurface (with \(m_{1}, \dotsc, m_{n-2} \gg 0\)). Apply now the restriction theorem \cite[Theorem 6.1]{GKPT19b} to deduce that the restriction \((\mathcal{E}_{X}, \theta_{X})|_{S}\) is stable as a Higgs {\em bundle}.

	\medskip

\noindent
{\em Step 3. We deduce the Miyaoka-Yau inequality.} Since \((\mathcal{E}_{X}, \theta_{X})|_{S}\) is a stable Higgs bundle, it satisfies the Bogomolov-Gieseker inequality by \cite[Proposition~3.4]{Sim88}, which implies that
	\[
		2 c_{2}(\mathcal{E}_{X}|_{S}) - \frac{n}{n+1}c_{1}(\mathcal{E}_{X}|_{S})^{2} \geq 0.
	\]
	If we compare with Definition~\ref{defi:MY}, we see that it implies that \(\Delta_{\rm MY}(X^{\ast}) \geq 0\).	
	\medskip

\noindent
{\em Step 4. In case of equality, we obtain the representation \(\rho\) and a period map on the surface.} As \((\mathcal{E}_{X}, \theta_{X})|_{S}\) is stable with vanishing Bogomolov-Gieseker characteristic, it underlies an irreducible {\em projective} p-\(\mathbb{C}\)VHS by \cite{Sim88}.
	Thus, fixing adapted base points \(p \in S\) and \(\widetilde{b} \in \widetilde{S}\), there exists an irreducible representation \(\rho : \pi_{1}(S) \to \mathrm{PU}(n, 1)\) and a (pointed) period map	
	\[
		\Psi_{S} : (\widetilde{S}, \widetilde{b}) \to (\mathbb{B}^{n}, o),
	\]
	that is \(\rho\)-equivariant. Also, the differential of this period map identifies with the composition
	\[
		T_{S}
		\hookrightarrow
		T_{X}|_{S}
		\overset{\theta_{X}|_{S}}{\longrightarrow}
		\mathrm{Hom}(\Omega_{X}|_{S}, \mathcal{O}_{S})
	\]
	so we deduce that \(\Psi\) is immersive at any point. 	
	\medskip

	\noindent
{\em Step 5. We show the compatibility of period maps \(\Psi_{S}\) for various surfaces \(S\).} By Goresky-McPherson's Lefschetz theorem \cite[Thm. in Sect. II.1.2]{GM88}, the natural morphism \(\pi_{1}(S) \to \pi_{1}(X)\) is an isomorphism. This implies that one has a fiber product of {\em connected} pointed spaces
	\[
		\begin{tikzcd}
			(\widetilde{S}, \widetilde{b}) \arrow[r, hook] \arrow[d] & (\widetilde{X}, \widetilde{b}) \arrow[d] \\
			(S, b) \arrow[r, hook] & (X, b)
		\end{tikzcd}
	\]
	so that when the hypersurfaces \(H_{i}\) vary in their linear system with \(S\) always containing \(b\), then the surfaces \(\widetilde{S}\) can link \(\widetilde{b}\) to any point of \(\widetilde{X}\).

	Let \(S_{1}, S_{2}\) be two such surfaces containing \(b\), and let \(\widetilde{x}\in \widetilde{S}_{1} \cap \widetilde{S}_{2}\) with projection \(x \in X\). Denote by \(\Psi_{k} : (\widetilde{S}_{k}, \widetilde{b}) \to (\mathbb{B}^{n}, o)\) the two period maps \((k=1, 2)\). We are going to show that \(\Psi_{1}(\widetilde{x}) = \Psi_{2}(\widetilde{x})\). First, pick another complete intersection surface \(S'\) passing through \(b\) and \(x\) (so that \(\widetilde{b}, \widetilde{x} \in \widetilde{S}'\)). By Bertini's theorem, we may choose \(S'\) so that both \(S' \cap S_{1}\), \(S' \cap S_{2}\) are smooth, and thus, by dealing with the two cases where the pair \((S_{1}, S_{2})\) is replaced by \((S_{1}, S')\) and \((S_{2}, S')\), we see that we may assume in turn that the curve \(C = S_{1} \cap S_{2}\) is {\em smooth}.

	Both restrictions \(\Psi_{k}|_{C \times_{X} \widetilde{X}}\) \((k=1, 2)\) are period maps associated to the same Higgs bundle \((\mathcal{E}_{X}, \theta)|_{C}\). The latter is stable since the corresponding representation \(\pi_{1}(C, b) \twoheadrightarrow \pi_{1}(X, b) \overset{\rho}{\to} \mathrm{Aut}(\mathbb{B}^{n})\) is irreducible (the first arrow is surjective by the Lefschetz theorem). By the uniqueness of harmonic metrics for stable Higgs bundles, we see that \(\Psi_{1}\) and \(\Psi_{2}\) must coincide above \(C\), and thus \(\Psi_{1}(\widetilde{x})= \Psi_{2}(\widetilde{x})\). 
	\medskip

	{\em Step 5. We glue the maps together.} The previous step shows that we may glue all period maps associated to the complete intersections \(S\) to get a well-defined \(\rho\)-equivariant map
\(
	\widetilde{X} \to \mathbb{B}^{n}.
\)
	This map is immersive. Indeed, for any tangent vector \(v\) to \(X\), there exists a surface \(S\) as above tangent to \(v\), and one has \(\Psi_{\ast}(v) = (\Psi_{S})_{\ast}(v) \neq 0\) since \(\Psi_{S}\) is immersive.
\end{proof}

\begin{rem}
	Steps 4 and 5 and this method of glueing together period maps on surfaces is directly inspired by the work of Mochizuki \cite[proof of Theorem~9.4]{Moch06}. We thank Y. Brunebarbe for pointing that reference to us.
\end{rem}

\section{Example of varieties with non-log-canonical singularities and non-complete period map} \label{sec:example}

In this section, we will use classical examples of Deligne-Mostow \cite{DM86} and Deraux \cite{Der05} to obtain varieties \(X^{\ast}\) satisfying the case of equality in Theorem~\ref{thm:existencemap} but with period map \(\psi\) failing to be an isomorphism; as we will see later on, restricting the type of singularities of \(X^{\ast}\) can be used to avoid this type of situation.

It is very easy to contract totally geodesic divisors in Deligne-Mostow and Deraux's examples to construct \(X^{\ast}\) as an abstract analytic space. One has to be a bit more careful to ensure they indeed admit a \(\mathbb{Q}\)-Cartier ample canonical divisor \(K_{X^{\ast}}\); we will provide a simple criterion for this in Theorem~\ref{thm:contractionproj}. We will then explain why this criterion applies to Deligne-Mostow and Deraux's manifolds.
\medskip

\noindent
{\bf Notation.} Let us consider a complex K\"{a}hler manifold \(X\) admitting a correspondence with the ball \(\mathbb{B}^{n}\), as given by the following diagram
\begin{equation} \label{eq:diagramMS}
	\begin{tikzcd}
			& M \arrow[r, "\psi"] \arrow[d, "\pi"] & \mathbb{B}^{n} \\
		D_{1}, \dotsc, D_{m} \arrow[r, hook] & X 
	\end{tikzcd}
\end{equation}
where
\begin{enumerate}
	\item the \(D_{i}\) are disjoint smooth irreducible divisors on \(X\) \((1 \leq i \leq m)\). Let \(D := \bigcup_{i} D_{i}\);
	\item \(\pi\) is an infinite Galois étale covering of Galois group \(\Gamma_{0} \subset \mathrm{Aut}(M)\); 
	\item \(\psi\) is a surjective holomorphic map, étale outside \(\pi^{-1}(D)\), and ramifying at order \(m_{i}\) along each component of \(\pi^{-1}(D_{i})\). 
\end{enumerate}

\noindent
{\bf Assumptions.} We make the following assumptions on this data: 
\begin{enumerate}[(a)]
	\item there exists a subgroup \(\Gamma \subset \mathrm{Aut}(\mathbb{B}^{n})\) with an action \(\Gamma \circlearrowleft M\) that makes \(\psi\) a \(\Gamma\)-equivariant map; 
	\item\(\Gamma_{0}\) acts on \(M\) as a subgroup of \(\Gamma\); 
	\item for any \(1 \leq i \leq m\), and for each connected component \(T_{i}\) of \(\pi^{-1}(D_{i})\), \(\psi\) realizes an isomorphism between \(T_{i}\) and a totally geodesic hypersurface \(H_{i} \subset \mathbb{B}^{n}\). Consequently, one has an identification
		\[
			D_{i} = \quotient{\mathrm{Stab}_{\Gamma}(H_{i})}{H_{i}},
		\]
		which makes in turn \(D_{i}\) a {\em smooth ball quotient}.
\end{enumerate}
\medskip

Note that we do not assume \(\psi\) to be a ramified cover: in Deligne-Mostow and Deraux's examples, the image \(\psi ( \pi^{-1}(D))\) is a dense set of hypersurfaces in \(\mathbb{B}^{n}\). 

\begin{rem} \label{rem:notballquotient} Since the \(D_{i}\) are ball quotients by torsion free groups, it follows from Proposition~\ref{prop:notballquotient} below that the open set \(X - D\) is not a ball quotient.
\end{rem}

\subsubsection{Singular Bergman metric on \(X\)}

The Bergman metric \(h_{\mathbb{B}^{n}}\) can be pulled back by \(\psi\) and pushed down to \(X\) via \(\pi\) to yield a singular metric \(h\) on \(T_{X}\), smooth on \(X - D\), and with conical singularities of order \(m_{i}\) around each \(D_{i}\). Note that this metric induces a well-defined distance \(d\) on the whole of \(X\). 
\medskip

For each \(i\), we pick \(\epsilon_{i} > 0\) and define 
\[
	U_{i} := \{ x \in X \; | \; d(x, D_{i}) < \epsilon_{i} \}.
\]

The following proposition follows easily from the fact that \(\Gamma_{0}\) act freely on \(M\) by isometries with respect to \(\psi^{\ast} h_{\mathbb{B}^{n}}\):

\begin{prop} \label{prop:setting}
	Let \(T_{i}\) be any connected component of \(\pi^{-1}(D_{i})\), and let 
	\[
		V_{i} := \{ x \in M \; | \; d_{\psi^{\ast}h_{\mathbb{B}^{n}}}(x, T_{i}) < \epsilon_{i}\}.
	\]
	\[
		W_{i} := \{ x \in \mathbb{B}^{n} \; | \; d(x, H_{i}) < \epsilon_{i} \},
	\]
	where \(H_{i} := \psi(T_{i}) \subset \mathbb{B}^{n}\).
	Let \(\Lambda_{i} := \mathrm{Stab}_{\Gamma_{0}}(T_{i})\). Then if \(\epsilon_{i}\) is chosen small enough, one has the following properties:
	\begin{enumerate}
		\item \(\Lambda_{i}\) acts freely on \(V_{i}\) and \(\pi|_{V_{i}} : V_{i} \to U_{i}\) can be identified with the quotient map. In particular, \(\pi_{1}(U_{i}) \cong \Lambda_{i}\);
		\item picking coordinates \(z_{1}, \dotsc, z_{n}\) on \(\mathbb{B}^{n}\) so that \(H_{i} = \{z_{1} = 0\}\), the map \(V_{i} \to W_{i}\) identifies with the corestriction to \(W_{i}\) of the map 
			\[
				\begin{array}{ccc}
				\left\{
					|z_{1}|^{2 m_{i}} + |z_{2}|^{2} + \dotsc + |z_{n}|^{2} < 1
				\right\}
				& 
				\longrightarrow 
				& 
				\mathbb{B}^{n}.
				\\
				(z_{1}, \dotsc, z_{n})
				&
				\longmapsto
				&
				(z_{1}^{m_{i}}, \dotsc, z_{n})
			\end{array}
			\]
	\end{enumerate}
\end{prop}

\begin{lem} \label{lem:pseudoconvex}
	Let \(H \subset \mathbb{B}^{n}\) be the hyperplane section \(\{z_{1} = 0\}\), and consider the following function on \(\mathbb{B}^{n}\): 
	\[
		z \in \mathbb{B}^{n}
		\longmapsto
		\delta(z) = \log \left(\frac{|z_{1}|^{2}}{1 - \sum_{j=2}^{n} |z_{j}|^{2}}\right)
	\]
	This function is invariant under \(\mathrm{Stab}_{\mathrm{Aut}(\mathbb{B}^{n})}(H)\). For any \(A \in \mathbb{R}_{+}\), the open subset \(\Omega_{A} := \{\delta < - A\}\) is a strictly pseudoconvex neighborhood of \(H\). 
\end{lem}
\begin{proof}
	Let us check the invariance. Recall from Section~\ref{sec:stabhyp} that \(S:=\mathrm{Stab}_{\mathrm{Aut(\mathbb{B}^{n})}}(H)\) is isomorphic to \(U(n-1, 1)\), acting on \(\mathbb{B}^{n}\) by its standard action on the last \(n\) homogeneous coordinates
	\[
		[ Z_{1} : Z_{2} : \dotsc : Z_{n} : Z_{0}] \in \mathbb{B}^{n} \subset \mathbb{P}^{n}
	\]
	Thus, if one takes \(z = [z_{1} : \dotsc : z_{n} : 1] \in \mathbb{B}^{n}\) and \(A \in U(n-1, 1)\), one may write \(A \cdot z = [z_{1} : T_{2} : \dotsc : T_{n} : T_{0}] \in \mathbb{B}^{n}\) with \((T_{2}, \dotsc, T_{0}) = A \cdot (z_{2}, \dotsc z_{n}, 1)\). Thus
	\[
		\delta(A \cdot z)
		=
		\log\left( 
		\frac{|z_{1}|^{2}}
		{|T_{0}|^{2} - \sum_{j = 2}^{n} |T_{j}|^{2}}
		\right).
	\]
	However, \(A\) leaves invariant the standard \((n-1, 1)\)-signature hermitian form on \(\mathbb{C}^{n}\), so \(|T_{0}|^{2} - \sum_{j=2}^{n} |T_{j}|^{2} = 1 - \sum_{j=2}^{n} |z_{j}|^{2}\). This shows the required invariance.
	\medskip

	Let us show the second statement. Let \(\pi : (z_{1}, z') \in \mathbb{B}^{n} \mapsto z' \in \mathbb{B}^{n-1}\) be the canonical projection. One has then, at any point \(z \in \partial \Omega_{A}\):
	\begin{align*}
		i \partial \overline{\partial} \delta(z) & = i \partial \overline{\partial} \log |z_{1}|^{2} - i \partial \overline{\partial} \log(1 - \sum_{j \geq 2} |z_{j}|^{2}) \\
		& = 0 + \pi^{\ast} \omega_{\mathbb{B}^{n-1}}
	\end{align*}
	where we used the standard fact that \(z' \in \mathbb{B}^{n-1} \mapsto \log(1 - ||z'||^{2})\) is a potential for the Bergman metric on \(\mathbb{B}^{n-1}\). To show that \(i \partial \overline{\partial} \delta(z)\) is positive definite in restriction to \(\ker(\partial \delta)_{z}\), is suffices to remark that the projection map
	\[
		\pi_{\ast} : \ker(\partial \delta)_{z} \subset T_{z}\, \mathbb{B}^{n}
		\to
		T_{\pi(z)}\, \mathbb{B}^{n-1}
	\]
	is injective. However, the latter is certainly true, since for \(u \in \mathrm{ker}(\pi_{\ast})_{z}\), one may write \(u = \alpha \frac{\partial}{\partial z_{1}}\), and one has then \(\partial_{u} \delta(z) \neq 0\) unless \(\alpha = 0\).
\end{proof}

\begin{prop}
	The manifold \(X\) is projective. There exists a normal complex space \(X^{\ast}\) and a morphism \(\tau : X \to X^{\ast}\) that contracts every connected component \(D_{i}\) to a single point.
\end{prop}
\begin{proof}
	The data of \(\Psi\) induces a polarized \(\mathbb{C}\)-VHS on \(X - D\), which has maximal variation since \(\Psi\) is \'{e}tale. Thus \(X\) is Moishezon by the main result of either \cite{BC18} or \cite{CD21}, hence projective since it is K\"{a}hler.

	To show that \(D_{i}\) can be contracted, use Lemma~\ref{lem:pseudoconvex} to obtain strictly convex neighborhoods \(\Omega_{A}\) of \(H_{i}\) for \(A > 0\). Then for \(A \gg 0\) so large so that \(\Omega_{A} \subset V_{i}\), the inverse image \(\psi^{-1}(\Omega_{A})\) yields a strictly pseudo-convex neighborhood of \(T_{i}\), invariant under \(\Lambda_{i}\). It thus goes down by the \'{e}tale map \(\pi\), providing in turn a strictly pseudoconvex neighborhood of \(D_{i}\). One may conclude by the criterion of Section~\ref{sec:contract}.
\end{proof}

\subsubsection{Positivity properties of orbifold canonical bundles on \(X\)} Our next goal is to study the positivity of various \(\mathbb{Q}\)-line bundles of the form \(\mathcal{O}(K_{X} + (1 + \alpha)D)\) with \(\alpha \in \mathbb{Q}\). The next lemma will allow to control the restriction of these line bundles to the boundary.

\begin{lem} \label{lem:numtrivial}
	For all \(1 \leq i \leq m\), the restriction to \(U_{i}\) of the line bundle
	\[
		N_{i} := \mathcal{O}(K_{X} + (1 + n m_{i}) D_{i})
	\]
	admits a flat unitary connection. In particular, it is numerically trivial when restricted to \(D_{i}\). 
\end{lem}

\begin{proof}
	We resume the notation of Proposition~\ref{prop:setting}, and again choose coordinates on \(\mathbb{B}^{n}\) so that \(H_{i} = \{z_{1} = 0\}\). 
\medskip

\noindent
	{\em Step 1. Corresponding line bundle on \(W_{i}\)}. Remark first that \(N_{i}\) is the line bundle descended by \(\pi\) from the \(\pi_{1}(U_{i})\)-equivariant line bundle \(M_{i} := \psi^{\ast} \mathcal{O}_{\mathbb{B}^{n}}(K_{\mathbb{B}^{n}} + (1 + n) H_{i})\) in the following diagram: 
	\[
		\begin{tikzcd}
			V_{i} \arrow[r, "\psi"] \arrow[d, "\pi"] & W_{i} \\
			U_{i}.
		\end{tikzcd}	
	\]

	The action of \(\pi_{1}(U_{i})\) on \(M_{i}\) is through the composite map
	\[
		\pi_{1}(U_{i}) \cong \Lambda_{i} \hookrightarrow \Gamma_{0}.
	\]
	This image stabilizes the totally geodesic hypersurface \(H_{i}=\{z_{1} = 0\}\), and thus defines a representation \(\rho_{i} : \pi_{1}(U_{i}) \longrightarrow \mathrm{Stab}_{\mathrm{Aut}(\mathbb{B}^{n})}(H_{i})\).

\noindent
{\em Step 2. Factorization through a unitary representation.} Consider the following global section of \(K_{\mathbb{B}^{n}}\otimes\mathcal{O}((n+1)H_{i})\), defined by
	\begin{equation} \label{eq:trivialization}
		\mathbf{e} := \frac{dz_{1} \wedge \dotsc \wedge dz_{n}}{z_{1}^{n+1}}
	\end{equation}
	By Lemma~\ref{lem:invariantsection} below, the action of \(\pi_{1}(U_{i})\) on \(\mathbf{e}\) factors through the composition \(\sigma_{i} := \mathrm{det} \circ \rho_{i} : \Lambda_{i} \to \mathrm{U}(1)\) so that
	\[
		\gamma_{\ast} \mathbf{e} = \sigma_{i}(\gamma) \mathbf{e}
	\]
	for all \(\gamma \in \pi_{1}(U_{i})\). 
	\medskip

\noindent
	{\em Step 3. Descent to \(U_{i}\) and conclusion.} By the previous step, the pullback \(\psi^{\ast} \mathbf{e}\) yields a global trivializing section of \(M_{i}\), on which \(\pi_{1}(U_{i})\) acts by the unitary representation \(\sigma_{i}\). Denoting by \(\mathbb{L}_{i}\) the complex local system associated with \(\sigma_{i}\) on \(U_{i}\), this proves that \(N_{i} = \mathbb{L}_{i} \otimes_{\mathbb{C}} \mathcal{O}_{U_{i}}\), which gives the result.
\end{proof}

The following lemma was used in the previous proof. The reader is invited to refer to Section~\ref{sec:stabhyp} for some of the notation.

\begin{lem} \label{lem:invariantsection}
	Let \(H \subset \mathbb{B}^{n}\) be the hyperplane section \(H = \{z_{1} = 0\}\), and let \(S := \mathrm{Stab}_{\mathrm{Aut}(\mathbb{B}^{n})}(H)\). Then the section
	\[
		\mathbf{e} := \frac{dz_{1} \wedge \dotsc \wedge dz_{n}}{z_{1}^{n+1}}
	\]
	is invariant under all elements of \(\mathrm{SU}(n-1, 1) \hookrightarrow S\).
\end{lem}

\begin{proof}
	The group \(SU(n-1, 1) \hookrightarrow \mathrm{Aut}(\mathbb{B}^{n})\) is generated by the standard transvections of axis \((Ox)\) where \(x \in \partial \mathbb{B}^{n-1} \cap H\), so it suffices prove invariance under any such transvection \(T\). One may find \(\alpha \in \{\mathrm{Id}_{\mathbb{C}}\} \times U(n-1)\) such that \(T = \alpha T_{0} \alpha^{-1}\), where \(T_{0}\) is a transvection in the direction \(x = (0, \dotsc, 0, 1) \in \partial \mathbb{B}^{n-1} \cap H\). The transvection \(T_{0}\) is then given by a matrix of the form
\[
	\left(
\begin{array}{ccc|c}
	1 & 0 & 0 & 0  \\ 
	0 & I_{n-2} & 0 & 0  \\
	0 & 0    &  c & s \\
        \hline
	0 &  0 & s & c\\
\end{array}
	\right)
\quad
	(
	c = \cosh(t), s = \sinh(t)
	\;\text{for some}\;
	t \in \mathbb{R})
	\]
One can now compute
	\begin{align*}
		T_{0}^{\ast} (dz_{i}) & = d\left( \frac{z_{i}}{s z_{n} + c}\right) \\
				  & = \frac{d z_{i}}{s z_{n} + c}
				      - \frac{ s z_{i} dz_{n}}{(s z_{n} + c)^{2}}
				      \quad\quad
				      (1 \leq i \leq n-1)
	\end{align*}	
	\begin{align*}
		T_{0}^{\ast} (dz_{n}) & = d\left( \frac{c z_{n} + s}{s z_{n} + c}\right) \\
				  & = \frac{d z_{n}}{(s z_{n} + c)^{2}}
				  \quad \quad
				      (\text{use}\; c^{2} - s^{2} = 1)
	\end{align*}	
	so \(T_{0}^{\ast} (d z_{1} \wedge \dotsc \wedge d z_{n}) = \frac{d z_{1} \wedge \dotsc \wedge d z_{n}}{(sz_{n} + c)^{n+1}}\).
	On the other hand,	
	\[
		T_{0}^{\ast} z_{1} = \frac{z_{1}}{s z_{n} + c}.
	\]
	Putting everything together, we get first that \(T_{0}^{\ast} \mathbf{e} = \mathbf{e}\), and then
	\begin{align*}
		T^{\ast} \mathbf{e} & =  (\alpha^{-1})^{\ast} T_{0}^{\ast} \alpha^{\ast} \mathbf{e} \\
				    & =  (\alpha^{-1})^{\ast} T_{0}^{\ast} ((\det \alpha) \mathbf{e}) \\
				    & =  (\alpha^{-1})^{\ast} ((\det \alpha) \mathbf{e}) = \mathbf{e}.
	\end{align*}
\end{proof}

\begin{defi} \label{defi:unitrepr}
	For any \(1 \leq i \leq m\), let us denote by \(\sigma_{i} : \pi_{1}(D_{i}) \to \mathrm{U}(1)\) the unitary representation underlying the line bundle \(N_{i}\), as discussed in Lemma~\ref{lem:numtrivial}. 
\end{defi}

Our goal in this section is to prove the following:
\begin{thm} \label{thm:contractionproj}
	The following are equivalent:	 	
	\begin{enumerate}[(i)]
		\item \(K_{X^{\ast}}\) is \(\mathbb{Q}\)-Cartier;
		\item every representation \(\sigma_{i}\) has finite image. 
	\end{enumerate}
	If these conditions are met, then we have the following as well:
	\begin{enumerate}
		\item \(K_{X^{\ast}}\) is ample. In particular, \(X^{\ast}\) is projective;
	\item If \(\dim X \geq 3\), the variety \(X^{\ast}\) satisfies the case of equality in the Yau-Miyaoka inequality i.e. \(\Delta_{\rm MY}(X^{\ast}) = 0\).
		\item The singularities of \(X^{\ast}\) are {\em not} log-canonical.
	\end{enumerate}
\end{thm}

We need a few lemmas before starting the proof.

\begin{lem} \label{lem:ample}
	For each \(1 \leq i \leq n\), the line bundles \(K_{D_{i}}\) and \(\mathcal{O}_{D_{i}}(-D_{i})\) are ample.
\end{lem}
\begin{proof}
	The ampleness of \(K_{D_{i}}\) is clear from the fact that \(D_{i}\) is a smooth ball quotient. As for the second assertion, remark that Lemma~\ref{lem:numtrivial} implies that
	\[
		\mathcal{O}_{D_{i}}(K_{D_{i}} + n m_{i} D_{i}) \cong \mathcal{O}(K_{X} + (1 + n m_{i}) D_{i})|_{D_{i}}
	\]
	is numerically trivial.
	\end{proof}

\begin{lem} \label{lem:ample}
	The metric \(h_{\mathbb{B}^{n}}\) induces a singular metric with positive curvature in the sense of currents on the \(\mathbb{Q}\)-line bundle 
	\[
		L := \mathcal{O}_{X}(K_{X} + \sum_{i} (1 - m_{i}) D_{i}).
	\]
	The line bundle \(L\) is ample.
\end{lem}

\begin{proof}
	{\em Step 1. \(L\) supports a orbifold positively curved metric.}
	Recall that the metric \(h\) is smooth on \(T_{X - D}\), with conical singularities of angle \(2\pi m_{i}\) around each \(D_{i}\). This implies that \(\det h\), seen as a singular metric on \(L^{\ast}\) has locally bounded potentials, and thus for any locally trivializing section  \(\eta\) of some power \(L^{\otimes m}\), one may write
	\[
		||s||^{2}_{(\det h^{\ast})^{m}} = e^{-\varphi},
	\]
	with \(\varphi\) bounded. The potential \(\varphi\) being psh outside of \(D\), it extends across the boundary as a psh function, strictly psh on \(X - D\). This shows that the singular metric \(\det (h^{\ast})\) has positive curvature in the sense of currents on \(L\). 
	\medskip

\noindent
	{\em Step 2. \(L\) is nef.} First remark that the potentials of \(\det h^{\ast}\) are bounded everywhere, so they have zero Lelong numbers on the whole \(X\) and thus the class \(c_{1}(L)\) is nef by a theorem of Demailly \cite[Corollary~6.4]{Dem92}.
\medskip

\noindent 
	{\em Step 3. Nakai-Moishezon criterion.} Let \(V \subset X\) be a any subvariety. We distinguish two cases.
	\begin{enumerate}[(a)]
		\item \(V \not\subset D\). Then, if \(f : W \to V\) is a resolution of singularities, the metric \(f^{\ast} h\) is a well-defined singular metric on \(f^{\ast} L\) with positive curvature in the sense of currents. This metric is smooth with strictly positive curvature at the generic point of \(W\). Thus, applying Boucksom's criterion for bigness \cite[Theorem 1.2]{Bou02}, one deduces that
	\begin{align*}
		\mathrm{vol}(f^{\ast} L) 
		& \geq \int_{W} (i \Theta(\det f^{\ast} h^{\ast})^{ac})^{\dim W} \\
		& = \int_{U} ( i \Theta(f^{\ast} \det h^{\ast}))^{\dim W} > 0,
	\end{align*}
	where \(U \subset W\) is the Zariski open subset where \(f\) is immersive and \(f^{\ast} h\) is smooth. This proves that \(\mathrm{vol}(f ^{\ast} L) > 0\). The latter number equals \(V \cdot c_{1}(L)^{\dim V}\) since \(L\) is nef.
		\item \(V \subset D_{i}\) for some \(1 \leq i \leq m\). Then, one has
			\[
				L|_{D_{i}}
				\cong
				\mathcal{O}_{D_{i}} (K_{D_{i}} - m_{i} D_{i})
			\]
			By Lemma~\ref{lem:ample}, one deduces that \(L|_{D_{i}}\) is ample, and thus has positive top intersection with \(V\).
	\end{enumerate}
	This is enough to obtain the ampleness of \(L\) by the Nakai-Moishezon criterion.
\end{proof}

\begin{prop} \label{prop:amplelambda}
	For any choices \(\lambda_{j} \in [-m_{i}, n m_{i}[\) for \(1 \leq j \leq m\), the line bundle
	\[
		L_{\lambda} := \mathcal{O}_{X}(K_{X} + D + \sum_{i} \lambda_{i} D_{i})
	\]
	is ample.
\end{prop}
\begin{proof} Write \(\mu_{i} = n m_{i} - \lambda_{i} > 0\). We can prove the result as follows.
\medskip

\noindent
	One claims first that for any closed subvariety \(V \subset X\), one has \(\mathrm{vol}(L_{\lambda}|_{V}) > 0\). Indeed, if \(V \subset D_{i}\) for some \(i\), this comes from the fact that
	\[
		L_{\lambda}|_{D_{i}} \equiv_{\mathrm{num}} \mathcal{O}_{D_{i}}(-\mu_{i} D_{i}).
	\]
	is ample. In the case where \(V \not\subset D\), one has \(\mathrm{vol}(L_{\lambda}|_{V}) \geq \mathrm{vol}(L|_{V}) > 0\) since \(L\) is ample.
\medskip

	Taking \(V\) to be any curve in \(X\), this shows that \(L_{\lambda}\) is nef. But then, one has \( \int_{V} c_{1}(L_{\lambda})^{\dim V} = \mathrm{vol}(V, L_{\lambda}|_{V}) > 0\) for each subvariety \(V\), and thus \(L_{\lambda}\) is ample again by the Nakai-Moishezon criterion.
\end{proof}

\begin{prop} \label{prop:ampleinterval}
	Assume that all \(\sigma_{i}\) have finite image. Let \(N := \mathcal{O}(K_{X} + \sum_{i} (1 + n m_{i}) D_{i})\). Then for \(m \gg 1\), the line bundle \(N^{\otimes m}\) is base-point free, and the associated morphism identifies with
	\[
		\tau : X 
		\longrightarrow
		X^{\ast}.
	\]
\end{prop}

\begin{proof}
	{\em Step 1. The stable base locus is included in the boundary.} Remark first that \(N\) is nef by Proposition~\ref{prop:ampleinterval}. It is also big since \(\mathrm{vol}(N) \geq \mathrm{vol}(L) > 0\). The proof of Proposition~\ref{prop:ampleinterval} shows that \(\mathrm{vol}(V, N|_{V}) > 0\) for any subvariety \(V \subset X\) with \(V \not\subset D\) and thus, Nakamaye's theorem \cite[Theorem 0.3]{Nak00} (see also Collins-Tosatti \cite{CT15}) implies that
	\[
		\mathbb{B}_{+}(N)
		=
		\bigcup_{\substack{V \subset X\\ N^{\mathrm{dim}(V)} \cdot V =0}} V
		=
		D.
	\]
	In particular, one has \(\mathbb{B}(N) \subset \mathbb{B}_{+}(N) \subset D\). The next steps will show that some power of \(N\) is also globally generated on the boundary.
	\medskip

	\noindent
	{\em Step 2. Sections on the thickenings of the boundary.} Pick \(\alpha \in \mathbb{N}\) so that \(\alpha > n m_{i} + 1\) for all \(i\), and denote by \(\alpha D\) the \(\alpha\)-th schematic thickening of the boundary divisor. One has a natural restriction map
	\[
		H^{0}(X, N^{\otimes q}) 
		\longrightarrow 
		H^{0}(\alpha D, N|_{\alpha D}^{\otimes q}).
	\]
	Let us show that this map is onto for \(q \gg 1\). The next term in the long exact sequence in cohomology is \(H^{1}(X, N^{\otimes q} \otimes \mathcal{O}(- \alpha D))\). However, one may write
	\begin{align*}
		q \left( K_{X} + \sum_{i}(1 + n m_{i}) D_{i} \right) 
		- \left(K_{X} + \alpha D\right) 
		& = (q-1) \left[ K_{X} + \sum_{i} ( 1 + \frac{q n m_{i} - \alpha + 1}{q-1}) D_{i}\right]
	\end{align*}

	Now, if \(q\) is large enough, one has
	\[
		0 < \frac{q n m_{i} - \alpha + 1}{q-1} < n m_{i}
	\]
	since \(\alpha > n m_{i} + 1\). This shows that \(N^{\otimes q} \otimes \mathcal{O}(-\alpha D - K_{X})\) is ample by Proposition~\ref{prop:amplelambda}, and thus one has the requested vanishing by Kodaira's theorem.
	\medskip

\noindent
	{\em Step 3. From thickenings to \(D\) itself.}  By our hypothesis, the line bundle \(N^{\otimes q}\) is trivial on a neighborhood of the boundary if \(q \gg 1\) is divisible enough. Thus, \(N|_{\alpha D}^{\otimes q} \cong \mathcal{O}_{\alpha D}\), and the constant section \(s \cong 1\) is sent onto \(1\) via the map
	\[
		H^{0}(\alpha D, N^{\otimes q}|_{\alpha D})
		\longrightarrow
		H^{0}(D, N^{\otimes q}) \cong H^{0}(D, \mathcal{O}_{D}) \cong \mathbb{C}.
	\]
	Joint with Step 2, this shows that the restriction map
	\[
		H^{0}(X, N^{\otimes q})
		\longrightarrow
		H^{0}(D, N^{\otimes q}|_{D}) \cong \mathbb{C}^{\pi_{0}(D)}
	\]
	is onto. This shows that \(N^{\otimes q}\) is globally generated also above \(D\), and thus \(\mathbb{B}(N^{\otimes q}) = \varnothing\).
\medskip

\noindent
	{\em End of proof.} Since \(N^{\otimes q}|_{D} \cong \mathcal{O}_{D}\), the linear system \(|N^{\otimes q}|\) contracts all the boundary components to points. Also recall that we have proved that for any curve \(C \subset X\) with \(C \not\subset D\), one has \(c_{1}(N) \cdot C > 0\). This implies that no subvariety of \(X\) is contracted unless it is included in the boundary. By classical results on semi-ample line bundles \cite[Section 2.1.B]{Laz04}, the linear system \(|L^{\otimes q}|\) realizes an isomorphism on \(X - D\), possibly after replacing \(q\) by a larger integer. This ends the proof.
\end{proof}

\begin{prop} \label{prop:iitoi}
	Under the hypothesis that \(\mathrm{Im}(\sigma_{i})\) if finite for all \(1 \leq i \leq m\), the divisor class of \(K_{X^{\ast}}\) is \(\mathbb{Q}\)-Cartier and ample, and one has \(\tau^{\ast} \mathcal{O}(K_{X^{\ast}}) \sim_{\mathbb{Q}} N\). 
\end{prop}
\begin{proof}
	Let \(\mathcal{O}(1)\) be the very ample line bundle on \(X^{\ast}\) associated with the linear system \(|N^{\otimes q}|\) for \(q \gg 1\) divisible enough, so that \(\tau^{\ast} \mathcal{O}(1) \cong N^{\otimes q}\). One may pick a hyperplane section \(s \in H^{0}(X^{\ast}, \mathcal{O}(1))\) not passing through any point of the boundary. Now, the restriction \(s|_{X - D}\) identifies with a section of \(\mathcal{O}(q K_{X- D})\), which proves that \(q K_{X^{\ast}}\) is Cartier near any point of the boundary. Also, one deduces that \(\mathcal{O}(q K_{X^{\ast}}) \cong \mathcal{O}(1)\), and thus \(K_{X^{\ast}}\) is ample.
\medskip

Since the only \(\tau\)-exceptional divisors are the \(D_{i}\), one may write
\[
	\tau^{\ast} K_{X^{\ast}} \cong_{\mathbb{Q}-\mathrm{lin}} K_{X} + \sum_{i} \alpha_{i} D_{i}
\]
	for some coefficients \(\alpha_{i}\). However, each \(D_{i}\) is contracted to a point, so this divisor has to be \(\mathbb{Q}\)-linearly trivial near each component of the boundary. As each \(D_{i}|_{D_{i}}\) is anti-ample, this implies because of Lemma~\ref{lem:numtrivial} that one needs \(\alpha_{i} = 1 + nm_{i}\) for all \(i\), and thus
	\(\tau^{\ast} K_{X^{\ast}} \sim_{\mathbb{Q}} N\), as was to be proved.
\end{proof}

In particular, the log-discrepancies of \(X^{\ast}\) at each boundary component \(D_{i}\)  are equal to \(- n m_{i} < 0\), so the singularities of \(X^{\ast}\) are not log-canonical.

\begin{prop} \label{prop:itoii}
	Assume that \(K_{X^{\ast}}\) is \(\mathbb{Q}\)-Cartier in the neighborhood of the point \(p_{i} := \tau(D_{i})\). Then the image of \(\sigma_{i}\) is finite.
\end{prop}

\begin{proof}
	Resume the notation of Proposition~\ref{prop:setting}, and denote \(U_{i}^{\ast} := \sigma(U_{i})\). Under the hypotheses, we may chose \(\epsilon_{i}\) small enough so there exists \(q \in \mathbb{N}_{>0}\) such that \(q K_{U_{i}^{\ast}}\) is cut out by a section \(s_{i}\) of its associated line bundle \(\mathcal{O}(q K_{U_{i}^{\ast}})\). We may choose \(s_{i}\), and adequately change \(q K_{X^{\ast}}\) in its linear equivalence class, so that \(p_{i} \notin \mathrm{Supp}(q K_{U_{i}^{\ast}})\). 

	But now, the pull-back \(\sigma^{\ast} s_{i}\) is a meromorphic \(q\)-canonical section, and thus gives a section of \(\mathcal{O}(q K_{U_{i}}) \otimes \mathcal{O}(\alpha D_{i})\) for some \(\alpha \in \mathbb{Z}\). By our choice of \(s_{i}\), the previous pullback is identically non-zero when restricted to \(D_{i}\) and thus
	\[
		\mathcal{O}(qK_{U_{i}}) \otimes \mathcal{O}(\alpha D_{i})
	\]
	is trivial. Lemma~\ref{lem:numtrivial} and \ref{lem:ample} together imply that \(\alpha = q(1 + nm_{i})\) and thus one deduces further that \(N_{i}^{\otimes q}|_{D_{i}}\) is trivial. Thus, \(\sigma_{i}^{\otimes q} : \pi_{1}(D_{i}) \to \mathrm{U}(1)\) is the trivial representation, and \(\sigma_{i}\) is a torsion representation.
\end{proof}

Propositions~\ref{prop:iitoi} and \ref{prop:itoii} together show that {\em (i), (ii)} are equivalent, and that {\em (1)} holds if they are satisfied. Showing that {\em (2)} holds presents no difficulty: 

\begin{prop}
	Assume that {\em (i)} and {\em (ii)} hold in Theorem~\ref{thm:contractionproj}. Then {\em (2)} holds as well.
\end{prop}
\begin{proof}
	Let \(S \subset X^{\ast}\) be a smooth complete intersection as in the hypotheses of the proposition. If \(S\) is sufficiently generic, then it does not meet any point of \(X^{\ast} - X\) and thus it identifies with its preimage \(S' \subset X - D\).

	The map \(\psi\) induces a structure of {\em flat} holomorphic vector bundle to the \(\mathcal{C}^{\infty}\) vector bundle underlying \(\mathrm{End}(\mathcal{O}_{X - D} \oplus \Omega_{X - D})\), and thus
	\begin{align*}
		\int_{S} c_{2}(\mathrm{End}(\Omega_{X^{\ast}}|_{S'} \oplus \mathcal{O}_{S'}))
			&
			= \int_{S'} c_{2}(\mathrm{End}(\Omega_{X-D} \oplus \mathcal{O}_{X-D})|_{S'})
			= 0.
	\end{align*}
	A standard computation shows that the first number is proportional to \eqref{eq:defiMY}, which gives the result. 
\end{proof}

\subsection{Deligne-Mostow and Deraux's examples} \label{sec:DMDex} Let us show in this section that several of the varieties constructed by Deligne-Mostow \cite{DM86, DM93}, or Deraux's three examples in \cite{Der05}, satisfy the hypotheses of Theorem~\ref{thm:contractionproj}.
\medskip

\subsubsection{A reminder on Deligne-Mostow's constructions.} We are going to recall several aspects of Deligne-Mostow's original construction \cite{DM86}, without going into full detail; we refer to \cite{Der11} for more information. 
\medskip

Let \(N > 3\) be an integer, and let \(\mu_{1}, \dotsc, \mu_{N}\) be positive rational numbers such that \(\sum_{j} \mu_{j} = 2\). We let \(PGL(2)\) act on the product \(P := (\mathbb{P}^{1})^{N}\) by its diagonal action, and we endow \(P\) with a polarization given by the \(\mathbb{Q}\)-line bundle \(M := \bigotimes_{j=1}^{N} \mathrm{pr}^{\ast} \mathcal{O}(\mu_{j})\). Mumford's Geometric Invariant Theory explains how to construct a GIT quotient \(Q := P\sslash \mathrm{PGL}(2)\) with respect to the linearization \(M\): this variety \(Q\) has dimension \(n := N - 3\), and is the target of a surjective map \(P^{sst} \to Q\), where \(P^{sst}\) is the {\em semi-stable locus}
\[
	P^{sst} = 
	\big\{ 
	(z_{1}, \dots, z_{n}) \in P 
	\; \big| \; 
	\text{if any}\; z_{i_{1}}, \dotsc, z_{i_{q}} \,\text{coincide, then} \sum_{j} \mu_{i_{j}} \leq 1
	\big\}.
\]
Also, \(Q\) contains as a dense open subset the quotient \(Q^{st} := \quotientd{P^{st}}{\mathrm{PGL}(2)}\) where \(P^{st}\) is the {\em stable locus}
\[
	P^{st} = 
	\big\{ 
	(z_{1}, \dots, z_{n}) \in P 
	\; \big| \; 
	\text{if any}\; z_{i_{1}}, \dotsc, z_{i_{q}} \,\text{coincide, then} \sum_{j} \mu_{i_{j}} < 1
	\big\}.
\]

Let \(Q^{0} \subset Q^{st}\) be the open locus corresponding to \(n\)-uples of all distinct points. On \(Q^{0}\), one may construct a variation of mixed Hodge structures by associating to \((z_{1}, \dotsc, z_{n})\) the natural mixed Hodge structure on \(H^{1}(\mathbb{P}^{1} - \{z_{1}, \dotsc, z_{n}\}, L_{\mu})\), where \(L_{\mu}\) is a rank \(1\) local system on \(\mathbb{P}^{1} - \{z_{1}, \dotsc, z_{n}\}\) with monodromy \(e^{2i \pi \mu_{i}}\) around each \(z_{i}\).  One has an associated period map \(\widetilde{Q^{0}} \to \mathbb{B}^{n}\), that is \(\rho_{\mu}\)-equivariant for some representation \(\rho_{\mu} : \pi_{1}(Q^{0}) \to \mathrm{PU}(n, 1)\). Let \(\Gamma_{\mu}\) be the image of \(\rho_{\mu}\); it is generated by complex reflections in \(\mathrm{PU}(n, 1)\) that correspond to letting two points \(z_{i}, z_{j}\) turn around one other (see \cite[Proposition 9.2]{DM86}).
\medskip

If no subset of the \(\mu_{i}\) sum up to \(1\), one has \(Q = Q^{st}\), and under various additional numerical conditions on the \(\mu_{i}\) detailed in \cite{DM86, DM93, Der05, Der11}, one can construct varieties sitting in a diagram
\begin{equation} \label{eq:diagramMS2}
	\begin{tikzcd}
			& M \arrow[r, "\psi"] \arrow[d, "\pi"] & \mathbb{B}^{n} \\
		D_{1}, \dotsc, D_{m} \arrow[r, hook] & X \arrow[d, "q"]  & \\
		                                     & Q                    &
	\end{tikzcd}
\end{equation}
where the top part is as in \eqref{eq:diagramMS}, and the morphism \(q : X \to Q\) is a finite map, étale over \(Q^{0}\). The variety \(Q\) is a (singular) quotient \(\quotient{\Gamma_{\mu}}{M}\), while \(X = \quotient{\Gamma_{\mu}'}{M}\) for some finite index subgroup \(\Gamma_{\mu}' \subset \Gamma_{\mu}\).

\begin{rem}
	Note that \(X - D\) could be larger than \(q^{-1}(Q^{0})\). In Deligne-Mostow original examples \cite{DM86, DM93}, the map \(\psi\) is an isomorphism so one may take \(D = \varnothing\). Alternatively, it is possible in some situations to artificially add a new totally geodesic component \(D_{i}\), taking \(m_{i} = 1\); that is what we will do in Example~2 below. 
\end{rem}

\subsubsection{Example 1. Contraction of hyperplanes in Deraux's 3-dimensional examples} In \cite{Der05}, Deraux generalizes to the dimension \(3\) a famous construction of negatively curved surfaces due to Mostow-Siu \cite{MS80} -- it yields three examples of varieties that can sit as \(Q\) in the diagram \eqref{eq:diagramMS2} above (where the \(\mu_{i}\) are given by the first three lines of [{\em op.cit} Table~1, p.516]). By contracting two exponents (i.e. replacing \(\mu_{i}\) and \(\mu_{j}\) by their sum \(\mu_{i} + \mu_{j}\)), one obtains a new datum \(\mu'\) of length shortened by \(1\). If this datum \(\mu'\) happens to satisfy the numerical condition (INT) of \cite{DM86}, then it gives a {\em ball quotient} admitting a totally geodesic embedding in \(Q\). Let \(Q' \subset Q\) be the union of all hypersurfaces obtained that way. According to \cite[Remark 4.2]{Der05}, the hypersurface \((q \circ \pi)^{-1}(Q')\) is the branch locus of \(\psi\). We take \(D := q^{-1}(Q')\).
\medskip

By [{\em op.cit.} Remark 3.5.], the divisor \(Q'\) is the disjoint union of either one or three connected components. Let \(Q_{1} \subset Q\) be one of these components, corresponding to letting points \(z_{i}, z_{j}\) coalesce. One may then write \(Q_{1} \cong \quotient{\Gamma_{\mu'}}{\mathbb{B}^{2}}\) where \(\mu'\) is as above. As in Definition~\ref{defi:unitrepr}, one has a unitary representation \(\sigma : \pi_{1}(Q_{1}^{\circ}) \to U(1)\) obtained by composing
\[
	\pi_{1}(Q_{1}^{\circ}) 
	\overset{\rho_{\mu'}}{\longrightarrow}
	\Gamma_{\mu'} 
	\longrightarrow 
	\mathrm{U}(2, 1)
	\overset{\det}{\longrightarrow}
	U(1)
\]
where \(\mathrm{U}(2, 1)\) is the stabilizer of a totally geodesic surface \(\mathbb{B}^{2} \subset \mathbb{B}^{3}\) associated to \(Q_{1}\) by the correspondence \eqref{eq:diagramMS2}.

\begin{prop} The representation \(\sigma\) has finite image.
\end{prop}
\begin{proof}
	Without loss of generality, we may renumber the \(\mu_{k}\) so that \(i = 1\) and \(j = 2\). The group \(\Gamma_{\mu'}\) is then obtained by the Deligne-Mostow construction consisting in letting \(5\) points \(z_{1} + z_{2}, z_{3}, \dotsc, z_{6}\) move on \(\mathbb{P}^{1}\). It is generated by finitely many monodromy elements, associated to the rotations of two of these five points around one another. Let \(\gamma\) be one of these monodromies, seen as an element of \(\Gamma_{\mu}\subset\mathrm{Aut}(\mathbb{B}^{3})\), and let us show that \(\gamma\) is in fact a torsion element. There are two cases to consider:
	\begin{enumerate}
		\item \(\gamma\) is given by rotating two points \(z_{k}, z_{l}\) for \(k, l \not\in \{1, 2\}\). Then, the image of \(\gamma\) in \(\Gamma_{\mu}\) is given by a complex reflection of rational angle by \cite[Proposition~9.2]{DM86}, so \(\gamma\) is torsion. 
		\item \(\gamma\) is given by rotating \(z_{k}\) around the merged point \(z_{1} + z_{2}\). Then the image of \(\gamma\) in \(\Gamma_{\mu}\) belongs to the group \(G \subset \mathrm{PU}(3,1)\) generated by the monodromies rotating only the \(3\) points \(z_{1}, z_{2}, z_{k}\). According to the discussion of \cite[p.516]{Der05}, Deraux's examples are precisely constructed so that \(G\) is finite for any triple of points \(z_{1}, z_{2}, z_{k}\). Thus \(\gamma\) is again torsion in this case. 
	\end{enumerate}

	Consequently, \(\mathrm{Im}(\sigma) \subset \mathrm{U}(1)\) is generated by a finite number of roots of unity, and thus is finite.
\end{proof}

Thus, if \(D_{1}\) is any component of \(q^{-1}(Q_{1})\), the representation of Definition~\ref{defi:unitrepr} is the composition \(\pi_{1}(D_{1}) \to \pi_{1}(Q_{1}^{\circ}) \overset{\sigma}{\to} \mathrm{U}(1)\), which has finite image. This shows that it is possible to contract all components in \(D\) in \(X\) to obtain a singular variety \(X^{\ast}\) satisfying the conclusion of Theorem~\ref{thm:contractionproj}.

\subsubsection{Example 2. Contraction of a totally geodesic surface in a 3-dimensional ball quotient} Deligne-Mostow \cite{DM86} provides several examples of compact ball quotients totally geodesically embedded in one another. Take for example \(N = 6\) and \(\mu = \frac{1}{12}(5, 5, 5, 3, 3, 3)\). Then, according to the table of \cite[p.86]{DM86}, the variety \(Q\) is a compact (singular) ball quotient by an arithmetic lattice \(\Gamma_{\mu} \subset \mathrm{Aut}(\mathbb{B}^{3})\). If \(\Gamma_{\mu}' \subset \Gamma_{\mu}\) is a sufficiently small sublattice, then \(X = \quotient{\Gamma_{\mu}'}{\mathbb{B}^{3}}\) is smooth, and we are in the situation of \eqref{eq:diagramMS2} with \(D = \varnothing\).

If one lets \(\mu_{1} = \frac{1}{12}(10, 5, 3, 3, 3)\), one finds this time a \(2\)-dimensional ball quotient \(Q_{1}\) and again for a sublattice of sufficiently high index \(\Gamma_{\mu_{1}}' \subset \Gamma_{\mu_{1}}\), we obtain a smooth ball quotient \(X_{1} = \quotient{\Gamma_{\mu_{1}}'}{\mathbb{B}^{2}}\). Now, by \cite[8.8]{DM86}, there is a totally geodesic embedding \(Q_{1} \to Q\) (whose image corresponds to the locus where \(z_{1}, z_{2} \in \mathbb{P}^{1}\) coalesce) and a natural morphism of lattices \(\Gamma_{\mu_{1}} \to \Gamma_{\mu}\). One may take \(X_{1}\) and \(X\) compatibly so as to have a diagram 
\begin{equation} 
	\begin{tikzcd}
		M_{1} \arrow[r, hook] \arrow[d] & M \arrow[r, "\psi", "\cong"'] \arrow[d, "\pi"] & \mathbb{B}^{3} \\
		X_{1} \arrow[r, hook] \arrow[d, "q_{1}"] & X \arrow[d, "q"]  & \\
		Q_{1} \arrow[r, hook] & Q                    &
	\end{tikzcd}
\end{equation}
where all horizontal arrows are totally geodesic embeddings, and \(M_{1} \cong \mathbb{B}^{2}\), \(M\cong \mathbb{B}^{3}\).
\medskip

We take \(D_{1}\) to be the image of \(X_{1}\); note that one has \(m_{1} = 1\) in this situation i.e. \(\psi\) does not ramify near \(D_{1}\). Then, with the notations of Definition~\ref{defi:unitrepr}, the representation \(\sigma_{1} : \Gamma_{\mu_{1}'} \to \mathrm{U}(1)\) is given by the composition
\[
	\sigma_{1} : \Gamma_{\mu_{1}'} \longrightarrow \mathrm{Stab}(T_{1}) \cap \Gamma_{\mu} \to \mathrm{U}(1),
\]
where \(T_{1}\) is the image of \(M_{1} \cong \mathbb{B}^{2}\) in \(\mathbb{B}^{3}\), and the second morphism is given by the right arrow in \eqref{eq:derexact}. With the same arguments as in our first example, one can show in this case that \(\mathrm{Im}(\sigma_{1})\) is finite, and thus it is possible to contract \(X_{1}\) in \(X\) to obtain again a variety satisfying Theorem~\ref{thm:contractionproj}.

\section{Special connectedness and log-canonical singularities} \label{sec:special}

Our next goal is to further investigate the situation of Theorem~\ref{thm:existencemap}, expanding on the hypothesis that the singularities of the variety \(X^{\ast}\) are log-canonical. To do this, we will need an extension to the log-canonical case of the results of Hacon-McKernan \cite{HM07}, based on Campana's theory of orbifolds and special varieties.

Let us introduce the following notion of connectedness by chains of special varieties, that generalizes the notion of connectedness by chains of rational curves.

\begin{defi} \label{defi:scc}
	Let \(X\) be a complex analytic space, not necessarily irreducible.
	One says that \(X\) is {\em specially chain connected} if any two points of \(X\) can be linked by a chain of special subvarieties of \(X\). One says that \(X\) is {\em \scc modulo} a sublocus \(Z \subset X\) if any point of \(X\) can be linked to \(Z\) by a chain of special subvarieties of \(X\). 
\end{defi}

Note that since \(\mathbb{P}^{1}\) is special, any rationally chain connected variety is specially chain connected.

\begin{rem} According to Campana's conjectures on special manifolds, a cycle should be specially chain connected if and only if it h-special in the sense of \cite[Definition 1.11]{CDY23}.
\end{rem}

The main result of this section is the following.

\begin{thm} \label{thm:scc} Let \(X\) be a complex analytic space with log-canonical singularities. Let \(\pi : Y \to X\) be a resolution of singularities. Then the fibers of \(\pi\) are \scc.
\end{thm}

The proof is strongly inspired by the work of Hacon-McKernan \cite{HM07}. We will derive the result from the following more precise statement.

\begin{thm} \label{thm:descrfibers} Let \(X\) be a complex analytic space with log-canonical singularities. Let \(p \in X\) be a point, and let \(\pi : Y \to X\) be a resolution of singularities such that both \(F := \pi^{-1}(p)\) and \(\pi^{-1}(p) \cup \pi^{-1}(\mathrm{Sing}(X))\) are divisors with normal crossings. For any prime divisor \(D \subset Y\), denote by \(a_{D}\) its discrepancy. Denote by
	\[
		\mathrm{LCS}(F) = \bigcup_{E \subset F, a_{E} = -1} E.
	\]
	the purely log-canonical locus in \(F\). Then there exists a decomposition \(F = F_{0} \cup F_{1} \cup \dotsc \cup F_{m}\), in which each \(F_{i}\) is a union of prime divisors, satisfying 
	\begin{enumerate}
		\item \(F_{0} = \mathrm{LCS}(F)\);
		\item for any \(i \geq 0\), \(F_{i+1}\) is rationally connected modulo \(F_{i}\).
		\item Let \(E\) be a component of \(F_{0}\), and let \(E^{\circ}\) be the set of points of \(E\) that do not belong to any other divisor \(D\) with \(a_{D} = -1\).  Then the quasi-projective variety \(E^{\circ}\) is special. In particular, \(E\) itself is special.
	\end{enumerate}

	{\em A fortiori}, \(F\) is \scc.
\end{thm}

Since any resolution of singularities of \(X\) can be dominated by one as in Theorem~\ref{thm:descrfibers}, we see from Proposition~\ref{prop:invariancespecial},~{\em (i)} that Theorem~\ref{thm:descrfibers} implies Theorem~\ref{thm:scc}.
\medskip

The existence of a decomposition \(F = F_{0} \cup F_{1} \cup \dotsc F_{m}\) for which Theorem~\ref{thm:descrfibers}, (1) and (2) hold, follows directly from \cite[Theorem~5.1]{HM07}. As we will see right away, the third point can be proven in a quite similar manner, replacing the use of the MRC fibration by Campana's core fibration.

\begin{proof}[Proof of Theorem~\ref{thm:descrfibers}, (3)] {\em Step 1. Preparation of the core fibration.} Let \(E \subset \mathrm{LCS}(F)\) be an irreducible component, and let \(\Lambda\) be the sum of the divisors \(D\) satisfying \(D \neq E\) and \(a_{D} = -1\). By definition, one has \(E^{\circ} = E - \Lambda\). Consider the core fibration \(c : (E, \Lambda|_{E}) \dashedrightarrow C\). We may blow-up further \(F\) at points of \(E\) to change the birational model of \(c\), so  we can assume without loss of generality that
	\[
		c : E \rightarrow C
	\]
	is a holomorphic fibration such that \(\kappa(C, K_{C} + \Delta(c)) = \dim C\). Our goal is to show that \(E^{\circ}\) is special, i.e.\ that this latter number is equal to \(0\).
	\medskip

	\noindent
	{\em Step 2. We introduce some divisors on \(E\).} We may write
	\[
		K_{Y} + \Gamma \sim_{\mathbb{Q}} \pi^{\ast} K_{X} + G,
	\]
	where \(G, \Gamma\) are effective, \(\pi\)-exceptional without common component, and \(E\) is a component of \(\Gamma\) with multiplicity \(1\). We may also ask that \(\Gamma + G\) has simple normal crossing support. Note that we have \(\Gamma \geq \Lambda + E\). Thus, if we let \(\Theta = (\Gamma - E)|_{E}\), one has
	\[
		K_{E} + \Theta = (K_{Y} + \Gamma)|_{E}
		\sim_{\mathbb{Q}}
		G|_{E}
	\]
	This latter divisor is effective, so if \(T \subset E\) denotes the general fiber of \(c\), one has
	\begin{equation} \label{eq:ineqhigher}
	\kappa(T, (K_{E} + \Theta)|_{T}) \geq 0.
	\end{equation}

	\noindent
	{\em Step 3. One shows that \(\kappa(E, K_{E} + \Theta) \leq 0\).} As in \cite{HM07}, we are going to apply a result of \cite{HM06}, stated below as Theorem~\ref{thm:refHM06}. With the notation of this statement, let us take \(C = 0\), and check that all hypotheses of the theorem are met. The first items {\em (a)} and {\em (b)} are obviously satisfied. The divisor \(K_{E} + \Theta \sim_{\mathbb{Q}} G|_{E}\) is effective as we saw above, so {\em (c)} holds as well. Finally, {\em (d)} holds since \(K_{Y} + \Gamma \sim_{\mathbb{Q}} G\) where \(G\) does not contain any stratum of \(\Gamma\).
	\medskip

	Thus, any section in \(H^{0}(E, m(K_{E} + \Theta) + H|_{E})\) lifts to a section of 
	\[
		\pi_{\ast} \mathcal{O}_{Y}(m(K_{Y} + \Gamma) + H + A)
		\cong \pi_{\ast} \mathcal{O}_{Y}(m G + H + A)
		\subset \pi_{\ast} \mathcal{O}_{Y}( H + A)^{\vee \vee}
	\]
	where the last inclusion holds since \(G\) is \(\pi\)-exceptional. Since the latter coherent sheaf does not depend on \(m\), this shows that
	\begin{equation} \label{eq:ineqlower}
		\kappa(E, K_{E} + \Theta) \leq 0.
	\end{equation}

	{\em Step 4. One applies the orbifold additivity of the Kodaira dimension.} Let \(\Delta(c, \Theta)\) be the \(\mathbb{Q}\)-divisor on \(C\) adapted to the fibration \((E, \Theta) \to C\) as it is defined in \cite[1.29]{Cam04}.
	Since \(\Gamma \geq \Lambda + E\), one has \(\Theta \geq \Lambda|_{E}\), and thus one has \(\Delta(c, \Theta) \geq \Delta(c, \Lambda|_{E})\), so that
	\[
		\kappa(C, K_{C} + \Delta(c, \Theta)) \geq \kappa(C, K_{C} + \Delta(c, \Lambda|_{E})),
	\]
	which implies that the pair \((C, \Delta(c, \Theta))\) is of general type.

	Further blowing-up \(E\) and \(C\) does not change the validity of the previous steps. One can thus assume that the morphism \((E, \Theta) \to C\) is {\em high} and {\em prepared} in the terminology of \cite{Cam04}. Thus, Campana's orbifold additivity theorem \cite[Theorem~4.2]{Cam04} (see also the discussion of \cite[p. 342]{Cam04}) implies that
	\[
		\kappa(E, K_{E} + \Theta) = \kappa(T, (K_{E} + \Theta)|_{T}) + \dim(C),
	\]
	for a general fiber \(T\) of \(c : E \to C\). Using \eqref{eq:ineqhigher} and \eqref{eq:ineqlower}, one concludes that \(\dim(C) = 0\), which implies that the pair \((E, \Theta)\) is special. Since for any component of \(E' \subset F_{0}\) with \(E' \neq E\), the divisor \(E'|_{E}\) appears in \(\Theta\) with multiplicity \(1\), this implies {\em a fortiori} that \(E^{\circ} = E \setminus \bigcup\limits_{E' \subset F_{0}, E' \neq E} E'\) is special.
\end{proof}

The next theorem is proved by Hacon--McKernan \cite{HM06}. We have replaced their symbol \(X\) (resp. \(S\)) by \(E\) (resp. \(X\)) to match our own notation.

\begin{thm}[\cite{HM06}, see Theorem~5.2 in \cite{HM07}] \label{thm:refHM06}
	Let \(Y\) be a smooth complex space, and let \(E \subset Y\) be a divisor. Let \(\pi : Y \to X\) be a projective morphism. Let \(H\) be a sufficiently \(\pi\)-very ample divisor, and let \(A = (\dim E + 1)H\). Assume that
	\begin{enumerate}[(a)]
		\item \(\Gamma\) is a \(\mathbb{Q}\)-divisor on \(Y\) with simple normal crossing support such that \(E \subset \Gamma\) has coefficient \(1\) and \(K_{Y} + \Gamma\) is log-canonical ;
		\item there exists a \(\mathbb{Q}\)-divisor \(C \geq 0\) on \(Y\) whose support does not contain \(E\);
		\item \(K_{E} + \Theta\) is pseff, where \(\Theta = (\Gamma - E)|_{E}\), and
		\item there exists a \(\mathbb{Q}\)-divisor \(G \geq 0\) on \(Y\) such that \(G \sim_{\mathbb{Q}} K_{Y} + \Gamma + C\), and \(|G|\) does not contain any log-canonical center of \((Y, \lceil \Gamma \rceil)\).
	\end{enumerate}
	Then for \(m \gg 0\) divisible enough, the image of
	\[
		\pi_{\ast} \mathcal{O}_{Y}(
		m(K_{Y} + \Gamma + C) + H + A)
		\longrightarrow
		\pi_{\ast} \mathcal{O}_{E}(
		m(K_{E} + \Theta + C|_{E}))
	\]
	contains the image of \(\pi_{\ast} \mathcal{O}_{F}(m(K_{E} + \Theta) + H|_{E})\), considered as a subsheaf of \(\pi_{\ast} \mathcal{O}_{F}(m(K_{E} + \Theta + C|_{E}) + H|_{E} + A|_{E})\).
\end{thm}

\begin{corol}  \label{corol:strataspecial}
	Let \(X\) be a complex analytic space with log-canonical singularities. Let \(p \in X\) be a point, and let \(\pi : Y \to X\) be a resolution of singularities such that both \(\pi^{-1}(p)\) and \(\pi^{-1}(p) \cup \pi^{-1}(\mathrm{Sing}(X))\) are divisors with normal crossings. Assume that for all irreducible exceptional divisor, one has \(a_{D} = - 1\). Then, any smooth locally closed stratum \(F\) of \(\pi^{-1}(p)\) is a special quasi-projective variety.
\end{corol}
\begin{proof}
	Let \(\rho : Y' \to Y\) be the blowing-up of \(Y\) along the smooth subvariety \(\overline{F}\), and denote by \(E\) its exceptional divisor. A simple computation shows that \(a_{E} = -1\), and thus all \((\pi\circ\rho)\)-exceptional divisors have discrepancy \(-1\).

Thus, denoting by \(E^{\circ}\) the points of the exceptional divisor \(E\) that do not belong to any other component, one sees by Theorem~\ref{thm:descrfibers} that \(E^{\circ}\) is special. But since the blowing-up gives a dominant morphism \(E^{\circ} \to F\), one sees that \(F\) is special by Proposition~\ref{prop:invariancespecial}~{\em (i)}.
\end{proof}

\section{Variations of Hodge structures on special varieties} \label{sec:VHSspecial}

\subsection{Complex algebraic monodromy group and isotriviality of p-\(\mathbb{C}\)VHS} \label{sec:monodromoy}

In this section, we will prove an isotriviality statement for polarized \(\mathbb{C}\)-VHS on special quasi-projective manifold, reminiscent in this setting of a similar theorem due to Taji \cite{Taji16} for families of canonically polarized varieties. The result would be a consequence of Deligne's theorem of the fixed part in the case of \(\mathbb{Q}\)-VHS, but does not seem to have been explicitly stated for p-\(\mathbb{C}\)VHS. Several of the following facts can also be seen as an adaptation to the quasi-projective case of the discussion of \cite[{\em Groups of Hodge type}, pp.46-48]{Sim92}. We will not try to recall the general theory, but will gather only the necessary facts relevant to our purposes. 
\medskip

Let \(X\) be a complex variety, endowed with a polarized \(\mathbb{C}\)-VHS \(\mathbb{V} := (E, \nabla, h, E^{p, q})\) of weight \(n\). Recall that this means that \(E = \bigoplus_{p + q = 0} E^{p, q}\) is a direct sum of \(\mathcal{C}^{\infty}\) complex vector bundles, endowed with a flat connection \(\nabla : A^{0}(X, E) \to A^{1}(X, E)\) that maps \(A^{0}(X, E^{p, q})\) into
\[
	A^{1, 0}(X, E^{p, q}) 
	\oplus A^{1, 0}(X, E^{p-1, q+1})
	\oplus A^{0, 1}(X, E^{p, q})
	\oplus A^{0, 1}(X, E^{p+1, q-1}).
\]
Accordingly, one gets a decomposition \(\nabla = \partial + \theta + \overline{\partial} + \theta^{\ast}\). The polarization \(h\) on \(E\) is a \(\nabla\)-flat hermitian pairing for which the \(E^{p, q}\) are orthogonal, with \(h\) of definite sign \((-1)^{p}\) on each \(E^{p, q}\). If \(e = \sum_{p, q} e_{p, q}\) is a section in \(E\), we will write
\[
	Q(e, e) = \sum_{p, q} (-1)^{p} h(e, \overline{e}); 	
\]
this defines a positive definite hermitian form.
\medskip

Fix a base point \(o \in X\), and let \(G_{\mathbb{R}} = \mathrm{U}(E_{o}, h_{o})\). This is a real algebraic group that acts transitively on the period domain \(\mathcal{D}\) parametrizing the Hodge structures of type given by \((E_{o}, E_{o}^{p,q}, h_{o})\). From the data associated to \(\mathbb{V}\), one gets a representation \(\rho : \pi_{1}(X, o) \to G_{\mathbb{R}}\) and a \(\rho\)-equivariant {\em period map} \(\psi : \widetilde{X} \to \mathcal{D}\).
\medskip

Another point of view on \(\mathcal{D}\) is to consider the morphism of real algebraic groups
\[
	\alpha : \mathrm{U}(1) \to G_{\mathbb{R}}
\]
that can be associated to the decomposition \(E_{o} = \bigoplus_{p, q} E_{o}^{p, q}\), letting a point \(z = e^{i\theta} \in \mathrm{U}(1)\) acts on \(E_{p,q}\) by \(z^{p} \overline{z}^{q}\). Then \(\mathcal{D}\) is in natural 1-1 correspondence with the connected component of \(\alpha\) in the \(G(\mathbb{R})\)-conjugacy class of \(\alpha\) in \(\mathrm{Hom}_{\mathbb{R}}(\mathrm{U}(1), G_{\mathbb{R}})\), in such a way that a point \(g \cdot o \in \mathcal{D}\) is sent to \(g \alpha g^{-1}\).

\begin{defi}
	We let the {\em real algebraic monodromy group} \(H_{\mathbb{R}}\) be the Zariski closure of the image of \(\rho\) in the real algebraic group \(G_{\mathbb{R}}\). 
\end{defi}

Our first goal is to show the following.

\begin{lem} \label{lem:inclusionDeligne}
	Assume \(X\) is a quasi-projective variety. Then \(\alpha(\mathrm{U}(1)) \subset H(\mathbb{R})\).
\end{lem}

\begin{proof}

\noindent
	{\em Step 1. \(H_{\mathbb{C}}\) is the stabilizer of some unitary holomorphic sub-line bundle in some \(\bigwedge^{m} E\)}. By Chevalley's classical lemma \cite{Che51}, there exists \(m \in \mathbb{N}\) and a tensor \(\omega \in \bigwedge^{m} E_{o} - \{0\}\) such that
\[
	H_{\mathbb{C}} = \{
		g \in G_{\mathbb{C}} \; | \; g \cdot \omega \in \mathbb{C} \omega
			 \}
\]
	Let us denote by \(\sigma : H \to \mathbb{C}^{\ast}\) the induced representation, so that \(g \cdot \omega = \sigma(g) \omega\) for all \(g \in H\).

	Since the line \(\mathbb{C}\omega\) is left invariant by \(\pi_{1}(X, o)\), it defines a flat \(\mathcal{C}^{\infty}\) line subbunble \((L, \nabla_{L}) \subset \bigwedge^{m} E, \nabla))\) (we also denote by \(\nabla = \overline{\partial} + \partial + \theta + \theta^{\ast}\) the flat connection inducing the natural \(\mathbb{C}\)VHS structure on \(\bigwedge^{m} E\)). Now, by Lemma~\ref{lem:mochizukibundle} below, one has
	\[
		\bigwedge^{m} E = L \, \overset{\perp}{\oplus} \, L^{\perp}
	\]
	and \(L\), \(L^{\perp}\) are both fixed by \(\overline{\partial}\), \(\theta\) and \(\theta^{\ast}\). Note however that \(\theta\) is nilpotent and \(\mathrm{rk}\, L = 1\) so actually \(\theta|_{L} = 0\), and thus \(\theta^{\ast}|_{L} = 0\). This means that 
	\[
		\nabla_{L} = (\partial + \overline{\partial})|_{L}.
	\]
	In other words, the representation \(\sigma\) induced by \((L, \nabla_{L})\) is {\em unitary}, and \(\sigma\) factors through \(\U(1) \subset \mathbb{C}^{\ast}\).
\medskip

\noindent
	{\em Step 2. One twists the representation to assume that \(H_{\mathbb{C}}\) is the fixator of a tensor in \(\bigwedge^{m} E\).} Consider the following action of \(H\) on \(\bigwedge^{m} E_{o}\):
	\[
		h \overset{\sigma}{\cdot} v = \sigma(h)^{-1} (h \cdot v)
	\]
	It is a well defined group action since the image of \(\sigma : H \to \mathrm{U}(1)\) commutes with any element of \(G(\mathbb{C})\). The associated composition \(\pi_{1}(X, o) \to H \to GL(\bigwedge^{m} E_{o})\) is the monodromy of the natural p-\(\mathbb{C}\)VHS defined on \(L^{-1} \otimes \bigwedge^{m} E\). With this notation, one has
	\begin{align*}
		H & = \{g \in G_{\mathbb{C}} \; |\; g \overset{\sigma}{\cdot} \omega \in \mathbb{C} \omega\} \\
		  & = \{g \in G_{\mathbb{C}} \; |\; g \overset{\sigma}{\cdot} \omega = \omega\}.
	\end{align*}	
	In particular, the element \(\omega\) induces a flat section \(e\) of \(L^{-1} \otimes \bigwedge^{m} E\). Since \(\theta|_{L} = 0\), one has \(\theta(e) = 0\) and similarly \(\theta^{\ast}(e) = 0\).
\medskip 

\noindent
{\em Step 3. The \((p, q)\)-components \(e^{p, q}\) of \(e\) are all flat.} Remark first that
\begin{align*}
	\theta(e^{p, q}) = 0, \theta^{\ast}(e^{p, q}) = 0
\end{align*}
	for all \(p, q\), since \(\theta(e) = 0\), \(\theta^{\ast}(e) = 0\) and both \(\theta\), \(\theta^{\ast}\) both preserve the decomposition \(E=\bigoplus_{p, q} E^{p, q}\) up to a shift. Then, the equation \(\nabla e = 0\) becomes
	\[
		(\partial + \overline{\partial} )(e) = 0,
	\]
	but since the operator \(\partial + \overline{\partial}\) preserves the \(E^{p, q}\), one gets \(0 = (\partial + \overline{\partial})(e^{p, q}) = (\partial + \overline{\partial} + \theta + \theta^{\ast})(e^{p, q}\) for all \(p, q\). This gives \(\nabla(e^{p, q}) = 0\).
\medskip

\noindent
	{\em Step 4. One concludes that \(\alpha(\mathrm{U}(1)) \subset H_{\mathbb{C}}\).} Since the \(e^{p, q}\) are flat, they are all \(H\)-invariant, and one sees immediately by double inclusion that
	\[
		H_{\mathbb{C}} = \bigcap_{p, q}\; \{ g \in G_{\mathbb{C}} \; | \;  g \overset{\sigma}{\cdot} e^{p, q} \in \mathbb{C}\, e^{p, q}. \}
	\]
	Note however that \(g \cdot e^{p, q} \in \mathbb{C} e^{p, q}\) for all \(g\in \mathrm{U}(1)\), so \(\alpha(\mathrm{U}(1)_{\mathbb{C}})\subset H_{\mathbb{C}}\). This gives the result since \(\alpha\) is a morphism of real varieties.
\end{proof}

\begin{lem}[{\cite{Moch06}, see also \cite[Lemma~3.12]{GKPT20}}] \label{lem:mochizukibundle}
	Let \(X\) be a smooth, quasi-projective variety and let \(\mathbb{V} = (E, \overline{\partial}_{E}, \theta, h)\) be a tame and purely imaginary harmonic bundle on X with induced flat connection \(\nabla_{E}\) {\em (e.g. a p-\(\mathbb{C}\)VHS)}. If \(F \subset E\) is any complex subbundle that is invariant with respect to \(\nabla_{E}\), then \(\overline{\partial}\) restricts to both \(F\) and \(F^{\perp}\) to give Higgs-invariant, holomorphic subbundles of \((E, \overline{\partial})\).
\end{lem}

\begin{prop}
	The group \(H_{\mathbb{R}}\) is reductive.
\end{prop}
\begin{proof}
	The center \(\mathfrak{z} := \mathfrak{z}(\mathfrak{g_{\mathbb{R}}}) \cong \mathbb{R}\) acts by homotheties on \(E_{o}\), so by Step~1, it is included in \(\mathfrak{h}_{\mathbb{R}}\). Thus, the adjoint algebra \(\mathfrak{g}' \cong \quotientd{\mathfrak{g}_{\mathbb{R}}}{\mathfrak{z}}\) contains naturally the quotient \(\mathfrak{h}' := \quotientd{\mathfrak{h}_{\mathbb{R}}}{\mathfrak{z}}\), and we just have to show that the latter is reductive. 
	
	The morphism \(\alpha : \mathrm{U}(1) \to G_{\mathbb{R}}\) yields a Cartan decomposition \(\mathfrak{g}'_{\mathbb{R}}  = \mathfrak{k}' \oplus \mathfrak{p}\), decomposing this space into the \(\pm 1\) eigenspaces for \(\mathrm{Ad}(\alpha(i))\). Recall that the Killing form \(B\) on \(\mathfrak{g}'\) is definite of signature \((-1, 1)\) with respect this decomposition.
	
	Since \(\alpha(\mathrm{U}(1)) \subset H_{\mathbb{R}}\), then \(\mathrm{Ad}(\alpha(i))\) leaves \(\mathfrak{h'}\) invariant and one can write
	\[
		\mathfrak{h}' 
		= 
		(\mathfrak{h}' \cap \mathfrak{k})
		\oplus
		(\mathfrak{h}' \cap \mathfrak{p}).
	\]
	But then \(B\) induces a negative definite form on the real Lie algebra
	\[
		\mathfrak{u} :=
		(\mathfrak{h}' \cap \mathfrak{k})
		\oplus
		i (\mathfrak{h}' \cap \mathfrak{p}) \subset \mathfrak{g}'_{\mathbb{C}}.
	\]
	As a consequence, \(\mathfrak{u}\) is a compact Lie algebra, and hence it is reductive (see \cite[II, Proposition~6.6]{Hel78}).	Since \(\mathfrak{u}\) and \(\mathfrak{h}'\) are two real Lie algebras with the same complexification, this shows that \(\mathfrak{h}'\) is reductive, as we wanted. 
\end{proof}

\begin{prop} \label{prop:abelianisotrivial}
	If the connected component \(H_{\mathbb{R}}^{\circ}\) is abelian, then \(\psi\) is constant. In particular, this holds if the image of \(\rho\) is virtually abelian.
\end{prop}

\begin{proof}
	We may replace \(X\) be a finite covering to assume that \(H_{\mathbb{R}}^{\circ} = H_{\mathbb{R}}\). In this case, if \(H_{\mathbb{R}}\) is connected and abelian, then \(g \alpha g^{-1} = \alpha\) for all \(g \in H(\mathbb{R})\), since \(\alpha\) has its image in \(H_{\mathbb{R}}\) by Lemma~\ref{lem:inclusionDeligne}. Thus, for all \(h \in H(\mathbb{R})\), one has
	\[
		h \cdot \psi(o) = \psi(o)
	\]
	and thus \(\psi(o)\) is fixed under the image of \(\rho\). It implies that the period map \(\psi : \widetilde{X} \to \mathcal{D}\) descends to a horizontal map \(X \to \mathcal{D}\). This map is then constant by Proposition~\ref{prop:constant}.
\end{proof}

The next result was needed in the proof of Proposition~\ref{prop:abelianisotrivial}; we include the proof for completeness. To state it correctly, we need to introduce some notation.
\medskip

Let \(V \subset G_{\mathbb{R}}\) (resp. \(K \subset G_{\mathbb{R}}\)) be the compact Lie subgroup with Lie algebra \(\mathfrak{v} := \mathfrak{g}_{\mathbb{R}} \cap \mathfrak{g}^{0, 0}\) (resp. \(\mathfrak{k}\)), one can form the associated period domain (resp. locally symmetric space of the non-compact type)
\[
	\mathcal{D} = \quotientd{G(\mathbb{R})}{V(\mathbb{R})}
	\quad
	\text{(resp. }
	\Omega = \quotientd{G(\mathbb{R})}{K(\mathbb{R})}
	\text{).}
\]

One has a natural projection map \(\pi : \mathcal{D} \to \Omega\) such that for any horizontal map \(f : V \to \mathcal{D}\) from a complex manifold, the composition \(V \to \Omega\) is pluriharmonic with respect to the invariant metric on \(\Omega\).

\begin{prop} \label{prop:constant}
	Let \(\mathcal{D}\) be the period domain for a certain type of p-\(\mathbb{C}\)VHS. Let \(U\) be a quasi-projective manifold, and let \(\psi : U \to \mathcal{D}\) be a horizontal map. Then \(\psi\) is constant.
\end{prop}
\begin{proof}
	Let \(X = U \cup D\) be a compactification of \(U\), where \(D \subset X\) is a simple normal crossing divisor. The map \(\psi\) can be seen as a p-\(\mathbb{C}\)VHS with trivial monodromies on   \(U\). Hence, by the results of Griffiths \cite{Gri68II} (see also \cite[(4.11)]{Schmid73}), the map \(\psi\) extends holomorphically to a holomorphic map \(\overline{\psi} : X \to \mathcal{D}\), which is also horizontal by continuity.
\medskip

	Now the composition \(X \to \Omega\) is a pluriharmonic map on a complex manifold, so it must be constant. This shows that \(\overline{\psi}(X)\) lies in a fiber of \(\pi\). However, \(\psi\) is horizontal, and since the horizontal directions are transversal to the fibers of \(\pi\), this shows that \(\psi\) is constant.
\end{proof}

\subsection{Isotriviality of polarized \(\mathbb{C}\)-VHS on special varieties}

\begin{defi} Given a p-\({\mathbb{C}}\)VHS \(\mathbb{V}\) on a complex manifold \(U\), we say that \(\mathbb{V}\) is {\em isotrivial} if its period map is constant.
\end{defi}

\begin{thm} \label{thm:isotrivial}
	Let \(U\) be a quasi-projective manifold. Let \(\mathbb{V} \equiv (\mathbb{V}, F^{\bullet}, h)\) be a p-\(\mathbb{C}\)VHS on \(U\). Assume that \(U\) is special. Then \(\mathbb{V}\) is isotrivial.
\end{thm}

\begin{proof}
	Let \(H_{\mathbb{R}}\) be the real algebraic monodromy group of \(\mathbb{V}\), and \(\rho :\pi_{1}(X, o) \to H(\mathbb{R})\) be the associated representation.
	\medskip

{\em Step 1. A first reduction step}. We apply \cite[Proposition~2.5]{CDY23} to obtain the existence of a diagram
	\[
		\begin{tikzcd}
			U'' \arrow[r, "\mu"] \arrow[d, "f"] & U' \arrow[r, "\nu"] &	U \\
			V
		\end{tikzcd}
	\]
	and a {\em big} representation \(\tau : \pi_{1}(V) \to H(\mathbb{R})\) such that \(f^{\ast}\tau = (\nu \circ \mu)^{\ast}\rho\). In this diagram \(\nu\) is finite étale, \(\mu\) is birational and proper, \(f\) is is a dominant morphism with connected general fibers.  By Proposition~\ref{prop:invariancespecial}, all varieties appearing in this diagram are special.
	\medskip	

	\noindent
	{\em Step 2. The algebraic group \(H_{\mathbb{R}}\) is abelian}. The representation \(\tau : \pi_{1}(V) \to H\) is a big representation with Zariski dense image in a reductive group, while \(V\) is special. Hence, by \cite[Theorem~0.9]{CDY23}, the group \(\pi_{1}(V)\) is virtually abelian, and so must be \(H = \overline{\Im(\tau)}^{\mathrm{Zar}}\). But then \(\psi\) is constant by Proposition~\ref{prop:abelianisotrivial}.
\end{proof}

\section{Asymptotic structure of period maps with values in the ball} \label{sec:asymptotic}

\subsection{The local case.}
In this section, we will describe the local structure of a period map with values in \(\mathbb{B}^{n}\). This description can be seen as a very particular case of Schmid's nilpotent orbit theorem \cite{Schmid73} or rather their versions presented by Sabbah-Schnell \cite{SS22} or Deng \cite{Deng23} in the case of \(\mathbb{C}\)VHS. We try to give a quite detailed presentation, as several features specific to the case of the ball will be crucial in our study of the uniformizing map near a log-canonical singularity.
\medskip

\subsubsection{Notation} \label{subsub:notation} We introduce the following notation.
\medskip

\begin{enumerate}[(1)]
	\item Let \(X = \left( \Delta^{\ast} \right)^{k} \times \Delta^{n-k}\) be a pointed polydisk, with its partial compactification \(\overline{X} = \Delta^{n}\). We endow \(\overline{X}\) with the standard coordinates \(z_{1}, \dotsc, z_{n}\). Let \(D = \overline{X} - X\). For any \(J \subset \llbracket 1, k \rrbracket\), we let \(D_{J} = \{ z_{j} = 0 \; | \; j \in J\}\). We fix the base point \(o := e^{-2\pi}(1, \dotsc, 1, 0, \dotsc 0) \in X\).
	\item Denoting by \(\mathbb{H}\) the Poincar\'{e} upper half-plane, we let \(\widetilde{X} := \mathbb{H}^{k} \times \Delta^{n-k}\), and endow it with the standard coordinates \(w_{1}, \dotsc, w_{n}\). We take as universal covering map \(\pi : \widetilde{X} \longrightarrow X\) the map sending \(w \in \widetilde{X}\) to
		\[
			\pi(w) = \left( e^{2i\pi w_{1}}, \dotsc, e^{2i\pi w_{k}}, w_{k+1}, \dotsc, w_{n}\right). 
		\]
		Note that \(\pi(i, \dotsc, i, 0, \dotsc, 0) = o\). We let \(F = \{ 0 \leq \mathrm{Re}(w) < 1 \}^{k} \times \Delta^{n-k}\); this is a fundamental domain for the action of \(\pi_{1}(X, o)\) on \(\widetilde{X}\). 
	\item Let \(\rho : \pi_{1}(X, o) \longrightarrow \mathrm{PU}(n, 1)\) be a representation, and consider a \(\rho\)-equivariant map \(\psi : \widetilde{X} \to \mathbb{B}^{n}\).
	Unless there is a risk of confusion, we will also use the letter \(o\) to denote the origin of \(\mathbb{B}^{n}\).
	\item Let \(\gamma_{1}, \dotsc, \gamma_{k} \in \pi_{1}(X, o)\) be the classes of the loops around the boundary components \(D_{j} := \{z_{j} = 0\}\). Let \(A_{i} = \rho(\gamma_{i}) \in \mathrm{PU}(n,1)\). The elements \(A_{i}\) are pairwise commuting, since \(\pi_{1}(X) \cong \mathbb{Z}^{k}\) is abelian.
\end{enumerate}

\subsubsection{Normal form of the period map} Consider a sequence of points \((q_{m})_{m\in\mathbb{N}} \in X^{\mathbb{N}}\) such that \(\lim\limits_{m \to \infty} q_{m} = 0 \in \overline{X}\) for the usual topology. This sequence admits a unique lift \((p_{m})_{m\in\mathbb{N}}\) to the fundamental domain \(F \subset \widetilde{X}\). We let \(b_{m} := \psi(p_{m})\) for \(m \in \mathbb{N}\). By compactness of \(\overline{\mathbb{B}^{n}}\), we may and will replace \((p_{m})\) by a subsequence so that there exists \(b_{\infty} \in \overline{\mathbb{B}^{n}}\) with
\[
	b_{m} 
	\underset{m \longrightarrow + \infty}{\longrightarrow}
	b_{\infty}
	\in \overline{\mathbb{B}^{n}}.
\]

 The main result of the section is as follows.

\begin{prop} \label{prop:asymptotic} One of the following two cases occur.

	\begin{enumerate}
		\item One has \(b_{\infty} \in \mathbb{B}^{n}\). Then there exist \(B \in \mathrm{PU}(n, 1)\) such that \(B \cdot b_{\infty} = o \in \mathbb{B}^{n}\), real numbers \((\alpha_{p, q})_{1 \leq p \leq n, 1 \leq q \leq k}\) in \([0, 1)\) and a holomorphic map \(\varphi : \overline{X} \to \mathbb{C}^{n}\) such that for any \(w = (w_{1}, \dotsc, w_{n}) \in \widetilde{X}\), one has
			\begin{equation} \label{eq:multival1}
				B \cdot \psi(w) = \left(
				e^{2i\pi \sum_{q \leq k} \alpha_{1, q} w_{q}} \varphi_{1}(z),
				\dotsc,
				e^{2i\pi \sum_{q \leq k} \alpha_{n, q} w_{q}} \varphi_{n}(z)
				\right) 	
			\end{equation}
			where \(z = \pi(w)\).
		\item One has \(b_{\infty} \in \partial \mathbb{B}^{n}\). Let \(\phi := \phi_{b_{\infty}} : \mathbb{B}^{n} \longrightarrow \mathbb{S}_{n}\) be the Siegel presentation of the ball with respect to \(b_{\infty}\). Then there exists \(B \in \mathrm{PU}(n, 1)\) fixing \(b_{\infty}\), real numbers \((\alpha_{p, q})_{1 \leq p \leq n-1, 1 \leq q \leq k}\) in \([0, 1)\), real non negative numbers \((\tau_{q})_{1 \leq q \leq k}\), not all zero, and a holomorphic map \(\varphi : \overline{X} \to \mathbb{C}^{n}\) such that for any \(w = (w_{1}, \dotsc, w_{n}) \in \widetilde{X}\), one has
			\begin{equation} \label{eq:multival2}
				B \cdot \psi(w)
				=
				\phi^{-1}
				\big(
				e^{2i\pi \sum_{q \leq k} \alpha_{1, q} w_{q}} \varphi_{1}(z),
				\dotsc,
				e^{2i\pi \sum_{q \leq k} \alpha_{n-1, q} w_{q}} \varphi_{n-1}(z),
				\varphi_{n}(z) + \sum_{q \leq k} \tau_{q} w_{q}
				\big)
			\end{equation}
			where \(z = \pi(w)\).
	\end{enumerate}
\end{prop}
\medskip

\begin{rem} Another way of stating \eqref{eq:multival1} and \eqref{eq:multival2} is to say that after composing by an automorphism \(B \in \mathrm{Aut}(\mathbb{B}^{n})\), the map \(\psi\) corresponds to a multivalued map
	\[
		z \in \Delta^{n} \longmapsto
		\mathrm{diag}(\prod_{q} z_{q}^{\alpha_{1, q}}, \dotsc, \prod_{q} z_{q}^{\alpha_{n, q}}) \cdot
		\varphi(z).
	\]
	or
	\begin{equation} \label{eq:multivalcase2}
	z \in \Delta^{n} \longmapsto
	\mathrm{diag}(\prod_{q} z_{q}^{\alpha_{1, q}}, \dotsc, \prod_{q} z_{q}^{\alpha_{n-1, q}}, 1) \cdot
	\left(
	\varphi(z)
	+
	(0, \dotsc, 0, \frac{1}{2i\pi} \sum_{q} \tau_{q} \log(z_{q}))
	\right)
	\end{equation}
\end{rem}

Let us prove Proposition~\ref{prop:asymptotic}. The following lemma is classical.

\begin{lem}  \label{lem:invlimit}
	One has \(A_{j} \cdot b_{\infty} = b_{\infty}\) for all \(j \in \llbracket 1, k \rrbracket\).
\end{lem}

\begin{proof}
	Let \((e_{1}, \dotsc, e_{n})\) be the standard frame for \(\mathbb{C}^{n}\). Letting \(d_{K, \widetilde{X}}\) be the Kobayashi distance on \(\widetilde{X}\), it is an easy exercise to check that
	\[
		d_{K, \widetilde{X}}( p_{m}, p_{m} + e_{j}) \underset{m \longrightarrow + \infty}{\longrightarrow} 0
	\]
	since \(\Im(p_{m}) \longrightarrow + \infty\). The decreasing property of the Kobayashi distance implies that
	\[
		d_{\mathbb{B}^{n}}(\psi(p_{m}), \psi(p_{m} + e_{j})) \leq d_{K, \widetilde{X}}( p_{m}, p_{m} + e_{j}) \longrightarrow 0
	\]
Now, one has
	\[
	 \psi(p_{m}) = b_{m}
	 \quad
	 \text{and}
	 \quad
	 \psi(p_{m} + e_{j})
	 = A_{j} \cdot b_{m},
	\]
	so \(d_{\mathbb{B}^{n}}(b_{m}, A_{j} \cdot b_{m})\) tends to \(0\). The euclidean distance on \(\mathbb{B}^{n}\) is bounded from above by the hyperbolic distance, so passing to the limit gives
	\(d_{\mathrm{eucl}}(b_{\infty}, A_{j} \cdot b_{\infty}) = 0\).
	This proves the result.
\end{proof}

To prove Proposition~\ref{prop:asymptotic}, we will deal separately with the two possible cases.
\medskip

\subsubsection{Case 1: \(b_{\infty} \in \mathbb{B}^{n}\)} This case is the easiest one. Pick \(B \in \mathrm{PU}(n, 1)\) such that \(B \cdot b_{\infty} = o\). Then we may replace \(\psi\) by \(B \cdot \psi\) and the \(A_{j}\) by \(B A_{j} B^{-1}\) to assume \(b_{\infty} = o\). Then, since \(A_{j} \cdot o = o\), all the elements \(A_{j}\) belong to \(\U(n) = \mathrm{Stab}_{o}(\PU(n, 1))\). Finally, as these elements are pairwise commuting, we may find again an element \(B \in \U(n)\) such that
\[
	B A_{j} B^{-1} = \mathrm{diag}
	(e^ {2 i \pi \alpha_{1, j}},
	\dotsc,
	e^ {2 i \pi \alpha_{n, j}}
	)
	\quad 
	(1 \leq j \leq k)
\]
in \(\U(n)\), where \(\alpha_{p, q} \in [0, 1)\) for all \(p, q\). Again, we may replace \(\psi\) and the \(A_{j}\) as above to assume that \(A_{j}\) are written as the previous diagonal form. 
\medskip

We may now untwist \(\psi\) by the \(n\)-parameter group
\[
	M: w \in (\mathbb{C}^{\ast})^{n}
	\longmapsto
	\mathrm{diag}
	(e^{2i\pi \sum_{q} \alpha_{1, q} w_{q}},
		\dotsc,
	e^ {2 i \pi \sum_{q} \alpha_{n, q} w_{q}}
	) \in \PGL(n)
\]
to get a single valued holomorphic map \(\varphi : z \in X \longmapsto \varphi(z) = M(-w) \cdot \psi(w)\), where \(z = \pi(w)\).
\medskip

To check that \(\varphi\) can be extended to \(\overline{X}\), it suffices to remark that since \(\psi\) lands in \(\mathbb{B}^{n}\), we must have for all \( j \in \llbracket 1, n \rrbracket\):

\[
	\prod_{q = 1}^{n} |z_{q}|^{\alpha_{j, q}} |\varphi_{j}(z)| = |\psi_{j}(w)| \leq 1.
\]
Since we have \(\alpha_{j, q} < 1\) for all \(j\), the holomorphic function \(\varphi_{j}\) cannot have have a pole along any component of \(D\). This ends the proof in the first case.

\subsubsection{Case 2: \(b_{\infty} \in \partial \mathbb{B}^{n}\)} 

With the notation of Proposition~\ref{prop:asymptotic}, we will replace for simplicity \(\mathbb{B}^{n}\) by its Siegel model \(\mathbb{S}_{n}\), and the map \(\varphi\) by the composite \(\phi \circ \varphi\), so as to consider a \(\rho\)-equivariant holomorphic map \(\varphi : \widetilde{X} \to \mathbb{S}_{n}\).
\medskip

As explained in Section~\ref{sect:stabboundary}, one has a decomposition \(N_{b_{\infty}} \cong  L_{b_{\infty}} \ltimes W_{b_{\infty}}\), where \(L_{b_{\infty}} \cong \mathrm{U}(n-1) \times \mathbb{R}\). For each \(j \in \llbracket 1, k \rrbracket\), let us write 
\begin{equation} \label{eq:levi}
	\hspace{4cm}
	A_{j} = (R_{j}, r_{j}, M_{j}) 
	\hspace{2cm}
	(R_{j} \in \mathrm{U}(n-1),
	r_{j} \in \mathbb{R},
	M_{j} \in W_{b_{\infty}})
\end{equation}
accordingly to this decomposition. The following two lemmas will give some restrictions concerning the elements of this expression; the first one is certainly well known, but we include a proof for lack of a better reference. 

\begin{lem}
	One has \(r_{j} = 0\) for all \(j \in \llbracket 1, k\rrbracket\).
\end{lem}

\begin{proof}
	Assume by contradiction that \(r_{j} \neq 0\). As we have seen in the proof of Lemma~\ref{lem:invlimit}, one has \(d_{\mathbb{B}^{n}}( b_{m}, A_{j} \cdot b_{m}) \longrightarrow 0\). Thus, to obtain a contradiction, it suffices to show that under our hypothesis, there exists \(\epsilon > 0\) depending only on \(r_{j}\) such that
	\[
		\forall p \in \mathbb{B}^{n},
		\;
		d_{\mathbb{B}^{n}}(p, A_{j} \cdot p) > \epsilon 
	\]

	\noindent
	{\em Step 1. We reduce to the case \(p = o\)}. Let \(p \in \mathbb{B}^{n}\), and let \(B_{p} \in \mathrm{PU}(n, 1)\) such that \(p = B_{p} \cdot o\). Using the Iwasawa decomposition, one can write \(B_{p} = S_{p} T_{p}\) with \(S_{p} \in N_{b_{\infty}}\) and \(T_{p} \in K := \mathrm{Stab}_{o}(\mathrm{PU}(n, 1))\). Then, one has
	\begin{align*}
		d_{\mathbb{B}^{n}}(p, A_{j} \cdot p) & = 
		d_{\mathbb{B}^{n}}(S_{p} T_{p} \cdot o, A_{j} S_{p} T_{p} \cdot o) \\ 
		& = d_{\mathbb{B}^{n}}(S_{p} \cdot o, A_{j} S_{p} \cdot o)  \hspace{2cm} (\text{since}\; T_{p} \in K) \\
		& = d_{\mathbb{B}^{n}}(o, \widehat{A}_{j} \cdot o)
	\end{align*}
	where \(\widehat{A}_{j} = S_{p}^{-1} A_{j} S_{p}\). Writing this element as in \eqref{eq:levi}, one finds \(\widehat{A}_{j} = (\widehat{R}_{j}, \widehat{r}_{j}, \widehat{M}_{j})\). Note that one must have \(\widehat{r}_{j} = r_{j}\) since \((\widehat{R}_{j}, \widehat{r}_{j})\) is the image of \(\widehat{A}_{j}\) under the quotient map \(N_{b_{\infty}} \to L_{b_{\infty}}\), and conjugation leaves the center \(\mathbb{R} \subset L_{b_{\infty}}\) invariant. We have reduced to showing that if \(A = (R, r, M)\) is written as in \eqref{eq:levi} with \(r \neq 0\), one has \(d_{\mathbb{B}^{n}}(o, A \cdot o) > \epsilon\), where \(\epsilon\) depends only on \(r\). We may assume \(r > 0\) by replacing \(A\) with \(A^{-1}\) if necessary.
\medskip

\noindent
	{\em Step 2. We reduce to the one dimensional case.} Write \(R = (\tau, a)\), accordingly to the notation of Section~\ref{sect:stabboundary}. Now, in the standard coordinates of the Siegel domain \(\mathbb{S}_{n}\), one has \(o = (0, \dotsc, 0, i)\), and one can write
	\[
		A \cdot o
		=
		(e^{r} R \cdot a, e^{2r}(i + i ||a||^{2} + \tau)) 
	\]
	The projection on the last coordinate \(\mathbb{S}_{n} \to \mathbb{H}\) is distance decreasing for the Kobayashi distance, so one has
	\[
		d_{\mathbb{B}^{n}}(A \cdot o, o)
		\geq
		d_{\mathbb{H}}(e^{2r}(1 + ||a||^{2}) i + e^{r} \tau, i)
	\]
	Now, it is easy to check that for any \(\lambda > 1\), and any \(t \in \mathbb{R}\), one has
	\(d_{\mathbb{H}}(i, \lambda i + t) \geq \log(\lambda)\). This shows that
	\[
		d_{\mathbb{B}^{n}}(A \cdot o, o)
		\geq
		2 r + \log(1 + ||a||^{2}) \geq 2r.
	\]
	This gives the result.
\end{proof}

We are now ready to make a first conjugation by an element \(B_{1} \in \mathrm{PU}(n, 1)\) -- this \(B_{1}\) will eventually give a factor of the element \(B\) mentioned in the statement of Proposition~\ref{prop:asymptotic}. Since all the matrices \(A_{j}\) are pairwise commuting, so are the elements \(R_{j}\), since the latter are the images of the \(A_{j}\) by the quotient map \(N_{b_{\infty}} \to L_{b_{\infty}}\). Thus, they may be diagonalized in a single unitary basis, meaning that there exists \(B_{1} \in \mathrm{U}(n-1)\) such that
\begin{equation} \label{eq:diagmatrix}
	B_{1} R_{q} B_{1}^{-1} = \mathrm{diag}(e^{2i\pi \alpha_{1, q}}, \dotsc, e^{2i\pi \alpha_{n-1, q}})
	\quad
	(1 \leq q \leq k)
\end{equation}
where the \((\alpha_{p, q})_{1 \leq p \leq n-1, 1 \leq q \leq k}\) are real numbers.

Without loss of generality, we may replace \(\psi\) by \(B_{1} \cdot \psi\) and the \(A_{q}\) by \(B_{1} A_{q} B_{1}^{-1}\), to assume all the \(R_{q}\) have the diagonal form \eqref{eq:diagmatrix}.

\begin{lem}
	For any \(j \in \llbracket 1, k \rrbracket\), write \(M_{j} = (\tau_{j}, a_{j}) \in \mathbb{R} \times \mathbb{C}^{n-1}\), accordingly to the decomposition of \(W_{b_{\infty}}\) detailed in Section~\ref{sect:stabboundary}. Then the image \((R_{j}, a_{j})\) of \(A_{j}\) under the quotient \(N_{b_{\infty}} \to \mathrm{Sim}(\mathbb{C}^{n-1}) \cong \U(n-1) \ltimes \mathbb{C}^{n-1}\) is a pure rotation.
\end{lem}

\begin{proof}
	Recall that \(R_{j}\) has the diagonal expression \eqref{eq:diagmatrix}; we may reorder the coordinates so that 
	\[
		R_{j}
		=
		\mathrm{diag}(e^{2i\pi \alpha_{1, j}}, \dotsc, e^{2i\pi \alpha_{p, j}}, 1, \dotsc, 1)
		=: \mathrm{diag}(R'_{j}, \mathrm{I}_{n-1 - p}).
	\]
	with the first \(p\) coordinates are all different from \(1\). The matrix \(R_{j}'\) is a diagonal unitary matrix of size \(p\).

	It follows from Lemma~\ref{lem:centunitsim} that we may conjugate \((R_{j}, a_{j})\) by a translation of in \(\mathbb{C}^{n-1}\), to assume that the vector \(a_{j}\) has zero coordinates except possibly for the last \(n-1-p\) entries, so that
	\[
		a_{j} = (0, a_{j}') \in \mathbb{C}^{p} \times \mathbb{C}^{n-1 - p}.
	\]
	Such a conjugation in \(\mathrm{Sim}(\mathbb{C}^{n-1})\) can by obtained by composing \(\psi\) with an adequate element in \(\PU(n,1)\). In the following, we assume that this conjugation has been made.
	\medskip
	
	 With the previous notation, the action of \(A_{j}\) on \(S\) then splits accordingly to the decomposition \(\mathbb{C}^{n-1} = \mathbb{C}^{p} \times \mathbb{C}^{n-1-p}\) to give a rotation on one factor, and a translation on the other. More precisely, if we pick \((y', y_{n}) \in S \subset \mathbb{C}^{n-1} \times \mathbb{C}\) and write \(y' =: (y'_{rot}, y'_{tran}) \in \mathbb{C}^{p} \times \mathbb{C}^{n-1 - p}\), then one has
	\begin{equation} \label{eq:transforule}
		A_{j} \cdot (y', y_{n})
		=
		(R'_{j} \cdot y_{rot}',\;
		y'_{tran} + a_{j}',\;
		y_{n} + 2 i \overline{a_{j}'} \cdot y_{tran}'
		+
		i || a_{j}' ||^{2} + \tau_{j}).
	\end{equation}
	\medskip

	The proof that \((R_{j}, a_{j})\) is a pure rotation will be complete with Lemma~\ref{lem:rotation}, which will make use of the \(\mathbb{C}\)-VHS version of Schmid's nilpotent orbit theorem \cite{Schmid73} proved by Sabbah-Schnell \cite{SS22}.
\end{proof}

	\begin{lem} \label{lem:rotation}
		We have \(a_{j}' = 0\) for all \(j \in \llbracket 1, k\rrbracket\).
	\end{lem}
	
	\begin{proof}[Proof of Lemma~\ref{lem:rotation}] We will derive a contradiction by assuming the contrary. Choose \(j \in \llbracket 1, k\rrbracket\), and fix a constant \(w_{l}\) for all \(l \neq j\), so as to consider the one-variable map
	\[
		\zeta : w \in \mathbb{H}
		\mapsto
		\zeta(w) :=
		\psi(w_{1}, \dotsc, w_{j-1}, w, w_{j+1}, \dotsc, w_{n})
	\]
	We have a natural decomposition
	\[
		\zeta(w) 
		=
		(\zeta_{rot}(w), \zeta_{tran}(w), \zeta_{n}(w)) \in \mathbb{C}^{p} \times \mathbb{C}^{n-1-p} \times \mathbb{C}.
	\]
		Recall that \(\zeta(w+1) = A_{j} \cdot \zeta(w)\). Projecting \(\zeta\) to its last \(n-p\) coordinates, one gets a holomorphic map with values in the \((n-p)\)-dimensional Siegel domain \(\mathbb{S}_{n-p} := \{(w', w_{n}) \in \mathbb{C}^{n-p-1} \times \mathbb{C} \; |\; \Im(w_{n}) > || w' ||^{2} \}\):
		\[
			\zeta' : w \in \mathbb{H} \longmapsto (\zeta_{tran}(w), \zeta_{n}(w)) \in S_{n-p} \subset \mathbb{P}^{n-p}.
		\]

		The map \(\zeta'\) can be seen as a \(\mathbb{C}\)-VHS with values in the compact period domain \(\mathbb{P}^{n-p}\). We see from \eqref{eq:transforule} that its monodromy is given by the unipotent element 
		\[
			\left(
			\begin{matrix}
				\mathrm{I}_{n-p-1} & 0 & a_{j}' \\
				2 i {}^{t} \overline{a_{j}'} & 1 & i ||a_{j}'||^{2} + \tau_{j}  \\
			        0                     & 0 & 1	
			\end{matrix}
			\right)
			\in
			\mathrm{Aut}(\mathbb{P}^{n-p}).
		\]
		(the vector \(a_{j}'\) is considered here as a line matrix) The previous element is the value for \(w=1\) of the one-parameter group
		\[
			M : w \in \mathbb{C}^{\ast}
			\longmapsto
			\left(
			\begin{matrix}
				\mathrm{I}_{n-p-1} & 0 & a_{j}' w \\
				2 i {}^{t} \overline{a_{j}'} w & 1 & i ||a_{j}'||^{2} w^{2} + \tau_{j} w  \\
			        0                     & 0 & 1	
			\end{matrix}
			\right)
			\in \PGL(n-p+1).
		\]

		The nilpotent orbit theorem (see \cite[4.9]{Schmid73}, \cite[{\bf 24.}]{SS22}) has the following consequence:
		\medskip
		
\begin{claim*} The monovalued untwisted period map
\[
	z \in \Delta^{\ast} \mapsto M(-w) \cdot [ \zeta_{tran}(w) : \zeta_{n}(w) : 1] \in \mathbb{P}^{n-p}
\]
extends holomorphically across the origin.
\end{claim*}

		We may thus write the application appearing in the previous claim as \(z \in \Delta \mapsto [\phi_{tran}(z) : \phi_{n}(z) : 1]\) where \(\phi_{tran} : \Delta^{\ast} \to \mathbb{C}^{n-p-1}\) (resp. \(\phi_{n} : \Delta^{\ast} \to \mathbb{C}\)) is a holomorphic map that is meromorphic across the origin. Twisting this map back by \(M(w)\), one may write 
		\[
		\zeta'(w)
		=
		(
		\phi_{tran}(z) + a_{j}' w,
		\phi_{n}(z) + 2 i \overline{a_{j}'} \cdot \phi_{tran}(z) w + i || a_{j}'||^{2} w^{2} + \tau_{j} w).
		\]

Note that this expression must land inside \(\mathbb{S}_{n-p}\) for all \(w \in \mathbb{H}\). This is however impossible if \(a_{j}' \neq 0\); to remark this, apply Lemma~\ref{lem:complexanalysis3} to \(f(z) = \phi_{n}(z)\), \(g(z) = 2 i \overline{a_{j}}' \cdot \phi_{tran}(z)\), \(\alpha = || a_{j}'||^{2}\) and \(\beta = \tau_{j}\). This ends the proof of Lemma~\ref{lem:rotation}.
\end{proof}

	We can now prove that in Case 2, the function \(\psi\) can be written as in \eqref{eq:multival2}. We have proven that all monodromies \(A_{j} = (R_{j}, 0, \tau_{j}, a_{j})\) are such that the elements \((R_{j}, a_{j})\) are pure rotations of \(\mathbb{C}^{m}\). These elements are all pairwise commuting because the \(A_{j}\) commute. Thus, by Lemma~\ref{lem:purerotcomm}, they all share a common fixed point : this allows to find a new element \(B_{2} \in W_{b_{\infty}}\) such that 
	\[
		B_{2} A_{j} B_{2}^{-1} = (R_{j}, 0, \tau_{j}, 0)
	\]
	for all \(j \in \llbracket 1, k \rrbracket\). Again, without loss of generality, we may replace \(\psi\) by \(B_{2} \cdot \psi\) and \(A_{j}\) by \(B_{2} A_{j} B_{2}^{-1}\). Applying the nilpotent orbit theorem once more, we may now untwist the map \(\psi\) by the \(n\)-parameter group
	\[
		M : (w_{i}) \in (\mathbb{C}^{\ast})^{n}
		\longmapsto
		\left(
		\begin{matrix}
		e^{2 i \pi \sum_{q} \alpha_{1, q} w_{q}} &        &                                             &                              \\
							 & \ddots &                                             &                              \\
							 &        &  e^{2 i \pi \sum_{q} \alpha_{n-1, q} w_{q}} &                              \\
							 &        &                                             & 1 & \sum_{q} \tau_{q} w_{q}  \\
							 &        &                                             &   & 1                        

		\end{matrix}
		\right) \in \mathrm{PGL}(n+1)
	\]
	to get a holomorphic map \(\varphi : z \in X \longmapsto M(-w) \cdot \psi(w) \in \mathbb{C}^{n}\) that is actually meromorphic across the boundary divisor. Equation \eqref{eq:multival2} now follows from the definition of \(\varphi\).
	\medskip

It remains to be shown that the function \(\varphi\) extends holomorphically across \(D\).

\begin{claim*}The function \(\varphi_{n}\) extends across \(D\) to a holomorphic function \(\varphi : \overline{X} \to \mathbb{C}^{n}\).
\end{claim*}

\begin{proof} To prove this, it suffices to show for example that \(\varphi_{n}\) extends across \(D_{1}\). Writing that the expression \eqref{eq:multival2} must land inside \(\mathbb{S}_{n}\), we get that
\begin{equation} \label{eq:condpos}
	\Im (\varphi_{n}(z) + \sum_{q=1}^{n} \tau_{q} w_{q}) > 0
\end{equation}
for all \(w \in \widetilde{X}\). Fixing all \(w_{q}\) but the first one, and using the fact that \(\Im(w_{1}) = - \frac{1}{2\pi} \log |z_{1}|\), we obtain that there exists \(C > 0\) such that
\[
	\Im (\varphi_{n}(z)) \geq \tau_{1} \frac{1}{2\pi} \log |z_{1}| + C.
\]
for all such \(w\). Applying Lemma~\ref{lem:riemannboundlog} with the variable \(z = z_{1}\), we get that \(\varphi_{n}\) has no pole along \(D_{1}\). We can do the same for all \(D_{j}\) and get the result. 
\end{proof}

\begin{claim*} All \(\varphi_{j}\) extend across \(D\) as holomorphic functions for \(j \leq n - 1\).
\end{claim*}

\begin{proof}
	This time, writing that \eqref{eq:multival2} lands inside \(\mathbb{S}_{n} = \{(z', z_{n}) \in \mathbb{C}^{n-1} \times \mathbb{C} \; | \; || z' ||^{2} < \mathrm{Im}(z_{n})\}\), one finds
\[
	\prod_{q=1}^{n} |z_{q}|^{2 \alpha_{j, q}}\; |\varphi_{j}(z)|^{2} = | e^{2i \pi \sum_{q} \alpha_{j, q} w_{q}} \varphi_{j} (z)|^{2} \leq \Im(\varphi_{n}(z) + \tau_{q} \sum_{q} w_{q}) 
\]
for all \(j \in \llbracket 1, n \rrbracket\). Let us fix all the \(z_{q}\) for \(q \neq 1\), and let \(z_{1} \to 0\).
\medskip

By the previous claim, one has \(\Im(\varphi_{n}(z)) = O(1)\), so the right hand side grows at most as \(O(|\log |z_{1}| |)\). Since \(\alpha_{p, q} \in [0, 1)\) for all \(p, q\), \(\varphi_{j}\) does not have a pole along \(D_{1}\).
\medskip

Since this also holds for \(D_{2}, \dotsc, D_{n}\), one gets the result.
\end{proof}

The proof of Proposition~\ref{prop:asymptotic} is now complete.
\medskip

Remark that the condition \eqref{eq:condpos} implies that we must have \(\tau_{q} \geq 0\) for all \(q \in \llbracket 1, k\rrbracket\). At least one of them must be positive in order to have \(\psi(q_{m}) \longrightarrow b_{\infty}\) as \(m \to +\infty\).

\begin{lem} \label{lem:complete}
	Let \(h_{X}\) be the KE metric on \(X\) descended from the pullback metric \(\psi^{\ast} h_{\mathbb{B}^{n}}\). Then
	\begin{enumerate}[(i)]
		\item the metric \(h_{X}\) is complete if and only if we are in the situation of Proposition~\ref{prop:asymptotic} (2) with all the \(\tau_{q} > 0\).
		\item the metric \(h_{X}\) has finite volume.
	\end{enumerate}
\end{lem}

\begin{proof}
	
	{\em (i)} 
	It is easy to see that if we are in the situation of Proposition~\ref{prop:asymptotic} (1) or Proposition~\ref{prop:asymptotic} (2) with at least one \(\tau_{q} = 0\), then the metric cannot be complete (either \(0\) or the component \(D_{q}\) with \(\tau_{q}=0\) are at finite distance). Let us  then focus on the second case, and assume all \(\tau_{q}\) are positive.

We will check below that \eqref{eq:edgeball} implies that the metric \(\psi^{\ast} h_{\mathbb{B}^{n}}\) is bounded from below by a Poincaré form depending on the factor \((\Delta^{\ast})^{k}\):
	\begin{equation} \label{eq:Poincarelowerbound}
		\psi^{\ast} \omega_{\mathrm{Berg}}
		\geq 
			C \sum_{j=1}^{k} \frac{i d z_{j} \wedge d \overline{z}_{j}}{|z_{j}|^{2} \log^{2}|z_{j}|},
	\end{equation}

	This inequality that any smooth path joining a point of \(U_{\alpha}\) to a point of the boundary is of infinite length, which gives the result.
	\medskip

	Let us verify that \(\psi^{\ast} \omega_{\mathrm{Berg}}\) indeed has the required growth. Since the element \(B\) leaves the Bergman metric invariant, we see that we may assume that \(\psi : \widetilde{U}_{\alpha} \to \mathbb{S}_{n}\) corresponds to a multivaluate function on \(X\) given by
	\begin{equation} \label{eq:multivaluate}
	\phi(z)
	=
	(\widetilde{\phi}_{1}(z),
	\dotsc
	\widetilde{\phi}_{n-1}(z),
		\phi_{n}(z) + \sum_{q \leq k} \frac{1}{2i\pi} \tau_{q} \log(z_{q}))
	\end{equation}
	where \(\widetilde{\phi}_{j}(z) = \prod_{q \leq k} z^{\alpha_{j, q}} \phi_{j}(z)\), and all \(\phi_{j}\) are bounded holomorphic single valued functions. By \eqref{eq:exprBergman}, one has
	\[
		\psi^{\ast} \omega_{\mathrm{Berg}}
		\geq 
		- i \frac{\phi^{\ast} (\partial l) \wedge \phi^{\ast}(\overline{\partial} l)}{l^{2}}
	\]
	Then, the lower bound \eqref{eq:Poincarelowerbound} comes from \eqref{eq:exprBergman} and the following estimates, that follows easily from the expression of the function \(l\) :
\begin{align*}
	l \circ \phi(z) & = - \sum_{q \leq k} \frac{\tau_{q}}{2\pi} \log |z_{q}| + O(1) \\
	\phi^{\ast}(\partial l)	& = - \frac{1}{8\pi} \sum_{q \leq k} \tau_{q} \frac{dz_{q}}{z_{q}} + \sum_{q \leq k} O(|z_{q}|^{\epsilon - 1}) \\
\end{align*}
	for some \(\epsilon > 0\). In each of these expressions, \(O(f)\) denotes a function or \(1\)-form on \(U_{\alpha}\) with possibly multivaluate coefficients, with absolute value bounded from above by a multiple of \(f\). To obtain the second estimate, note that \(\partial_{j} (z_{j}^{\alpha_{j, q}}) = \alpha_{j, q} z^{\alpha_{j, q}-1}\) if \(\alpha_{j, q} > 0\).

	\medskip

	{\em (ii)} Let us prove the result in the case of Proposition~\ref{prop:asymptotic} {\em (2)}, with all the \(\tau_{q} > 0\) for \(1 \leq q \leq k\). Then, with the notation of Section~\ref{sec:siegelmodel}, the volume form on \(\mathbb{S}_{n}\) is
	\[
		\frac{\omega_{\mathbb{S}_{n}}^{n}}{n!}
		=
		\frac{1}{l^{n+1}} 
		i dy'_{1} \wedge d\overline{y}'_{1}
		\wedge \dotsc \wedge
		i dy'_{n-1} \wedge d\overline{y}'_{n-1}
		\wedge
		\frac{i}{4} \left(d y_{n} + \sum_{1 \leq j\leq n-1} \overline{y}_{j}' dy_{j}' \right) \wedge \left( d \overline{y}_{n} + \sum_{1 \leq j\leq n-1} y_{j}' d\overline{y}_{j}'\right).
	\]
	Thus, pulling back this form by the multivaluate expression \eqref{eq:multivalcase2} yields an upper bound
	\[
		\frac{\omega_{X}^{n}}{n!}
		\leq
		C \frac{1}{[- \sum_{q \leq k} \log |z_{q}|]^{n+1}}
		\frac{\bigwedge_{1 \leq j \leq k} i dz_{j} \wedge d\overline{z}_{j}}{\prod_{q \leq k}|z_{q}|^{2}}
	\]
	as \(z = (z_{1}, \dotsc, z_{n})\) goes to the boundary. The expression above is integrable on \((\Delta^{\ast})^{k} \times \Delta^{n-k}\), which gives the result. The other cases can be proved with a similar computation.
\end{proof}

The following fact must be well-known to experts, but we prefer to recall a proof for completeness. In Remark~\ref{rem:notballquotient}, we used this proposition to show that the open part of the variety \(X^{\ast}\) is not a ball quotient.

\begin{prop} \label{prop:notballquotient}
	Let \(\overline{X}\) be a smooth projective variety, endowed with a SNC divisor \(D\), and let \(X := \overline{X} - D\). Assume that \(X = \quotient{\Gamma}{\mathbb{B}^{n}}\) is a ball quotient by a subgroup \(\Gamma \subset \mathrm{Aut}(\mathbb{B}^{n})\) that acts freely and properly discontinuously. Then \(\Gamma\) is a lattice i.e. it has finite covolume. At least one component of \(D\) must be birational to a quotient of an abelian variety by a finite group. 
\end{prop}
\begin{proof}
	Let \(h_{X}\) be the metric descended on \(X\) from \(h_{\mathbb{B}^{n}}\) on \(\mathbb{B}^{n}\). By Lemma~\ref{lem:complete} applied on a set of polydisks covering \(D\), we deduce that \(h_{X}\) has finite volume. Hence \(\Gamma\) is a lattice. Let \(\Gamma' \lhd \Gamma\) be a normal subgroup of finite index such that \(\Gamma'\) has only unipotent parabolic isometries, and let \(G := \quotient{\Gamma'}{\Gamma}\). Let \(X' := \quotient{\Gamma'}{\mathbb{B}^{n}}\), and let \(X' \to X\) be the associated covering with Galois group \(G\). Since \(\Gamma'\) has unipotent parabolic isometries, the variety \(X'\) admits a minimal smooth compactification \(Y'\) with boundary made of a disjoint union of abelian varieties.
	
	Denote by \(p : \overline{X}' \to \overline{X}\) the normalization of \(\overline{X}\) in the function field of \(X'\). The variety \(\overline{X}'\) has at most klt singularities by \cite[Corollary~5.20]{KM98}, and thus one may apply \cite[Lemma A.4]{Deng22} to deduce the existence of a birational morphism \(f : \overline{X}' \to Y'\) sitting in a diagram
	\[
		\begin{tikzcd}
			\overline{X}' \arrow[r, "f"] \arrow[d, "p"] & Y' \arrow[d] \\
			\overline{X}  \arrow[r, "g"] & \quotient{G}{Y'}
		\end{tikzcd}
	\]
	In this diagram, the existence of \(g\) comes by taking the quotient of \(f\) by \(G\), using the fact that \(\overline{X} = \quotient{G}{\overline{X}'}\).  Then, since \(f\) is birational, so is \(g\); this implies that for any boundary component of \(\quotient{G}{Y'}\), there is a component of \(D = \overline{X} - X\) that is birational to it. This gives the result.
\end{proof}

\subsubsection{Local limit of the period map across the boundary.} \label{subsub:localext}

Passing to the limit \(\Im(w_{j}) \longrightarrow + \infty\) (\( j \in J\)), we see that the expressions \eqref{eq:multival1} and \eqref{eq:multival2} induces a holomorphic map on \(\{0, \dotsc, 0 \} \times \prod_{j > k} \Delta\), with values in either in \(\mathbb{B}^{n}\) or in \(\mathbb{C}^{n-1}\).
\medskip

Let us introduce a definition to encapsulate the behaviour of this limiting map; in the next section, this definition will become the local model of the global extension of the period map across the boundary. 

\begin{defi} \label{def:locallimitperiod} Let \(\psi_{0}\) be a holomorphic map on \(\{0\}^{k} \times \Delta^{n-k}\), with target space to be decided below. We say that \(\psi_{0}\) is {\em a limiting map for \(\psi\), with limit point \(b_{\infty} \in \overline{\mathbb{B}^{n}}\)} if one of the following occurs.
	\begin{enumerate}
		\item One has \(b_{\infty} \in \mathbb{B}^{n}\). Then \(\psi_{0}\) takes its values in \(\mathbb{B}^{n}\). There exist \(B \in \mathrm{PU}(n, 1)\) such that \(B \cdot b_{\infty} = o \in \mathbb{B}^{n}\), real numbers \((\alpha_{p, q})_{1 \leq p \leq n, 1 \leq q \leq k}\) in \([0, 1)\) and a holomorphic map \(\varphi : \overline{X} \to \mathbb{C}^{n}\) such that for any \(w = (w_{1}, \dotsc, w_{k}, z_{k+1}, \dotsc, z_{n}) \in \widetilde{X}\), one has
			\begin{equation} \label{eq:intball}
				B \cdot \psi(w) = \left(
				e^{2i\pi \sum_{q \leq k} \alpha_{1, q} w_{q}} \varphi_{1}(z),
				\dotsc,
				e^{2i\pi \sum_{q\leq k } \alpha_{n, q} w_{q}} \varphi_{n}(z)
				\right),
			\end{equation}
			and for any point \(z = (0, \dotsc, 0, z_{k+1}, \dotsc, z_{n})\), one has
			\begin{equation} \label{eq:limitintball}
				B \cdot \psi_{0}(z) = \left(
				\delta_{1}\,	
				\varphi_{1}(z),
				\dotsc,
				\delta_{n}\,
				\varphi_{n}(z)
				\right) 	
			\end{equation}
			where \(z_{q} = e^{2i\pi w_{q}}\) for all \(q\), and \(\delta_{j} \in \{0, 1\}\) equals \(1\) if and only if all \(\alpha_{j, q}\) are zero for \(1 \leq q \leq k\).
			\medskip

		\item One has \(b_{\infty} \in \partial \mathbb{B}^{n}\). Then \(\psi_{0}\) takes its values in \(\mathbb{C}^{n-1}\). 
			
			If we let \(\phi := \phi_{b_{\infty}} : \mathbb{B}^{n} \longrightarrow \mathbb{S}_{n}\) be the Siegel presentation of the ball with respect to \(b_{\infty}\), there exists \(B \in \mathrm{PU}(n, 1)\) fixing \(b_{\infty}\), real numbers \((\alpha_{p, q})_{1 \leq p \leq n-1, 1 \leq q \leq k}\) in \([0, 1)\), real non negative numbers \((\tau_{q})_{1 \leq q \leq k}\), not all zero, and a holomorphic map \(\varphi : \overline{X} \to \mathbb{C}^{n}\) such that for any \(w = (w_{1}, \dotsc, w_{k}, z_{k+1}, \dotsc, z_{n}) \in \widetilde{X}\), one has
			\begin{equation} \label{eq:edgeball}
				B \cdot \psi(w)
				=
				\phi^{-1}
				\big(
				e^{2i\pi \sum_{q \leq k} \alpha_{1, q} w_{q}} \varphi_{1}(z),
				\dotsc,
				e^{2i\pi \sum_{q\leq k} \alpha_{n-1, q} w_{q}} \varphi_{n-1}(z),
				\varphi_{n}(z) + \sum_{q\leq k} \tau_{q} w_{q}
				\big)
			\end{equation}
			and for any \(z = (0, \dotsc, 0, z_{k+1}, \dotsc, z_{n})\), one has
				\begin{equation} \label{eq:limitedgeball}
				B \cdot \psi_{0}(z)
				=
				\big(
				\delta_{1}\,
				\varphi_{1}(z),
				\dotsc,
				\delta_{n-1}
				\varphi_{n-1}(z)
				\big)
			\end{equation}
			where \(z_{q} = e^{2 i \pi w_{q}}\) for all \(q\). In this last equation, the action of \(B\) on \(\mathbb{C}^{n-1}\) is induced by the quotient map \(W_{b} \to \mathrm{Sim}(\mathbb{C}^{n-1})\), and again \(\delta_{j} \in \{0, 1\}\) equals \(1\) if and only if all \(\alpha_{j, q}\) are zero \((1 \leq q \leq k)\).
	\end{enumerate}
\end{defi}

\begin{rem}
	In case (1), unless all \(A_{j}\) are trivial for all \(j \in \llbracket 1, k\rrbracket\), there is at least one \(\delta_{j}\) which is zero. In this case, the map \(\psi_{0}\) actually takes its values in a smaller dimensional ball given by the intersection of \(\mathbb{B}^{n}\) with an affine subspace.
\end{rem}

Proposition~\ref{prop:asymptotic} now allows to construct the limiting map from \(\varphi\). Inspecting \eqref{eq:multival1} and \eqref{eq:multival2} permits indeed to obtain it simply by taking the limit of the corresponding multivaluate map, as expressed in the next proposition.

\begin{prop} \label{prop:limitingmapgeom} With the notation of Section~\ref{subsub:notation}, there are two possible cases:
	\begin{enumerate}
		\item the limit \(b_{\infty} = \lim_{\pi(w) \to 0} \psi(w)\) exists in \(\mathbb{B}^{n}\). In this case, for all \(z \in \{0\}^{k} \times \Delta^{n-k}\), the following limit
			\[
				\psi_{0}(z) := \lim_{\pi(w) \to z} \psi(w)
			\]
			is also well-defined as a point in \(\mathbb{B}^{n}\), and the map \(\psi_{0}\) is a limiting map for \(\psi\). Unless the monodromy is trivial, the map \(\psi_{0}\) factors through a totally geodesically embedded ball
			\[
				\bigcap_{1 \leq j \leq k} \mathrm{Stab}(A_{j}) \subsetneq \mathbb{B}^{n}.
			\]
		\item the limit \(b_{\infty} = \lim_{\pi(w) \to 0} \psi(w)\) exists in \(\partial \mathbb{B}^{n}\). Let \(\phi := \phi_{b_{\infty}} = \mathbb{B}^{n} \to \mathbb{S}_{n}\) be the Siegel model at \(b_{\infty}\), and consider the composition \(\pi_{n-1} := \mathrm{proj}_{\mathbb{C}^{n-1}} \circ \phi\), where \(\mathrm{proj}_{\mathbb{C}^{n-1}}\) is the projection on the first coordinates. Then, for all \(z \in \{0\}^{k} \times \Delta^{n-k}\), the following limit
			\[
				\psi_{0}(z) := \lim_{\pi(w) \to z} \pi_{n-1}(\psi(w))
			\]
			is well-defined as a point in \(\mathbb{C}^{n-1}\), and the map \(\psi_{0}\) is a limiting map for \(\psi\). 
	\end{enumerate}
\end{prop}

\subsection{The global case} \label{sec:globsec} We will now consider the global situation, and describe the structure of period maps induced on the boundary of a complex open variety. Let us introduce the following geometric data.  

\begin{enumerate}
	\item Let \(\overline{X}\) be a complex manifold of dimension \(n\), and let \(D\) be a simple normal crossing divisor on \(\overline{X}\). We let \(X = \overline{X} - X\), and pick a base point \(o \in X\). For each \(j \in \llbracket 0, n\rrbracket\), we let \(D_{k} \subset D\) the smooth \(k\)-codimensional locally closed stratum.
	\item Let \(\rho : \pi_{1}(X, o) \to \PU(n, 1)\) be a representation, and consider a \(\rho\)-equivariant map \(\psi : \widetilde{X} \to \mathbb{B}^{n}\).
	\item Let \(k \in \llbracket 1, n\rrbracket\) and \(Y \subset D_{k}\) be a connected component. We fix \(Y' \to Y\), a connected component of the covering over \(Y\) induced by \(\widetilde{X} \to X\), in the sense of Section~\ref{sec:inducedcovering}.
\end{enumerate}

We are going to construct a representation \(\sigma\) of \(\pi_{1}(Y)\) and a \(\sigma\)-equivariant map \(\psi_{0}\) on \(\widetilde{Y}\), with values in \(\mathbb{B}^{n}\) or \(\mathbb{C}^{n-1}\), so that this data is locally compatible with \(\psi\) in the sense of Definition~\ref{def:locallimitperiod}. We start by constructing the map.

\subsubsection{Construction of the map.} To properly describe the local model of the limiting map in our global situation, we need to introduce a few more notation.

\begin{nota} \label{nota:local} Let \(y \in Y\) be any point, and let \(U \cong (\Delta^{\ast})^{k} \times \Delta^{n-k} \subset X\) be a pointed polydisk centered at \(y\). Consider the fiber product 
	\[
		\begin{tikzcd}
			U \times_{X} \widetilde{X} \arrow[r] \arrow[d] & \widetilde{X} \arrow[d] \\
			U          \arrow[r]                           & X
		\end{tikzcd}
	\]
	The manifold \(U \times_{X} \widetilde{X}\) is a disjoint union of copies of open manifolds of the form \((\Delta^{\ast})^{l} \times \mathbb{H}^{k-l} \times \Delta^{n-k}\). If \(V\) is one of these components, its universal covering is \(\widetilde{V} \cong \mathbb{H}^{k} \times \Delta^{n-k}\), and we will denote by 
	\(
	\psi_{V} : \mathbb{H}^{k} \times \Delta^{n-k} \to \mathbb{B}^{n}
	\)
	the composition of the natural maps 
	\[
		\widetilde{V}  \to U \times_{X} \widetilde{X} \to \widetilde{X} \to \mathbb{B}^{n}.
		\]
\end{nota}

\begin{prop} \label{prop:globalmap}
	In the following, we will pick \(y' \in Y'\) with projection \(y \in Y\). We then introduce \(V \overset{\pi}{\to} U\) as above with \(U\) is centered at \(y\), where \(V\) is the connected component of \(U \times_{X} \widetilde{X}\) neighboring \(y'\) in the sense of Section~\ref{sec:inducedcovering}.

	 With this notation, there are two possibilities.
	\begin{enumerate}
		\item For all \(y' \in Y'\) as above, the limit \(\psi_{0}(y') := \lim_{\pi(z) \to y} \psi_{V}(z)\) exists in \(\mathbb{B}^{n}\). This limits depends only on \(y'\), but not on the choice of \(U\).

			The map \(\psi_{0} : Y' \to \mathbb{B}^{n}\) is holomorphic. Unless the monodromy around all component of \(D\) containing \(Y\) is trivial, the map \(\psi_{0}\) factors through a totally geodesically embedded ball \(\mathbb{B}^{p} \hookrightarrow \mathbb{B}^{n}\) \((p < n)\).
			\medskip

		\item There exists \(b_{\infty} \in \partial \mathbb{B}^{n-1}\) such that for all \(y' \in Y'\) as above, one has \(\lim_{\pi(z) \to y} \psi_{V}(z) = b_{\infty}\). In this case, for all such \(y'\in Y\), the limit
			\[
				\psi_{0}(y') := \lim_{\pi(z) \to y} \mathrm{proj}_{\mathbb{C}^{n-1}} \circ \phi_{b_{\infty}} \circ \psi(z)
				\]
				exists in \(\mathbb{C}^{n-1}\).

				The map \(\psi_{0} : Y' \to \mathbb{C}^{n-1}\) is holomorphic.
	\end{enumerate}
\end{prop}

To prove this proposition, the only thing left to check is there cannot be two points of \(Y'\) that do no satisfy the same item in Proposition~\ref{prop:limitingmapgeom}. This simply comes from the connectedness of \(Y'\), and the fact that each of these two situations is satisfied on an open subset of \(Y'\).

\subsubsection{Construction of the induced representation} We are now going to construct a representation of \(\pi_{1}(Y)\) (or even more precisely, of \(\quotientd{\pi_{1}(Y)}{\pi_{1}(Y')}\)), under which the map \(\psi_{0}\) constructed above is equivariant. To define the image of a loop \(\gamma\) inside \(Y\), the idea will simply be to move it a little bit to get a loop \(\mu\) inside \(X\), and then to check that the image \(\rho(\mu) \in \PU(n,1)\) has a well-defined action on the image of \(\psi_{0}\), that does not depend on the choice of the moved loop.
\bigskip

We fix a base point \(b_{0} \in Y\), and consider a pointed polydisk \(U = (\Delta^{\ast})^{k} \times \Delta^{n-k}\) adapted to \(D\) and centered at \(b_{0}\). We let \(W \hookrightarrow \overline{X}\) be a euclidean tubular neighborhood of \(Y\) with \(\mathcal{C}^{\infty}\) projection map \(q : W \to Y\). We pick a base point \(b \in W \cap X\) so that \(q(b) = b_{0}\). Let \(\gamma_{1}, \dotsc, \gamma_{k} \in \pi_{1}(X, b)\) be the classes of meridian loops around the components of \(D\), and let \(A_{j} = \rho(\gamma_{j}) \in \PU(n, 1)\).
\medskip

Consider now a class \([\gamma] \in \pi_{1}(Y, b_{0})\), for which we want to construct the image \(\sigma([\gamma])\). The lift of \(\gamma\) to \(Y'\) connects two points \(y_{1}', y_{2}'\) in the fiber of \(Y'\to Y\) above \(b_{0}\).  Let \(V_{1}, V_{2} \subset U \times_{X} \widetilde{X}\) be the connected components neighboring \(y_{1}', y_{2}'\) (see again Section~\ref{sec:inducedcovering}).
\medskip

We choose a loop \(\mu\) in \(X\), based at \(b\), that is also a section of the restriction to \(\gamma\) of the projection \(q : W \cap X \to Y\). Note that the class \([\mu] \in \pi_{1}(X, b)\) depends on the choice of \(\mu\), but only up to a product of some of the \(\gamma_{j}\). Then the lift of \(\mu\) to \(\widetilde{X}\) links a point of \(V_{1}\) to a point in \(V_{2}\), so since \(\psi\) is \(\rho\)-equivariant, one has for all \(y \in V_{2}\):
\begin{equation} \label{eq:V1toV2}
	\psi|_{V_{2}}(\mu \cdot y) = \rho([\mu]) \cdot \psi|_{V_{1}}(y).
\end{equation}

We will now construct the image of \([\gamma]\) by passing to the limit in the previous equation. Again, one has two possibilities.
\medskip

\noindent
{\em Case 1 of Proposition~\ref{prop:globalmap}.} Denote by \(\psi_{0}^{1} : U_{1} \to \mathbb{B}^{n}\) and \(\psi_{0}^{2} : U_{2} \to \mathbb{B}^{n}\) be the limiting maps for \(\psi|_{V_{1}}\) and \(\psi|_{V_{2}}\) (where \(U_{1}, U_{2} \subset Y'\) are \(n-k\)-dimensional polydisks). In this case, the two limiting maps factor by a ball \(\mathbb{B}^{p} \subset \mathbb{B}^{n}\) that is fixed by all \(A_{j}\). We may assume that \(p\) is the smallest possible dimension.

Letting \(y\) tend to a point of \(\{0\}^{k} \times \Delta^{n-k}\) in \eqref{eq:V1toV2} shows that 
\[
	\psi_{0}^{2}(\gamma \cdot y') = \rho([\mu]) \cdot \psi_{0}^{1}(y')
\]	
for all \(y' \in U_{1}\). By analytic continuation, one sees that the element \(\rho([\mu])\) is an isometry of \(\mathbb{B}^{n}\) that preserves the image of \(\psi_{0} : Y' \to \mathbb{B}^{n}\); it must then also preserve the smallest dimensional ball \(\mathbb{B}^{p}\) that contains it. 

Thus, \(\rho([\mu])\) induces an element \(\sigma([\gamma]) \in \mathrm{Aut}(\mathbb{B}^{p})\). This element does not depend on the choice of \(\mu\), since as we said earlier, two different choices differ by a product of \(A_{j}\), all of them fixing \(\mathbb{B}^{p}\).
\medskip

It is then straightforward to check that we have obtained a morphism of groups
\[
	\sigma :\pi_{1}(Y, b_{0}) \to \mathrm{Aut}(\mathbb{B}^{p}).
\]
The map \(\psi_{0} : Y' \to \mathbb{B}^{p}\) is equivariant by construction.
\bigskip

\noindent
{\em Case 2 of Proposition~\ref{prop:globalmap}.} The discussion is completely parallel to the first case. This time, one gets a representation
\[
	\sigma: \pi_{1}(Y, b_{0}) \to \mathrm{Sim}(\mathbb{C}^{n-1})
\]
with respect to which the map \(\psi_{0} : Y' \to \mathbb{C}^{n-1}\) is equivariant.

\section{Uniformizing maps in presence of isolated pure log-canonical singularities} \label{sec:uniformization}

In the rest of this section, we will fix the following data:

\begin{enumerate}
	\item \(X^{\ast}\) is a normal projective variety with isolated log-canonical singularities. Denote by \(X\) its smooth locus, and let \(\sigma :\overline{X} \to X^{\ast}\) be a log-resolution of singularities. {\em We assume that all the exceptional divisors in this resolution have discrepancy equal to \(-1\)}.
	\item \(\rho : \pi_{1}(X) \to \PU(n, 1)\) is a representation ;
	\item \(\psi : \widetilde{X} \to \mathbb{B}^{n}\) is an étale holomorphic map which is \(\rho\)-equivariant.
\end{enumerate}

Under these conditions, the metric \(\psi^{\ast} h_{\mathbb{B}^{n}}\) on \(\widetilde{X}\) is \(\pi_{1}(X)\)-invariant, and descends to \(X\) to define a metric that we will denote by \(h_{X}\).

We will prove the following result.

\begin{thm} \label{thm:completeness}
	Under the assumptions made at the beginning of the section, the following two claims hold.
	\begin{enumerate}
		\item the pullback metric \(\psi^{\ast} h_{\mathbb{B}^{n}}\) is complete on \(\widetilde{X}\) ;
		\item the map \(\psi\) is a biholomorphism.
	\end{enumerate}
\end{thm}

Let us first remark that the first point implies the second one, by the classical Lemma~\ref{lem:criterioncomplete}, applied to \(Y = \widetilde{X}\) and \(Z = \mathbb{B}^{n}\). The main problem will then be to show the completeness of the metric \(\psi^{\ast} h_{\mathbb{B}^{n}}\), or equivalently of \(h_{X}\).
\medskip

The core of the proof will be based on the following dichotomy, whose second case will eventually lead to a contradiction, later allowing us to prove that the pull-back metric \(h_{X}\) is complete on \(X\). This result claims essentially that in our situation, the limit of the period map sends any given connected component of the exceptional divisor to a single point, which is either in \(\mathbb{B}^{n}\) or in its boundary.

\begin{prop} \label{prop:limitpoint} Let \(q \in X_{\mathrm{sing}}\) be a singular point, and let \(E \subset \overline{X}\) be the exceptional divisor over \(q\). 	Let \(\Omega \subset X^{\ast}\) be a closed neighborhood of \(q\) such that \(\sigma^{-1}( \Omega - \{ q \}) \subset \overline{X} \) is contained in a finite union of pointed polydisks centered around points of \(E\). Let \(\Omega_{0} = \Omega - \{q\}\), and fix a connected component \(\widetilde{\Omega}_{0}\) of \(\pi^{-1}(\Omega_{0}) \subset \widetilde{X}\).

	For each smooth strata \(Y\) of \(E\), we will denote by \(Y' \to Y\) the covering induced by \(\widetilde{\Omega}_{0} \to \Omega_{0}\).
	
	Then one of the following two cases holds.
	\begin{enumerate}
		\item there exists \(b_{\infty} \in \partial \mathbb{B}^{n}\) such that for any smooth strata \(Y\) of \(E\), the map \(\psi\) has a limit along \(Y'\) as in Proposition~\ref{prop:globalmap}~(2). In this case, for any closed neighborhood \(\Omega \ni q\) small enough, the pullback metric \(h_{X}\) is complete on \(\Omega - \{p\}\).
		\item there exists \(b_{\infty} \in \mathbb{B}^{n}\) such that for any smooth strata \(Y\) of \(E\), the map \(\psi\) has a limit along \(Y'\) as in Proposition~\ref{prop:globalmap}~(1). The induced map on \(Y'\) is {\em constant, equal to} \(b_{\infty}\).
	\end{enumerate}
\end{prop}

As we explained above, we will later show that the second case cannot happen.
\medskip

Let us prove Proposition~\ref{prop:limitpoint}. There are two possibilities to distinguish.
\medskip

\noindent
{\em Case 1. All smooth strata \(Y\) of \(E\) are in the situation of Proposition~\ref{prop:globalmap}~(2) i.e.\ along all components \(Y'\), the map \(\psi\) has limiting map towards some \(b_{Y'} \in \partial \mathbb{B}^{n}\).}

\begin{lem} \label{lem:samebY}
	All \(b_{Y'}\) are equal, i.e. there exists \(b_{\infty} \in \partial \mathbb{B}^{n}\) such that \(b_{Y'} = b_{\infty}\) for all \(Y\) and \(Y'\) as above. In the local description of Proposition~\ref{prop:asymptotic} (2), all the \(\tau_{j}\) are non-zero.
\end{lem}
\begin{proof}
	Pick two smooth strata \(Y_{1}\), \(Y_{2}\) of \(E\), and let \(Y_{1}'\), \(Y_{2}'\) be two arbitrary components of their induced coverings. For \(j = 1, 2\), let \(U_{j} \subset \Omega_{0}\) be a pointed polydisk adapted to \(Y_{j}\), and let \(U_{j}' \subset \widetilde{\Omega}_{0}\) be the connected component of \(U_{j} \times_{\Omega_{0}} \widetilde{\Omega}_{0}\) neighboring \(Y_{j}'\). Then, since \(\widetilde{\Omega}_{0}\) is connected, one can find a path \(\lambda : [0, 1] \to \widetilde{\Omega}_{0}\) linking \(U_{1}'\) to \(U_{2}'\).
	\medskip

	Let \(U_{1} = V_{0}, V_{1}, \dotsc, V_{m} = U_{2}\) be a sequence of polydisks in \(\Omega_{0}\) adapted to \(E\) containing the image of \(\pi \circ \lambda\), and let \(V_{0}' = U_{1}', V_{1}', \dotsc, V_{m}' = U_{2}'\) be the connected components of the \(V_{j} \times_{X} \widetilde{\Omega}_{0}\) crossed by \(\lambda\). Then, on each pair \((V_{j}, V_{j}')\), the map \(\psi\) is described as in Definition~\ref{def:locallimitperiod}~{\em (ii)}. As \(E\) is connected, one may always shrink \(\Omega_{0}\) to assume that two consecutive \(V_{i}\) are in contact with a common stratum of \(E\). By continuity, one sees that the limiting point in \(\partial \mathbb{B}^{n}\) does not vary as one moves along \(\lambda\), and thus \(b_{Y_{1}'} = b_{Y_{2}'}\).
	\medskip

	Since the map has a limit in \(\partial \mathbb{B}^{n}\) for any component of \(E\), this implies that their corresponding translation length \(\tau_{j} \in \mathbb{R}_{+}\) must be positive.
\end{proof}

Since the covering \(\widetilde{\Omega}_{0}\) is connected, one sees from the local description of Section~\ref{subsub:localext} that the point \(b_{F}\) does not vary if one changes strata continuously in a given connected component of \(E\).
\medskip

Thus, we have proven that we are in the situation of Proposition~\ref{prop:limitpoint}~(1): this implies that the metric \(h_{X}\) is complete near \(p\):

\begin{lem}
	Assume that \(\Omega\) is such that \(\sigma^{-1}( \Omega_{0}) \subset \widetilde{X} \) is contained in a finite union of pointed polydisks on which \(\psi\) and its limiting map near \(E\) are of the forms \eqref{eq:edgeball} and \eqref{eq:limitedgeball}.
	Then the pull-back metric \(h_{X}\) is complete on \(\Omega - \{q\}\).
\end{lem}

Indeed, we can apply Lemma~\ref{lem:complete} to each polydisk above, using the positivity of all the translation lengths \(\tau_{j}\).

\bigskip

\noindent
{\em Case 2. There is at least one smooth stratum \(Y \subset E\) and a connected component \(Y'\) of the induced covering on \(Y\), that is in the situation of Proposition~\ref{prop:globalmap}~(1).} As explained in Section~\ref{sec:globsec}, we deduce the existence of a representation \(\sigma : \pi_{1}(Y) \to \mathrm{Aut}(\mathbb{B}^{p})\) \((p < n)\) and a \(\sigma\)-equivariant map \(\psi_{F} : Y' \to \mathbb{B}^{p} \subset \mathbb{B}^{n}\).
\medskip

On the other hand, the quasi-projective variety \(Y\) is special by Corollary~\ref{corol:strataspecial}, so the map \(\psi_{Y}\) is constant by Theorem~\ref{thm:isotrivial}. This proves that locally around each stratum of \(E\), the map \(\psi\) is either in the situation of Definition~\ref{def:locallimitperiod}~(1) (with a map \(\psi_{0}\) {\em constant} on each connected component of \(Y'\)), or in the situation of Definition~\ref{def:locallimitperiod}~(2).

With the exact same proof as in Lemma~\ref{lem:samebY}, one sees by connectedness of \(\widetilde{\Omega}_{0}\) and \(E\) that the only possibility is that the first situation is realized around any stratum \(Y'\) with the {\em same constant limiting map} \(\psi_{0} = b_{\infty} \in \mathbb{B}^{n}\).
\medskip

\begin{lem} In the situation of Case~2, one has the following.
	\begin{enumerate}[(i)]
		\item The metric completion of \(\Omega_{0}\) for the distance induced by \(h_{X}\) is obtained by adding only one point.
	\end{enumerate}
			By {\em (i)}, we may endow \(\Omega = \Omega_{0} \cup \{q\}\) with a complete metric induced by \(h_{X}\). Let \(\Gamma \subset \pi_{1}(X)\) be the subgroup generated by the meridian loops around the components of \(E\). Let \(r_{0} > 0\) be small enough so that \(M_{0} := B_{h_{X}}(q, r_{0}) - \{q\} \subset \Omega\), and let \(\widetilde{M}_{0} \subset \widetilde{\Omega}_{0}\) the inverse image of \(M_{0}\) by \(\pi\). Let \(N = B_{h_{\mathbb{B}^{n}}}(b_{\infty}, r_{0}) - \{b_{\infty}\}\).
	\begin{enumerate}[(i)]
		\setcounter{enumi}{1}
	\item Then \(\widetilde{M}_{0}\) is invariant under the action of \(\Gamma\), and the restriction \(\psi : \widetilde{M}_{0} \to N\) is a \(\Gamma\)-equivariant map that is \'{e}tale and isometric at any point. 
	\end{enumerate}
\end{lem}

\begin{proof}
\noindent
	{\em (i)} The description of \eqref{eq:intball} shows that 
	\[
		|| d \psi_{\ast} (v_{m}) ||_{h_{\mathbb{B}^{n}}} 
	\underset{m \longrightarrow + \infty}{\longrightarrow}
	0
	\]
	if \((v_{m})\) is a sequence of lifts of tangents vectors to \(X\), that tends on \(\overline{X}\) to a tangent vector to \(E\). Thus, one sees by connectedness of \(E\) that if \((a_{m}), (b_{m})\) are Cauchy sequences in \(\Omega_{0}\) that tend to different points of \(E\), that one has
	\[
		d_{h_{X}} (a_{m}, b_{m}) 
		\underset{m \longrightarrow + \infty}{\longrightarrow}
		0.
	\]
	This shows that the metric completion can be obtained by adding a single point.
	\medskip

	{\em (ii)} This point is clear, since the metric \(\psi^{\ast} h_{\mathbb{B}^{n}}\) is invariant under the action of \(\Gamma\). 
\end{proof}

\begin{lem}
	The two assumptions of Lemma~\ref{lem:universalcovering} are satisfied, with \(M = \widetilde{M}_{0}\), \(g = h_{X}\) and \(h = h_{\mathbb{B}^{n}}\). In particular, \(\widetilde{M}_{0} \cong N\).
\end{lem}

\begin{proof} For later reference, let us first remark that since for any \(r \in (0, r_{0})\), the set \(A_{r} := \Omega_{0} - B_{h_{X}}(q, r)\) is compact (recall that \(\Omega\) is closed), then one can find \(\epsilon_{r} > 0\) such that for any \(p \in A_{r}\), the injectivity radius of \(h_{X}\) at \(p\) is higher than \(\epsilon_{r}\).
	\medskip

Let us prove that the two assumptions hold. In the proof, the notation \(\psi\) will denote the restriction \(\psi|_{\widetilde{M}_{0}}\). We fix \(y\in N\), and let \(r_{1} := d_{h_{\mathbb{B}^{n}}}(b_{\infty}, y) \in (0, r_{0})\). Choose any \(r \in (0,r_{1})\).
	\medskip

	\noindent
	{\em (1)} Pick any \(x \in \psi^{-1}(y)\). Remark that \(d_{h_{X}}(q, x) \geq d_{h_{\mathbb{B}^{n}}}(b_{\infty}, y) > r\) since the differential of \(\psi\) being isometric at any point, the map \(\psi\) cannot increase distances. This proves that \(x \in A_{r}\), so the injectivity radius at \(x\) is higher that \(\epsilon_{r}\). We may decrease \(\epsilon_{r}\) a bit so that \(B(y, \epsilon_{r})\) injects in \(N\); note that this can be done independently of \(x\).

	We then have a commutative diagram
	\[
		\begin{tikzcd}
			& T^{\mathbb{R}}_{x}\widetilde{M_{0}} \cong T^{\mathbb{R}}_{X}\widetilde{N}_{0} & \\
			& B(0, \epsilon_{r})\arrow[dl, "\exp", "\sim" swap] \arrow[dr, "\exp", "\sim" swap] \arrow[hook, u] &  \\
			B(x, \epsilon_{r}) \arrow[rr, "\psi"]& & B(y, \epsilon_{r})
		\end{tikzcd}
	\]
	The exponential arrows are diffeomorphic since \(\epsilon_{r}\) is smaller than the injectivity radius at \(x\). This proves that the bottom line is a global isometry, and thus implies that the first assumption is satisfied, since \(\epsilon_{r}\) is independent of \(x\).
	\medskip

	\noindent
	{\em (2)} Let \(r' > r_{0}\) be small enough so that \(B_{h_{X}}(q, r_{0}) - \{q\} \subset \Omega_{0}\), and let \(K = \overline{B}(b_{\infty}, r') - B(b_{\infty}, r)\). It is a compact subset, clearly invariant under the action of \(\Gamma \subset \mathrm{Aut}(\mathbb{B}^{n})\), since \(\Gamma\) leaves \(b_{\infty}\) fixed.
	\medskip

	We may choose \(L \subset \psi^{-1}(K) = \overline{B}_{h_{X}}(q, r') - B_{h_{X}}(q, r)\) as follows. First, let \(U_{1}, \dotsc, U_{m} \subset \Omega_{0}\) be a family of pointed polydisks adapted to \(E \subset \overline{X}\), that completely cover \(E\), and which are {\em closed in \(\Omega_{0}\)}. For each polydisk \(U_{j}\), we may chose a closed subset \(F_{j} \subset \widetilde{\Omega}_{0}\), that is a fundamental domain for the action of the meridian loops of \(U_{j}\).

	Now, if we let \(L := (\bigcup_{j = 1}^{m} F_{j})) \cap \psi^{-1}(K)\), one sees that the orbit \(\Gamma \cdot L\) fully covers \(\psi^{-1}(K)\), which gives the result.
	\medskip

	We may now apply Lemma~\ref{lem:universalcovering}, since \(N \cong \mathbb{B}^{n} - \{0\}\) is simply connected (recall that \(n \geq 2\)).
	
\end{proof}

\begin{lem}
	The situation of Proposition~\ref{prop:limitpoint} (ii) cannot happen.
\end{lem} 

\begin{proof}
	Under the assumptions of Proposition~\ref{prop:limitpoint}~(ii), we have proven that \(\widetilde{M}_{0}\) identifies with \(\mathbb{B}^{n} -\{0\}\), on which \(\Gamma\) acts as a subset of \(\U(n)\). The manifold \(M_{0}\) is then the quotient \(\quotient{\Gamma}{\widetilde{M}_{0}}\). For \(M_{0}\) to be a separated topological space, the group \(\Gamma\) must be discrete, and thus finite since \(\U(n)\) is compact. But then, this shows that the analytical space \(B_{h_{X}}(q, r_{0}) = M_{0} \cup \{q\}\) can be seen as a quotient:
	\[
		B_{h_{X}}(q, r_{0}) = \quotient{\Gamma}{\mathbb{B}^{n}}.
	\]
	Hence \(q \in X^{\ast}\) is actually a quotient singularity, and thus is klt. We obtain a contradiction with the assumption made on the discrepancies of \(X^{\ast}\).
\end{proof}

All in all, we have proven that the situation of Proposition~\ref{prop:limitpoint} holds at any singular point \(q \in X^{\ast}\), which proves that the metric \(h_{X}\) is complete on the open part \(X\). This ends the proof of Theorem~\ref{thm:completeness}.

\section{Results from topology} \label{sec:topology}
\subsection{Criteria for universal coverings}

\begin{lem} \label{lem:universalcovering}
	Let \((M, g)\) and \((N, h)\) be two Riemannian manifolds. Let \(\psi : M \longrightarrow N\) be a differentiable map that is \'etale and isometric at any point of \(M\). Let \(\Gamma\) be a group that acts by isometries on \(M\) and \(N\), such that \(\psi\) is \(\Gamma\)-equivariant.

	We make the following two assumptions.
	\begin{enumerate}
		\item For any \(y \in N\), there exists \(\epsilon > 0\) such that for any \(x \in \psi^{-1}(y)\), the map
			\[
				\psi|_{B(x, \epsilon)} : B(x, \epsilon) \to B(y, \epsilon)
			\]
			is a diffeomorphism;
		\item for any \(y \in N\), there exists a \(\Gamma\)-invariant compact subset \(K \subset N\) containing \(y\), and a compact subset \(L \subset \psi^{-1}(N)\) such that
	\[
		\psi^{-1}(K) = \bigcup_{\gamma \in \Gamma} \gamma \cdot L.
	\]
	\end{enumerate}

	Then \(\psi\) is a covering. In particular, if \(N\) is simply connected, then \(\psi\) is a diffeomorphism.
\end{lem}

	The result will follow for the following more precise claim, that implies immediately that \(\psi\) is a covering.

	\begin{claim} \label{claim:disjointballs} For any \(y \in N\), there exists \(\epsilon > 0\) such that one has a {\em disjoint} union:
		\[
			\psi^{-1}(B(y, \epsilon)) = \bigsqcup_{x \in \psi^{-1}(y)} B(x, \epsilon).
		\]
	\end{claim}
	\begin{proof}[Proof of Claim~\ref{claim:disjointballs}] Suppose the contrary. Then we may find \(y \in N\) and two sequences \((a_{n}), (b_{n}) \in \psi^{-1}(y)^{\mathbb{N}}\) of pairwise distinct points such that
		\[
			\forall n, \quad
			B(a_{n}, 2^{-n}) \cap B(b_{n}, 2^{-n}) \neq \varnothing.
		\]
		In particular, one has \(\lim\limits_{n \longrightarrow +\infty} d_{g}(a_{n}, b_{n}) = 0\). Let now \(K\) and \(L\) be as in assumption \((2)\). By the defining property of \(L\), there exists a sequence \((\gamma_{n}) \in \Gamma^{\mathbb{N}}\) such that \(\gamma_{n} \cdot a_{n} \in L\) for all \(n \in \mathbb{N}\). One also has
		\begin{equation} \label{eq:images}
			\psi(\gamma_{n} \cdot a_{n})
			= \psi(\gamma_{n} \cdot b_{n}) 
			= \gamma_{n} \cdot y \in K
		\end{equation}
		for all \(n\), since \(K\) is \(\Gamma\)-invariant and \(\psi\) is \(\Gamma\)-equivariant. Thus, we may extract sequences to assume that \(\gamma_{n} \cdot a_{n} \longrightarrow x_{\infty} \in L\), and \(\gamma_{n} \cdot y \longrightarrow y_{\infty} \in K\). One has \(\psi(x_\infty) = y_{\infty}\) by continuity of \(\psi\). Also, since \(d(\gamma_{n} \cdot a_{n}, \gamma_{n} \cdot b_{n}) = d(a_{n}, b_{n}) \longrightarrow 0\), one has also \(\gamma_{n} \cdot a_{n} \longrightarrow x_{\infty}\). But now \eqref{eq:images} contradicts our assumption (1) that implies that \(\psi\) must be injective in a neighborhood of \(x_{\infty}\).
	\end{proof}
\medskip

\begin{lem} \label{lem:criterioncomplete}
	Let \(Y, Z\) be connected complex manifolds, and let \(\psi : Y \to Z\) be a holomorphic map, étale at any point. Endow \(Z\) with a hermitian metric \(h_{Z}\), and assume that \(Z\) is complete without self-crossing geodesic. If the pullback \(\psi^{\ast} h_{Z}\) is complete, then \(\psi\) is a biholomorphism.
\end{lem}
\begin{proof}
	\noindent
{\em \(\psi\) is surjective.} Choose a base-point \(b \in Y\), and let \(z \in Z\). Since \(Z\) is complete and connected, one can link \(\psi(b)\) to \(z\) by a geodesic \(\gamma\). Let \(l \in \mathbb{R}\) be its length and \(v \in T_{b} Y\) be a vector such that \(\psi_{\ast}(v)\) directs \(\gamma\) at \(z\). Since \(Y\) is complete, one can consider the endpoint \(y \in Y\) of the geodesic starting at \(b\), directed by \(v\) with length \(l\). Since \(\psi\) preserves geodesics, one has \(\psi(y) = z\).
\medskip

\noindent
	{\em \(\psi\) is injective.} Let \(y, y' \in Y\) be two different points such that \(\psi(y) = \psi(y')\). Let \(\gamma\) be a geodesic linking \(y\) to \(y'\). Then its image by \(\psi\) must be a non-trivial self-crossing geodesic, which does not exist by assumption. 
	
	This gives the result.
\end{proof}

\bibliographystyle{amsalpha}
\bibliography{biblio.bib}

\end{document}